\documentclass[a4paper]{amsart}
\usepackage{graphicx}
\usepackage{etex}
\usepackage{amsthm,stmaryrd}
\usepackage{amssymb}
\usepackage{array}
\usepackage{tabu}
\usepackage{epsfig}
\usepackage[usenames,dvipsnames]{color}
\usepackage{verbatim}
\usepackage{hyperref}
\usepackage{mathrsfs}
\usepackage[all]{xy}
\usepackage{xcolor}
\usepackage{enumerate}
\usepackage{tikz}   
\usepackage{relsize}
\usetikzlibrary{arrows,calc,positioning,decorations.pathreplacing}
\usepackage{xcolor}
\usepackage{pst-node}
\usepackage{tikz-cd}

\newcommand{\End}{\operatorname{End}}
\newcommand{\Hom}{\operatorname{Hom}}

\newcommand{\Z}{\mathbb{Z}}
\newcommand{\Q}{\mathbb{Q}}

\newcommand{\C}{\mathbb{C}}
\newcommand{\N}{\mathbb{N}}

\newcommand{\cat}{\mathscr{C}}
\newcommand{\F}{\mathbb{F}}
\newcommand{\D}{\mathbb{D}}

\newcommand{\HH}{{\ensuremath{\mathcal{H}}}}

\newcommand{\unit}{\ensuremath{\mathbb{I}}}

\newcommand{\tcoev}{\stackrel{\longleftarrow}{\operatorname{coev}}}

\newcommand{\tev}{\stackrel{\longleftarrow}{\operatorname{ev}}}
\newcommand{\ev}{\stackrel{\longrightarrow}{\operatorname{ev}}}

\newcommand{\qs}{q}

\newcommand{\coev}{\stackrel{\longrightarrow}{\operatorname{coev}}}

\newtheorem{definition}{Definition}[subsection]
\newtheorem{theorem}[definition]{Theorem}
\newtheorem*{theorem*}{Theorem}

\newtheorem{proposition}[definition]{Proposition}
\newtheorem{lemma}[definition]{Lemma}
\newtheorem{remark}[definition]{Remark}
\newtheorem*{remark*}{Remark}

\newtheorem{notation}[definition]{Notation}

\newtheorem{corollary}[definition]{Corollary}
\newcounter{exo} \newcounter{numexercice}
\renewcommand{\theexo}{\arabic{exo}}

\newcounter{IntroCounter}
\stepcounter{IntroCounter}

\begin{document}
\title[A topological model for the coloured Alexander invariants]{A topological model for the coloured Alexander invariants} 
  \author{Cristina Ana-Maria Anghel}
\address{Mathematical Institute, University of Oxford, Oxford, United Kingdom} \email{palmeranghel@maths.ox.ac.uk;cristina.anghel@imj-prg.fr} 
  \thanks{ }
  \date{\today}

\begin{abstract}
Coloured Alexander polynomials form a sequence of non-semisimple quantum invariants coming from the representation theory of the quantum group $U_q(sl(2))$ at roots of unity. This sequence recovers the original Alexander polynomial as the first term. We give a topological model for this invariants, showing that they can be obtained as graded intersection pairings between homology classes in a covering of the configuration space in the punctured disc. 
\end{abstract}

\maketitle
\setcounter{tocdepth}{1}
 \tableofcontents
 {
\section{Introduction} 
One of the first invariants discovered in knot theory was introduced by Alexander in 1923,  who defined a knot invariant as a polynomial in one variable. This invariant has many definitions, including one that uses skein theory, but also it can be obtained from the topology of the complement. Later on, the discovery of the Jones polynomial and coloured Jones polynomials due to V. Jones opened up a new area of low dimensional topology, called quantum topology. In 1991, Reshetikhin and Turaev introduced an algebraic and combinatorial recipe that provides link invariants having as input the representation theory of a quantum group.   

In this context, the representation theory of the quantum group $U_q(sl(2))$ with generic $q$, leads to the sequence of coloured Jones polynomials $\{J_N(L,q) \in \Z[q^{\pm}]\}_{N\in \mathbb N}$. This family has the original Jones polynomial as the first term. 

In a mirror situation, if one fixes $q=\xi$ to be a root of unity, then the representation theory of the quantum group at roots of unity $U_{\xi}(sl(2))$ leads to a sequence of invariants called coloured Alexander polynomials $\{ \Phi_{N}(L,\lambda)\in \Z[q^{\pm 1},q^{\pm \lambda}] \}_{N \in \N}$.
This sequence has as the first term the original Alexander polynomial.

In the original case, there is a fundamental question related to the two invariants for knots from above. The Alexander polynomial is well understood in terms of the complement of the knot. However, the connection between the Jones polynomial and the topology of the knot complement is still a deep and mysterious question. 

Another important direction is connected to the fact that these two families of quantum invariants evaluated at roots of unity recover the family of Kashaev invariants. On this line, another main question of quantum topology is the Volume Conjecture stated by Kashaev, which predicts that the limit of the Kashaev invariants recovers the hyperbolic volume of a knot. 
From this, it follows that the two sequences of quantum invariants are conjectured to recover at the limit certain topological information, which is the volume of the complement. 

As we have seen, the Reshetikhin-Turaev receipe that provides the sequences of coloured Jones and Alexander polynomials that recover these two polynomials at the first terms, is typically defined through algebraic and combinatorial tools. 
The question that we are interested in is to give topological models for these invariants. More precisely, to see them as graded intersection pairings between certain homology classes in coverings of configuration spaces.

On the topological part of the story, R. Lawrence \cite{Law} defined a family of braid group representations on the homology of coverings of configuration spaces in the punctured disc.
In \cite{Big}, \cite{Law1}, Lawrence and later Bigelow based on her work, showed that the Jones polynomial can be seen as an intersection pairing between homology classes in the Lawrence representation. They used the skein characterisation of the invariant for the proof.
 Later on, in \cite{Ito3}, Ito presented the loop expansion for the coloured Jones polynomials, as sums of traces of certain Lawrence representations and using this concluded a proof of the Melvin-Morton conjecture.  Then, in \cite{Ito2}, he showed a homological representation formula for the coloured Alexander polynomials, describing them as sums of traces of Lawrence type representations.

Our purpose is to create a connection between the initial definition of these quantum invariants, which has a purely representation theory nature and topology, using braid group actions on the homology of covering spaces and topological intersection pairings between these homologies.

 In \cite{Cr},  we presented a topological model for the sequence of coloured Jones polynomials. More precisely, we showed that one can obtain the coloured Jones polynomial $J_N(L,q)$ as a graded intersection pairing between two homology classes in certain coverings of configuration spaces of points in the punctured disc. We used the Lawrence representation and the description of these invariants as quantum invariants. 
 
 Pursuing this question, we will describe a topological model for the sequence of coloured Alexander invariants. We show that one can obtain $\Phi_N(L,\lambda)$ as a topological intersection pairing between two homology classes in certain versions of homological representations, called special Lawrence representations. 
\subsection*{ Description of the topological tools}

\

In this part, we will present a brief description of the homology groups that occur in the topological model. For $n,m \in \N$ we consider $C_{n,m}$ to be the unordered configuration space of $m$ points in the $n$-punctured disc (the two dimensional disc where we remove $n$ points). Then, we define a $\Z \oplus \Z$ covering space of this configuration space, by specifying a certain surjective morphism $$\varphi: \pi_1(C_{n,m})\rightarrow\Z \oplus \Z.$$
Let  $\tilde{C}_{n,m}$ of $C_{n,m}$ be the covering space corresponding to this morphism. In other words, the deck transformations of the covering are $\Z \oplus \Z=<x>_{\Z} \oplus <d>_{\Z}$. The first homology group that we use is the Borel-Moore homology of this covering space $$\HH^{lf}_{m}(\tilde{C}_{n,m}, \Z).$$ Using the deck transformations, this becomes a $\Z[x^{\pm},d^{\pm}]$-module. On the other hand, since the braid group $B_n$ is the mapping class group of the punctured disc, one gets an action of the braid group on this homology of the covering, which is compatible with the deck-transformations structure:

$ \ \ \ \ \ \ \ \ \ \ \ \ \ \ \ \ \ \ \ \ \ \ \ \ MCG \ \ \ \ \ \ \ \ \ \ \ \ \ \ \ \ \ \ \ \ \ \ \ \ \ \ Deck \ \ \ \ \ \ \ $
$$B_n\curvearrowright H^{lf}_m(\tilde{C}_{n,m},\Z)=\Z[x^{\pm},d^{\pm}]-module.$$
Certain subspaces in this homology were used by Kohno \cite{Koh}, who related them to the highest weight spaces in the $U_q(sl(2))$-Verma modules and later by Ito \cite{Ito2}, who defined a certain quotient of this group in order to create the connection towards the highest weight spaces of $U_q(sl(2))$-modules at roots of unity. All these are versions of Lawrence representations and they are discussed in Section \ref{3}. 

For our purpose, we have to handle in the same time good homological correspondents for highest weight spaces at roots of unity as well as defining dual homological groups. Here, by dual we mean that we need a non-degenerate intersection form between the first homological group and its dual, which should remain non-degenerate when specialised to roots of unity. All these details are discussed in section \ref{3'}. In order to solve the non-degeneracy problem, we will use a variation of the usual Borel-Moore homology, by splitting the "infinity" part of this space into two pieces. The boundary of $C_{n,m}$ has a component which is described by the fact that one of the points approaches a puncture and a different component which is characterised by the fact that two points collide at the open boundary of the configuration space. The first component corresponds to the hyperplanes inside the symmetric power of the whole disc which are defined by the equations corresponding to the $n$ punctures and we denote it by $H_{\ast}^{lf,\infty}$. The second component $H_{\ast}^{lf,\Delta}$ corresponds to the diagonal of the symmetric power. We will use a Poincar\'e-Lefschetz type duality for the covering $\tilde{C}_{n,m}$ considering the middle homlogies $$H^{lf,\infty}_m(\tilde{C}_{n,m},\Z) \ \ \ \text{and}  \ \ \ H^{lf,\Delta}_m(\tilde{C}_{n,m},\Z,\partial ).$$The rigorous definition of these spaces, called special Lawrence representations is discussed in \cite{CM}.
\subsection*{I.First homology group-special truncated Lawrence representation}
\

Now, we fix $N \in \N$ be the colour of the coloured Alexander invariant that we want to study. We will use certain subspaces inside these groups generated by classes of Lagrangian submanifolds, parametrised by partitions. We use the following notation:
$$ E_{n,m}=\{e=(e_1,...,e_{n-1})\in \N^n \mid e_1+...+e_{n-1}=m \}.$$
$$E^{\geq N}_{n,m}=\{e=(e_1,...,e_{n-1})\in E_{n,m} \mid \exists i , e_i \geq N\}.$$
For each partition $e$, one associates an $m$-dimensional disc $\F_{e}$ in the base space, which can be lifted to a submanifold $\tilde{\F}_e$ in the covering space. The explicit construction is presented in definition\ref{f}. Then, we consider the spaces generated by the classes given by these submanifolds \ref{P:20}:
$$\HH^{\infty}_{n,m}:= <[\tilde{\F}_e] \ | \  e\in E_{n,m}>_{\Z[x^{\pm},d^{\pm}]} \subseteq H^{\text{lf},\infty}_m(\tilde{C}_{n,m}, \Z).$$
$$\HH^{\infty, \geq N}_{n,m}:=<[[\tilde{\F}_e]] \ | \  e\in E^{\geq N}_{n,m}]>_{\Z[x^{\pm}, d^{\pm}]}\subseteq \HH^{\infty}_{n,m}.$$
In order to have a correspondent for the roots of unity situation on the quantum side, we consider the quotient space, and we call it truncated special Lawrence representation:
$$\HH^{\infty,N}_{n,m}:=\HH^{\infty}_{n,m}/\HH^{\infty,\geq N}_{n,m}.$$
\subsection*{ II.The dual homology space}
Dually we consider the subspace in the Borel-Moore homology with respect to collisions and relative to the boundary, which is generated by configurations spaces on vertical segments prescribed by partitions, as in definition \ref{D:21}:
$$\HH^{\Delta,\partial, N}_{n,m}:=<[[\tilde{\D}_f]] \ | \  f\in E^{\geq N}_{n,m}]>_{\Z[x^{\pm}, d^{\pm}]}\subseteq H^{\text{lf},\Delta}_m(\tilde{C}_{n,m}, \Z).$$
\subsection*{ III.Topological pairing}
The advantage of the choice of this splitting of the boundary of the configuration space is the fact that the duality leads to a topological pairing that can be easily computed. We have the following non-degenerate intersection form:$$\ll , \gg: \HH^{\infty, N}_{n,m} \otimes  \HH^{\Delta,\partial,N}_{n,m}\rightarrow \Z[x^{\pm}, d^{\pm}]$$
$$\ll \{ \{ \ [[\tilde{\F}_e]] \} \} \ ,[[\tilde{\D}_f]]\gg=\delta_{e,f}, \ \forall e,f \in E^{N}_{n,m}.$$
In the following picture we sketched the covering space $\tilde{C}_{n,m}$ that we work with as well as the submanifolds whose homology classes lead to the generators for the two homology groups, for the case where $ \left(n\rightarrow 2n-1; m\rightarrow(n-1)(N-1)\right)$.
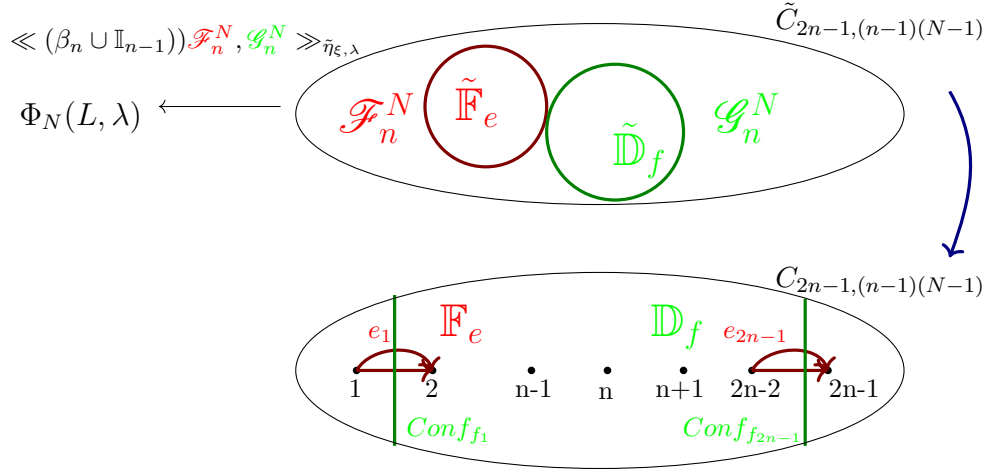
\begin{center}
\begin{tikzpicture}
\foreach \x/\y in {-1.2/2, -0.2/2 , 1.1/2 , 4/2 , 2.1/2 , 3.1/2 , 5/2 } {\node at (\x,\y) [circle,fill,inner sep=1pt] {};}
\node at (-1,2) [anchor=north east] {$1$};
\node at (0,2) [anchor=north east] {$2$};
\node at (1.5,2) [anchor=north east] {n-1};
\node at (2.3,1.9) [anchor=north east] {n};
\node at (3.5,2) [anchor=north east] {n+1};
\node at (4.5,2) [anchor=north east] {2n-2};
\node at (5.8,2) [anchor=north east] {2n-1};
\node at (-0.6,2.7) [anchor=north east] {$\color{red}e_1$};
\node at (4.6,2.7) [anchor=north east] {$\color{red}e_{2n-1}$};
\node at (0.7,1.5) [anchor=north east] {\color{green}$Conf_{f_1}$};
\node at (4.8,1.5) [anchor=north east] {$\color{green}Conf_{f_{2n-1}}$};
\node at (0.6,3) [anchor=north east] {\huge{\color{red}$\F_e$}};
\node at (3.5,3) [anchor=north east] {\huge{\color{green}$\D_f$}};
\node at (0.8,6) [anchor=north east] {\huge{\color{red}$\tilde{\F}_e$}};
\node at (3,5.4) [anchor=north east] {\huge{\color{green}$\tilde{\D}_f$}};
\node at (-0.3,5.8) [anchor=north east] {\huge{\color{red}$\mathscr{F}^N_n$}};
\node at (4.5,5.8) [anchor=north east] {\huge{\color{green}$\mathscr{G}^N_n$}};
\node at (-3.9,5.7) [anchor=north east] {\Large{$\Phi_N(L,\lambda)$}};
\node at (7.2,3.5) [anchor=north east] {\large{$C_{2n-1,(n-1)(N-1)}$}};
\node at (7.2,7) [anchor=north east] {\large{$\tilde{C}_{2n-1,(n-1)(N-1)}$}};

\node at (-1,6.7) [anchor=north east] {$ \ll(\beta_n \cup \mathbb I_{n-1})) {\color{red}\mathscr{F}^N_n }, {\color{green}\mathscr{G}^N_n} \gg_{\tilde{\eta}_{\xi,\lambda}}$};
\draw[->,color=black]             (-2.2,5.5) to node[yshift=-3mm,font=\small]{} (-3.8,5.5);
\draw [very thick,black!50!red,->](-1.2,2) to[out=60,in=110] (-0.2,2);
\draw [very thick,black!50!red,->](-1.2,2) to (-0.2,2);
\draw [very thick,black!50!red,->](4,2) to[out=60,in=110] (5,2);
\draw [very thick,black!50!red,->](4,2) to (5,2);
\draw [very thick,black!50!green,-](-0.7,1) to (-0.7,3);
\draw [very thick,black!50!blue,->](6.6,5.7) to[out=-60,in=70] (6.6,3.5);
 \draw[very thick,black!50!red] (0.5, 5.5) circle (0.8);
\draw [very thick,black!50!green,-](4.7,1) to (4.7,2.95);
 \draw[very thick,black!50!green] (2.2, 5.16) circle (0.9);

\
\draw (2,2) ellipse (4cm and 1.3cm);
\draw (2,5.4) ellipse (4cm and 1.2cm);

\end{tikzpicture}
\end{center}

\subsection*{Topological model}
The main result of this paper is the following:
 
 \begin{theorem}({\bf Topological model for coloured Alexander invariants})\label{THEOREM}
Let $N \in \mathbb N$ be the colour of the invariant and $\xi=e^{\frac{2\pi i}{2N}}$ the primitive $2N^{th}$ root of unity. Consider $n \in \N$ to be a natural number. Then, there exist two families of homology classes $$\mathscr F^{N}_n \in {H^{\infty,N}_{2n-1, (n-1)(N-1)}|}_{\tilde{\gamma}}  \ \ \ \text{ and } \ \ \  \mathscr G^{N}_n \in {H^{\Delta, \partial,N}_{2n-1, (n-1)(N-1)}|}_{\tilde{\gamma}}$$ such that if $L$ is a link for which there exists a braid $\beta \in B_{n}$ with $L=\hat{\beta} $ (normal closure), the $N^{th}$ coloured Alexander invariant of $L$ has the following form:
$$\Phi_N(L,\lambda)= \xi^{(N-1)\lambda w(\beta_n)} \ll (\beta_n \cup \mathbb I_{n-1})) \mathscr{F}^N_n , \mathscr{G}^N_n\gg_{\tilde{\eta}_{\xi,\lambda}}, \ \ \ \forall \lambda \in \C \setminus \Z .$$ 
\end{theorem}
\begin{remark*}
In the above formula $\ll \cdot , \cdot \gg$ is a topological intersection pairing which has values in a certain localisation $\mathscr I$ of the Laurent polynomial ring with two variables, included in $\mathbb Q(q,s)$, and the function $\tilde{\eta}_{\xi,\lambda}$ is a certain specialisation of coefficients, which evaluates this pairing to a complex number.
\end{remark*}

In the sequel we will discuss about the terms that occur in this model. We would like to emphasise that the two homology classes $\mathscr F^{N}_n$ and $\mathscr G^{N}_n$ are intrinsic and do not depend on the link. Moreover, they live in a specialisation of the special truncated Lawrence representation which is defined over a ring $\mathscr I$ in two variables.
The moment when the link plays a role is encoded into the braid group action of $\beta_n \cup I^{n-1}$. This action leads to a homology class which is still over two variables. Moreover, the particular choice of the dual space, ensures that the pairing $\ll \cdot , \cdot \gg$ specialised to the ring $\mathscr I$ is still non-degenerate. We obtain:
$$\ll (\beta_n \cup I^{n-1}) \mathscr{F}^N_n , \mathscr{G}^N_n\gg_{\tilde{\gamma}} \ \  \in \mathscr I=\mathbb Z[s^{\pm},q^{\pm1}](\mathscr L)\subseteq \mathbb Q(q,s).$$
The precise definitions are discussed in section \ref{10}. The last step, as presented in the picture below, is to evaluate this polynomial using the specialisation $\tilde{\eta_{\xi,\lambda}}$, which leads to a complex number. This is the coloured Alexander invariant.
\begin{center}
\begin{tikzpicture}
[x=1mm,y=1mm]
\node (t1) [color=red]   at (-4,-8)   {$(\beta_n\cup \mathbb I_{n-1})\mathscr F_{n}^N$};
\node (t2) [color=green]   at (34,-8)   {$\mathscr G_{n}^N$};
\node (b0) [color=red]   at (-20,7)   {$B_{2n-1}\curvearrowright$};

\node (b1) [color=cyan]   at (-30,0)   {$\ll ,\gg_{\tilde{\gamma}}:$};
\node (b2)  [color=black] at (-5,0)    {${\HH^{\infty,N}_{\big(2n-1,(n-1)(N-1)\big)}|}_{\tilde{\gamma}}$};
\node (b3) [color=black] at (35,0)   {${\HH^{\Delta,\partial,N}_{\big(2n-1,(n-1)(N-1)\big)}|}_{\tilde{\gamma}}$};

\node (b4) [color=cyan]   at (75,0)   {$\mathscr I=\mathbb Z[s^{\pm},q^{\pm1}](\mathscr L)$};

\node (b4') []   at (80,-30)   {$\C$};
\draw[->,color=black]             (b3) to node[yshift=-3mm,font=\small]{} (b4);
\node (b3') [color=black] at (15,0)   {$\bigotimes$};
\node (c) [color=cyan]   at (70,-8)   {$\ll (\beta_n \cup \mathbb I_{n-1})) {\color{red}\mathscr{F}^N_n }, {\color{green}\mathscr{G}^N_n} \color{cyan}\gg_{\tilde{\gamma}}$};
\node (ADO) []   at (68,-30)   {$\Phi_N(L,\lambda)\in$};
\draw[->,color=black]             (68,-14) to node[xshift=-8mm,font=\small]{$\tilde{\eta}_{\xi,\lambda}$} (68,-25);
\node (eta) [color=black] at (79,-15)   {$\tilde{\eta}_{\xi,\lambda}(q)=\xi$};
\node (eta') [color=black] at (80,-20)   {$\tilde{\eta}_{\xi,\lambda}(s)=\xi^{\lambda}$};
\end{tikzpicture}
\end{center}
We would like to stress that the non-semisimplicity of the coloured Alexander invariant on the algebraic side can be read homologically. For a link that can be obtained as a closure of a braid with $n$ strands, we used the configuration space of points in the $(2n-1)$-punctured disc. In contrast, for the case of the coloured Jones polynomials, for example, in the topological model from \cite{Cr}, we used the configuration space of points in the $2n$-punctured disc. This difference is the result of the fact that the coloured Jones polynomial as a quantum invariant comes from a quantum trace whereas the coloured Alexander invariant is defined using a partial quantum trace. The topological effect of the partial trace can be read through the fact that we used one puncture less in the non-semi simple situation.

\subsection*{ Structure of the paper}

 This paper is split into ten main parts. In Section \ref{section2}, we present the quantum group $U_q(sl(2))$ that we work with and properties about its representation theory. Then, Section \ref{section3} concerns the notion of highest weight spaces and generic quantum representations of the braid groups. Section \ref{2} is devoted to the case where $q$ is a root of unity, the representation theory in this situation and the definition of the coloured Alexander polynomials. In Section \ref{3}, we present the homological braid group representations introduced by Lawrence and their truncated version defined by Ito. Then, section \ref{3'} continues with a version of these representations, called special Lawrence representations, which will help when one makes the specialisation to the roots of unity. Section \ref{4} is devoted to the connection between homological and quantum representations of the braid group. In Section \ref{4'}, we define a suitable version of the dual Lawrence representation, and we discuss the non-degeneracy of a topological intersection pairing between the special Lawrence representation and its dual. Section \ref{5} deals with the specialisations of this pairing at roots of unity. In Section \ref{6}, we present a topological model for the coloured Alexander invariants with homology classes that live in certain specialisations of the homology groups. Then, Section \ref{10} is devoted to the globalisation of these homology classes, showing that they can be lifted towards two classes in certain homology groups which are defined over a ring which has two parameters.
\clearpage
\subsection*{Acknowledgements} 
 I would like to thank Professor Christian Blanchet for guiding me towards this research direction and for inspiring conversations that we had over the years. This paper was prepared at the University of Oxford, and I acknowledge the support of the European Research Council (ERC) under the European Union's Horizon 2020 research and innovation programme (grant agreement No 674978).
 
\section{Representation theory of $U_q(sl(2))$ at roots of unity}\label{section2}
\subsection{The quantum group $U_q(sl(2))$}
In this part, we will present the quantum enveloping algebra of $sl(2)$ that we will work with and discuss about its representation theory at roots of unity.
We will work over $\mathbb C$.
\begin{notation} 1) We use the following notations:
 $$ \{ x \} :=\qs^x-\qs^{-x} \ \ \ \ [x]_{\qs}:= \frac{\qs^x-\qs^{-x}}{\qs-\qs^{-1}}$$
$$[n]_\qs!=[1]_\qs[2]_\qs...[n]_\qs$$
$${n \brack j}_\qs=\frac{[n]_\qs!}{[n-j]_\qs! [j]_\qs!}.$$
\end{notation}
\begin{definition}
Two complex numbers $\lambda, q \in \C$ are said to be generic if $$[\lambda-k]_q\neq0, \forall k\in \N.$$
\end{definition}
\begin{notation}(Specialisation of coefficients)\label{N:spec}\\ 
Let $R$ be a ring and $M$ an $R$-module with a basis $\mathscr B$.
Consider $S$ to be another ring and suppose that we have a specialisation of the coefficients, given by a ring morphism:
$$\psi: R \rightarrow S.$$
We will denote by $M|_{\psi}$ to be the specialisation of the module $M$ by the change of coefficients $\psi$, which is the following $S$-module: $$M|_{\psi}:=M \otimes_{R} S$$ with the basis $$\mathscr B_{M|_{\psi}}:=\mathscr B \otimes_{R}1 \in M|_{\psi}. $$
\end{notation}
\begin{definition}
2) Consider the quantum group $U_q(sl(2))$ generated by the elements $E, F, K^{\pm1}$ subject to the relations:
$$KK^{-1}=K^{-1}K=1; KE=\qs^2EK; KF=\qs^{-2} FK$$
$$[E,F]=\frac{K-K^{-1}}{q-q^{-1}}.$$ 
\end{definition}
This is a Hopf algebra with the following comultiplication, counit and antipode:
\begin{align*}
\Delta(E)= &E\otimes K+ 1 \otimes E ,  & S(E)=&-EK^{-1}\\
\Delta (F) = &F\otimes 1+ K \otimes F, & S(F)=&-KF\\
\Delta(K^{\pm 1})=&K^{\pm 1}\otimes K^{\pm 1}, & S(K)=&K^{-1}.\\
\end{align*}

Now we will describe the representation theory of $U_q(sl(2))$. 
\begin{definition} (The universal Verma module)
Let $\lambda \in \C$ a complex parameter.
Consider $\hat{V}_{\lambda}$ to be a complex vector space generated by the following infinite family of vectors:
$$\hat{V}_{\lambda}:=<\hat{v}_0, \hat{v}_1,... >_{\C}$$ defined with the following $U_q(sl(2))-$ action:
\begin{align*}
K\hat{v_i}=&\qs^{\lambda-2i}\hat{v}_i\\
E \hat{v_i}=&\hat{v}_{i-1}\\
F \hat{v_i} =& [i+1]_q[\lambda-1]_q  \ \hat{v}_{i+1}.
\end{align*}
\end{definition}
\begin{remark} Let $N$ be a natural number. We notice that if we look at the action of the generator $F$ on the vector $v_{N-1}$ that is
$$F \hat{v}_{N-1}=[(N-1)+1]_q[\lambda-(N-1)]_q \hat{v}_{N}.$$ 
If one has this coefficient that vanishes, 
$$[(N-1)+1]_q[\lambda-(N-1)]_q=0 $$
one gets as an $N$-dimensional subrepresentation in the generic Verma module:
$$U^{q,\lambda}_N:=<\hat{v}_0,...\hat{v}_{N-1}>\subseteq \hat{V}_{\lambda}. $$ 
There are two cases when this happens:\\
$1)$ \textbf{Coloured Jones polynomial}

( {\bf $\lambda=N-1$ fixed parameter and $q$ variable}).

In this case, $U^{q,N-1}_N$ is isomorphic to the standard $N$-dimensional irreducible representation of $U_q(sl(2))$. The Reshetikhin-Turaev procedure for this invariant leads to the $N^{th}$  coloured Jones polynomial for links $$J_N(L,q)\in \Z[q^{\pm 1}].$$

$2)$\textbf{Coloured Alexander polynomial (ADO)}

{\bf ( $\lambda$ variable and $q=\xi=e^{\frac{2 \pi i}{2N}}$-$2N^{th}$ root of unity )}

 As we will discuss in the following section, $U^{\xi,\lambda}_N$ rise to a one parameter family of representations of $U_{\xi}(sl(2))$. The $N$-th ADO invariant is obtained from this input through a renormalised Reshetikhin-Turaev type procedure:
  $$\phi_N(L,\lambda)\in \Z[q^{\pm 1},q^{\pm \lambda}].$$
\end{remark}

\begin{proposition}
1) The universal R-matrix of the quantum group $U_q(sl(2))$ has the following form:
$$ R=q^{\frac{H \otimes H}{2}} \sum_{n=0}^{\infty} q^{\frac{n(n-1)}{2}}\frac{\{1\}^{2n}}{ \{ n\}_q!} E^n\otimes F^n.$$
where $q^H=K$ as operators. This action on the tensor product of two generic Verma modules $\hat{V}_{\lambda_1}\otimes \hat{V}_{\lambda_2} $ means:
$$q^{\frac{H \otimes H}{2}}(\hat{v}_i\otimes \hat{v}_j)= q^{\frac{1}{2}(\lambda_1-2i)(\lambda_2-2j)}.$$
2) Consider the following operator:
$$\mathscr R_{\lambda_1,\lambda_2}:= R \circ \tau \in Aut_{U_q(sl(2))}(\hat{V}_{\lambda_1}\otimes \hat{V}_{\lambda_2}).$$
\end{proposition}
\begin{proposition}
This leads to a braid group representation:
$$ \hat{\varphi}^{q,\lambda}_n: B_n \rightarrow \End_{U_q(sl(2))}\left( (\hat{V}_{\lambda})^{\otimes n}\right)$$ 
$$\sigma _i\longleftrightarrow Id_V^{(i-1)}\otimes \mathscr R_{\lambda,\lambda} \otimes Id_V^{(n-i-1)}.$$
\end{proposition}
We will use a certain submodule of this Verma module, which leads to a different action of the quantum group, which will be useful regarding the correspondence with homological braid group representations. We discuss this in section \ref{4}
\begin{notation}
Let us define the vectors:
$$v_i:= [\lambda;i] \ \hat{v}_i \in \hat{V}_{\lambda}.$$
\end{notation}
\begin{definition}(Verma module $V^{q,\lambda}$ of $U_q(sl(2))$ )

Consider the following submodule generated by the normalised vectors from above:
$$V^{q,\lambda}:=<v_0,v_1,...>_{\C} \ \subseteq \hat{V}_{\lambda}.$$
As for the previous case, the $R$-matrix leads to a braid group representation, induced by the representation $\hat{\varphi}^{q,\lambda}_n$ :
$$ \varphi^{q,\lambda}_n: B_n \rightarrow \End_{U_q(sl(2))}\left( (V^{q,\lambda})^{\otimes n}\right).$$ 

\end{definition}
\begin{proposition}\label{P:actions}
The action of the quantum group on the Verma module $V^{\lambda,\xi}$ is the following:
\begin{align*}
Kv_i=&\qs^{\lambda-2i}v_i\\
Ev_i=&[\lambda+1-i]_q \  v_{i-1}\\
Fv_i =& [i+1]_q \ v_{i+1}.
\end{align*}
\end{proposition}
\begin{remark}

For natural parameters $\lambda=N-1 \in \N$, $V^{q,N-1}$ is the standard $N$-dimensional representation of $U_q(sl(2))$.

Otherwise, for non-natural cases $\lambda \in \C \setminus \N$, this module is the same as the generic Verma module:
$$V^{q,\lambda}\simeq\hat{V}_{\lambda}.$$
\end{remark}

\subsection{The quantum group at roots of unity}
Let $N\in \mathbb N$ and $q=\xi=e^{\frac{2 \pi i}{2N}}$ a $2N^{th}$-root of unity.
We will denote by $U_{\xi}(sl(2))$ the quantum group corresponding to $q=\xi$.
\begin{definition}(Finite dimensional representations of $U_{\xi}(sl(2))$)\\
Let $\lambda \in \C$ generic with respect to $\xi$. Then we denote by:
$$U_N^{\lambda}:=V_N^{\xi,\lambda}\subseteq \hat{V}_{\lambda}$$ 
\end{definition}
\begin{lemma}
The braid group action $\hat{\varphi}^{\xi,\lambda}_n$ preserves the submodule $U_N^{\lambda}$, and leads to a braid group representation:
$$ \varphi^{\xi,\lambda}_n: B_n \rightarrow End_{U_{\xi}(sl(2))}(U_N^{\lambda}{^{\otimes n}})$$ 
$$\sigma _i\longleftrightarrow Id_V^{(i-1)}\otimes \mathscr R_{\lambda,\lambda} \otimes Id_V^{(n-i-1)}.$$
\end{lemma}

\begin{proposition}
Let us consider the following dualities in the category of finite dimensional representations of $U_\xi(sl(2))$:
The dualities of this category have the following form:
$$\forall \lambda \in \C \setminus \N$$ 
\begin{align}\label{E:DualityForCat}
\tcoev_{\lambda} :\, & \C \rightarrow U_N^{\lambda}\otimes {U_N^{\lambda}}^{*} \text{ is given by } 1 \mapsto
  \sum v_j\otimes v_j^*,\notag
  \\
 \tev_{\lambda}:\, & {U_N^{\lambda}}^*\otimes U_N^{\lambda}\rightarrow \C\text{ is given by } f\otimes w
  \mapsto f(w),\notag
  \\
\coev_{\lambda}:\, & \C \rightarrow {U_N^{\lambda}}^*\otimes U_N^{\lambda} \text{ is given by } 1 \mapsto \sum
   v_j^*\otimes K^{N-1}v_j,
  \\
\ev_{\lambda}:\, & U_N^{\lambda}\otimes {U_N^{\lambda}}^{*}\rightarrow \C \text{ is given by } v\otimes f
  \mapsto  f(K^{1-N}v),\notag
\end{align}
for $\{v_j\}$ a basis of $U_N^{\lambda}$ and $\{v_j^*\}$ the dual basis
of ${U_N^{\lambda}}^*$.\\
\end{proposition}
\begin{definition} (Category of coloured oriented tangles)\label{tan}

Let $\cat$ be a category. 
Then, let $\mathscr T_\cat$ to be the category of $\cat$-colored framed tangles, defined in the following manner:
$$Ob( \mathscr T_\cat)= \{\mathbb{V}=\big( (V_1,\epsilon_1),...,(V_m,\epsilon_m) \big) \mid m\in \N,\epsilon_i \in \{ \pm 1 \},V_i \in  \cat \}.$$ 
$$\Hom_{ \mathscr T_\cat}\left(\mathbb{V},\mathbb{W}) \right) = \{ \cat- \text{colored framed tangles} \ \mathscr T : \mathbb{V} \nearrow \mathbb W \} / \text{isotopy}$$
$$\text{ for any } \mathbb{V}=(V_1,\epsilon_1),...,(V_m,\epsilon_m); \ \mathbb{W}=(W_1,\delta_1),...,(W_n,\delta_n) \in Ob(T_{\cat}).$$
Remark: We look at morphisms $T \in \Hom_{ \mathscr T_\cat}\left(\mathbb{V},\mathbb{W} \right) $ in the category $\mathscr C$ as being represented geometrically by tangles which go from the bottom part, coloured with $\mathbb V$ to the top part, coloured with $\mathbb W$. 
The colours on the strands of the tangle  $T$ have to agree with the colors at its boundaries $\mathbb V$ and $\mathbb W$.
Once we have such a tangle, it has an induced orientation, coming from the signs $\epsilon_i$, using the following conventions:
$$(V,-) \ \downarrow,  \ \ \ (V,+) \ \uparrow.$$
\end{definition}

\begin{theorem}(Reshetikhin-Turaev)\\
There exist an unique monoidal functor $$\F: \mathscr T_{Rep^{f. dim}(U_\xi(sl(2)))}\rightarrow Rep^{f. dim}(U_\xi(sl(2)))$$
such that $\forall \ U_N^{\lambda_1},U_N^{\lambda_2} \in Rep^{f. dim}(U_\xi(sl(2)))$, it respects the following local relations:
$$1) \ \F((U_N^{\lambda_1},+))=U_N^{\lambda_1}; \ \ \ \F((U_N^{\lambda_1},-))=(U_N^{\lambda_1})^{*} \ \ \ \ \ \ \ \ \ \ \ \ $$ 
\begin{align*} 
2) \ \F(
\tikz[x=1mm,y=1mm,baseline=0.5ex]{\draw[->] (3,3)--(0,0); \draw[line width=3pt,white] (0,3)--(3,0); \draw[->] (0,3)--(3,0);}
) &=  R_{{\lambda_1},{\lambda_2}} \in Hom_{U_q(sl(2))}(U_N^{\lambda_1}\otimes U_N^{\lambda_2}\rightarrow U_N^{\lambda_2} \otimes U_N^{\lambda_1}) \\
\F(
\tikz[x=1mm,y=1mm,baseline=0.5ex]{\draw[<-] (0,3) .. controls (0,0) and (3,0) .. (3,3); \draw[draw=none, use as bounding box](-0.5,0) rectangle (3,3);}
) &= \tcoev_{{\lambda_1}} :\C\rightarrow U_N^{\lambda_1} \otimes (U_N^{\lambda_1})^{*} \\
\F(
\tikz[x=1mm,y=1mm,baseline=0.5ex]{\draw[<-] (3,0) .. controls (3,3) and (0,3) .. (0,0); \draw[draw=none, use as bounding box](0,0) rectangle (3.5,3);}
) &= \ev_{{\lambda_1}} :U_N^{\lambda_1} \otimes (U_N^{\lambda_1})^{*} \rightarrow \mathbb \C \\
\F(
\tikz[x=1mm,y=1mm,baseline=0.5ex]{\draw[->] (0,3) .. controls (0,0) and (3,0) .. (3,3); \draw[draw=none, use as bounding box](-0.5,0) rectangle (3,3);}
) &= \coev_{{\lambda_1}} :\C \rightarrow U_N^{\lambda_1} \otimes (U_N^{\lambda_1})^{*} \\
\F(
\tikz[x=1mm,y=1mm,baseline=0.5ex]{\draw[->] (3,0) .. controls (3,3) and (0,3) .. (0,0); \draw[draw=none, use as bounding box](0,0) rectangle (3.5,3);}
) &= \tev_{{\lambda_1}} :U_N^{\lambda_1} \otimes (U_N^{\lambda_1})^{*} \rightarrow \mathbb \C.
\end{align*}
\end{theorem}
\section{Quantum representation of the braid group}\label{section3}
In this section, we will introduce certain important subspaces in the tensor power of $U_q(sl(2))$-representations. They are called highest weight spaces and carry a braid group action, called quantum representation. They have a rich structure, and we will see in later sections that they are related to homological braid group representations.

\clearpage

\subsection{(Highest) weight spaces}
\begin{center}
\begin{tabu} to 1.02\textwidth { | X[c] | X[c] | X[c] | X[c] | } 
 \hline
Quantum group:&   $U_q(sl(2))$   &   $U_q(sl(2))$  &  $ U_{\xi}(sl(2))$\\
\hline
&Universal Verma module    &       Verma module    &        N-dimensional representation\\
\hline
& $\hat{V}_{\lambda}$  &   $V^{q,\lambda}$ &  $ U_N^{\lambda}=V^{\lambda,\xi} $\\
\hline
Braid group action  & $\hat{\varphi}^{q,\lambda}_{n}$ &  $\varphi^{q,\lambda}_{n}$ &    $\varphi^{\xi,\lambda}_{n}$\\

\hline
Weight spaces    & $\hat{V}^{q,\lambda}_{n,m}$ &  $V^{q,\lambda}_{n,m}$ &  $V^{\xi,\lambda}_{n,m}$\\
\hline
Highest weight spaces & $\hat{W}^{q,\lambda}_{n,m}$ &  $W^{q,\lambda}_{n,m}$  & $W^{\xi,\lambda}$ \\
\hline
Braid group action (quantum representation) & &  $\varphi^{q,\lambda}_{n,m}$ &    $\varphi^{\xi,\lambda}_{n,m}$\\
\hline
\end{tabu}
\end{center}
                                                   
\begin{notation}
For $n,m \in \N$, we consider the following sets of partitions:
$$E_{n,m}=\{e=(e_1,...,e_{n-1})\in \N^n \mid e_1+...+e_{n-1}=m \}$$
$$E^{N}_{n,m}=\{e \in E_{n,m} \mid e_i\leq N-1, \ \forall i \in \overline{1,n-1} \ \}.$$
It is known that:
$$\text{card} (E_{n,m})={n+m-2 \choose m}=^{not}d_{n,m}.$$
We denote by
$$\text{card} (E^N_{n,m})=^{not}d^N_{n,m}.$$
\end{notation}

\begin{definition}Let us consider the parameters $n,m \in \N$.

{\bf I.Generic (highest) weight spaces:}
$$\text{Generic weight space} \ \ \ \ \ \ \ \ \ \ \hat{V}^{q,\lambda}_{n,m}:=Ker(K-q^{n\lambda-2m}Id) \ \subseteq \hat{V}_{\lambda}^{\otimes n}. \ \ \ \ \ $$ 
$$\text{Generic highest weight space} \ \ \hat{W}^{q,\lambda}_{n,m}:= \ \hat{V}^{q,\lambda}_{n,m} \ \ \cap \ Ker E \ \ \  \subseteq \ \hat{V}_{\lambda}^{\otimes n}. \ \ \ \ \ $$ 
{\bf II. (Highest) weight spaces:}
$$\text{Weight space} \ \ \ \ \ \ \ \ \ \ V^{q,\lambda}_{n,m}:=Ker(K-q^{n\lambda-2m}Id) \ \subseteq (V^{q,\lambda})^{\otimes n}. \ \ \ \ \ $$ 
$$\text{Highest weight space} \ \ W^{q,\lambda}_{n,m}:= \ V^{q,\lambda}_{n,m} \ \ \cap \ Ker E \ \ \  \subseteq \ (V^{q,\lambda})^{\otimes n}. \ \ \ \ \ $$ 
\end{definition}
\begin{remark}(Basis for the weight spaces)\label{R:basisweight}
From the formulas for the $K$-action, one gets a natural basis for weight spaces:
$$\hat{V}^{q,\lambda}_{n,m}=<\hat{v}_{e_1}\otimes ... \otimes \hat{v}_{e_{n-1}}\mid \ e=(e_1,...,e_{n-1}) \in E_{n,m}>_{\C}.$$
$$V^{q,\lambda}_{n,m}=<v_{e_1}\otimes ... \otimes v_{e_{n-1}}\mid \ e=(e_1,...,e_{n-1}) \in E_{n,m}>_{\C}.$$
\end{remark}

\subsection{Basis for the highest weight spaces}

We have seen above that the weight spaces have natural bases given by elements parametrised by certain partitions of integers. In contrast to this case, it is more subtle to describe bases for highest weight spaces. In the following part, we will present such bases following \cite{JK} and \cite{Ito}.
\begin{definition} (Basis for $W^{q,\lambda}_{n,m}$)

For $e \in E_{n+1,m}$, consider the following vector:
$$\bar{v}^{q,\lambda}_e:= q^{\lambda \sum_{i=1}^ni e_i}v_{e_1}\otimes... \otimes v_{e_n}.$$
Then, the set $\mathscr B_{V^{q,\lambda}_{n,m}}:=\{ \bar{v}^{q,\lambda}_e | e \in E_{n+1,m} \}$ gives a basis for the weight space $V^{q,\lambda}_{n,m}$.
\end{definition}
Further on, the highest weight space $W^{q,\lambda}_{n,m}$ will be identified with a certain subspace inside the weight spaces $V^{q,\lambda}_{n,m}$.

Let us consider the inclusion $\iota: E_{n,m} \rightarrow E_{n+1,m}$ defined by:
$$\iota ((e_1,...,e_{n-1}))=(0,e_1,...,e_{n-1}).$$
Denote the subspace $$\bar{V}^{q,\lambda}_{n,m}:=\C v_0 \oplus V^{q,\lambda}_{n-1,m}\subseteq V^{q,\lambda}_{n,m}.$$ 
Then, $\mathscr B_{\bar{V}^{q,\lambda}_{n,m}}:= \{ \bar{v}^{q,\lambda}_{\iota(e)}| e \in E_{n,m} \}$ will give a basis for the space $\bar{V}^{q,\lambda}_{n,m}$.
\begin{proposition}(\cite{Ito}\label{P:bqu})

Define the following function $\phi_{q,\lambda}: \bar{V}^{q,\lambda}_{n,m}\rightarrow W^{q,\lambda}_{n,m}$ by the formula:
$$\phi_{q,\lambda}(w):=\sum_{k=0}^{m} (-1)^k q^{-k(n-1)\lambda} q^{2mk-k(k+1)} v_k \otimes E^k(w).$$

If $\lambda$ is generic with respect to $q$,  then $\phi_{q,\lambda}$ is an isomorphism of $\C$-vector spaces.

The set $\mathscr B_{W^{q,\lambda}_{n,m}}=\{ \phi_{q,\lambda}(\bar{v}^{q,\lambda}_{\iota(e)})| e \in E_{n,m}\}$ gives a basis for the highest weight space  $W^{q,\lambda}_{n,m}$.
It follows that:
$$\text{ dim} \ (W^{q,\lambda}_{n,m} )=d_{n,m}={n+m-2 \choose m}.$$
\end{proposition}
\subsection{Generic quantum representation}
\

\begin{definition}(Braid group action)

Since the braid group representation $\varphi^{q,\lambda}_{n}$ commutes with the quantum group action, it leads to well defined representations on the (generic) highest weight spaces.
$$ \hat{\varphi}^{q,\lambda}_{n,m}:B_n\rightarrow End(\hat{W}^{q,\lambda}_{n,m}). $$
$$ \varphi^{q,\lambda}_{n,m}:B_n\rightarrow End(W^{q,\lambda}_{n,m}).$$
\end{definition}
\begin{definition} (Quantum representation)

Suppose $\lambda,q \in \C$ such that $\lambda$ is generic with respect to $q$. Then, the braid group representation $\varphi^{q,\lambda}_{n,m}$ written in the basis $\mathscr B_{W^{q,\lambda}_{n,m}}$ is called the quantum representation of the braid group:
$$\varphi^{q,\lambda}_{n,m}:B_n \rightarrow Aut(W^{q,\lambda}_{n,m}, B_{W^{q,\lambda}_{n,m}})\simeq GL(d_{n,m}, \C).$$

\end{definition}
\subsection{Highest weight spaces at roots of unity}
Let $N \in \N$ to be the level and $\xi=e^{\frac{2 \pi i}{2N}}$ be a $2N$-root of unity. We use the following subspaces.
\begin{definition} 
Let $n,m \in \N$ two parameters and $\lambda \in \C\setminus \N$. Consider the following subspaces in the tensor power of the representation of $U_{\xi}(sl(2))$:
$$\text{Weight space at roots of unity} \ \ \ \ \ \ \ \ \ \ V^{\xi,\lambda}_{n,m}:=Ker(K-\xi^{n\lambda-2m}Id )\ \subseteq (U^{\lambda}_N)^{\otimes n}. \ \ \ \ \ $$ 
$$\text{Highest weight space at roots of unity} \ \ W^{\xi,\lambda}_{n,m}:= \ V^{\xi,\lambda}_{n,m} \ \ \cap \ Ker E \ \ \  \subseteq \ ((U_N^{\lambda})^{\otimes n}. \ \ \ \ \ $$ 
\end{definition}
\subsection{Basis for the highest weight spaces at roots of unity}

In the sequel, we are interested in describing a basis for the highest weight spaces at roots of unity. For the generic case, we have seen that a good indexing set for $W^{q,\lambda}_{n,m}$ is the set of partitions $E_{n,m}$. When we pass to the case of roots of unity, because we have the vectors with indices maximum $N-1$, we will have a basis indexed by a smaller set. We will present this following \cite{Ito2}.

As before, consider the subspace $$\bar{V}^{\xi,\lambda}_{n,m}:=\C v_0 \oplus V^{\xi,\lambda}_{n-1,m}\subseteq V^{\xi,\lambda}_{n,m}.$$ 
Then, $\mathscr B_{\bar{V}^{\xi,\lambda}_{n,m}}:= \{ \bar{v}^{q,\lambda}_{\iota(e)}| e \in E^N_{n,m} \}$ give a basis for the space $\bar{V}^{\xi,\lambda}_{n,m}$.
\begin{proposition}(\cite{Ito2}\label{P:bqu1})

Define the following function $\phi_{\xi,\lambda}: \bar{V}^{\xi,\lambda}_{n,m}\rightarrow W^{\xi,\lambda}_{n,m}$ by the formula:
$$\phi_{\xi,\lambda}(w):=\sum_{k=0}^{N-1} (-1)^k {\xi}^{-k(n-1)\lambda} {\xi}^{2mk-k(k+1)} v_k \otimes E^k(w).$$
Then, for $\lambda$ generic with respect to $\xi$, the set 
$$\mathscr B_{W^{\xi,\lambda}_{n,m}}=\{ \phi_{\xi,\lambda}\big(\bar{v}^{\xi,\lambda}_{\iota(e)}\big)| e \in E^N_{n,m}\}$$ gives a basis for the highest weight space at roots of unity $W^{\xi,\lambda}_{n,m}$.
\end{proposition}
\begin{remark}
It follows that:
$$\text{ dim} \ (W^{\xi,\lambda}_{n,m} )=\text{card} \ (E^N_{n,m})=d^N_{n,m} \leq d_{n,m}=\text{ dim} \ (W^{q,\lambda}_{n,m} ).$$
which emphasises the difference that occurs when we pass towards roots of unity. 
\end{remark}

\subsection{Quantum representation at roots of unity}{\label{Quantumrep}}
Now we are ready to introduce the representation of the braid group on the highest weight spaces at roots of unity, which will play an inportant role for the topological model for the coloured Alexander polynomials, as we will see later. 

\begin{definition}(\cite{Ito2})(Quantum representation at roots of unity)

Suppose $\lambda \in \C \setminus \N$ is generic with respect to $\xi$. Then, the braid group representation $\varphi^{\xi,\lambda}_{n,m}$ written in the basis $\mathscr B_{W^{\xi,\lambda}_{n,m}}$ is called the quantum representation at roots of unity of the braid group:
$$\varphi^{\xi,\lambda}_{n,m}:B_n \rightarrow Aut(W^{\xi,\lambda}_{n,m}, B_{W^{\xi,\lambda}_{n,m}})\simeq GL(d^N_{n,m}, \C).$$

\end{definition}

\section{The ADO polynomial}\label{2}
Let us fix a colour $\lambda \in \C$. We will consider an invariant for oriented links coloured with the representation $U_N^{\lambda}$. 
\begin{notation}
Let $\mathscr T_{U_N^{\lambda}}\subseteq \mathscr T_{Rep^{f. dim}(U_q(sl(2)))}$ the full subcategory of oriented tangles whose objects are sequences of points coloured with:
$$(U_N^{\lambda},+) \ \ \ (U_N^{\lambda},-).$$
We denote by 
$$\mathscr F_{\lambda}:=\F \mid_{\mathscr T_{U_N^{\lambda}}}: \mathscr T_{U_N^{\lambda}}\rightarrow Rep^{f. dim}(U_\xi(sl(2))) .$$
\end{notation}
one of the main features of this method is the property that in general, the Reshetikhin-Turaev construction for quantum groups leads to an invariant for coloured links. The main issue consists in the problem that for quantum groups at roots of unity, this invariant vanishes.

We start with $L$ which is an oriented link. Then, we can see it in the category of tangles as:
$$L \in Hom_{\mathscr T_{U_N^{\lambda}}}(\varnothing, \varnothing).$$
\begin{proposition}
For any $\lambda \in \C \setminus \N$, the previous construction leads to the vanishing invariant:
$$\mathscr F_{\lambda}(L)=0 \in Hom_{U_{\xi}(sl(2))}(\C, \C).$$
\end{proposition}
In their paper, Akustu, Deguchi and Ohtsuki (\cite{ADO}) introduced the idea of normalisation, which involves cutting a strand of the link and applying the Reshetikhin-Turaev functor to the new $(1,1)-$tangle that is obtained out of it. Let us make this precise.

Let us choose a strand of the link $L$ and denote by $T$ the $(1,1)$- tangle obtained from $L$ after cutting this strand. 
\begin{remark}
It is known that for $\lambda \in \C \setminus \N $, the representation $U_N^{\lambda}$ is simple. This means that:
$$ End_{U_{\xi}(sl(2))}(U_N^{\lambda})\simeq \C \cdot Id_{U_N^{\lambda}}.$$
For $ f \in End_{U_{\xi}(sl(2))}(U_N^{\lambda})$, which will be a scalar times identity, we denote the scalar as follows:
$$ f=<f> \cdot Id \in \C \cdot Id_{U_N^{\lambda}}.$$
\end{remark}
We notice that all the coefficients that occur in the braiding formula and also in the expressions used for dualities are in $\Z[q^{\pm1},q^{\pm \lambda}]$. This shows that all the coefficients that appear in the Reshetikhin-Turaev functor are in fact in this ring. Putting these together, the ADO invariant is defined as follows. 
\begin{definition} (The ADO invariant \cite{ADO}, \cite{Ito2})
Consider the following quantity associated to the $(1,1)-tangle$:
$$\Phi_{N}(T,\lambda):= <\mathscr F_{\lambda}(T)> \ \in \Z[q^{\pm1},q^{\pm \lambda}].$$
Then, it leads to a well defined invariant for oriented links:
$$\Phi_{N}(L,\lambda):=\Phi_{N}(T_L,\lambda)$$
where the tangle $T_L$ is obtained from $L$ by cutting one strand.
This is called the $N^{th}$ coloured Alexander invariant (ADO invariant) of the oriented link $L$.
\end{definition}
Following the functoriality of the Reshetikhin-Turaev construction, we can express the coloured Alexander invariant of a link using the dualities $\ev$, $\tcoev$ and the braid group action $\varphi_n^{\lambda}$. We will see this in the following part.
\begin{notation}
We denote by $$w: B_n\rightarrow \Z$$ the abelianisation map for the braid group.\\
Also, by $\mathbb I_n$ we denote the trivial braid with $n$ strands in $B_n$.
\end{notation}
If we start with a link $L$, let us consider a braid $\beta_n\in B_n$ that leads to $L$ by braid closure. Then, we cut the first strand and obtain a $(1,1)-$tangle. In the sequel, we split this diagram into three main parts: the bottom part, the braid and the upper level, which will have the following components:

1) the evaluation                              
$ \ \ \ \ \ \ \ \ \   \uparrow \tikz[x=1mm,y=1mm,baseline=0.5ex]{\draw[<-] (3,0) .. controls (3,3) and (0,3) .. (0,0); \draw[<-] (6,0) .. controls (6,6) and (-3,6) .. (-3,0); \draw[draw=none, use as bounding box](0,0) rectangle (3.5,3);}
$

2) braid level                                    $ \ \ \ \ \ \ \ \ \ \ \ \ \ \ \beta_n \  \otimes \ \unit_{n-1}.$

3) the coevaluation                                 
$ \ \ \ \ \ \  \uparrow  \tikz[x=1mm,y=1mm,baseline=0.5ex]{\draw[<-] (0,3) .. controls (0,0) and (3,0) .. (3,3); \draw[<-] (-3,3) .. controls (-3,-3) and (6,-3) .. (6,3); \draw[draw=none, use as bounding box](-0.5,0) rectangle (3,3);}
$

The ADO invariant of $L$ is constructed using the functor $\mathscr F_{\lambda}$. We write the expression that one gets by evaluating this functor onto the three levels of the tangle, which are drawn above.
\begin{notation}
For $V_1,...,V_n$ vector spaces over $\C$, and $i\in \N$ we consider the twisted operator:
$\tau_(i,n):V_1\otimes...\otimes V_n \rightarrow V_1 \otimes V_{i-1} \otimes V_n \otimes V_{i+1}...\otimes V_{n-1} \otimes V_i$
which interchanges the $i^th$ and $n^{th}$ components.

We denote by 
$$\ev^i_{\lambda}:(U_N^{\lambda})^{\otimes i}\otimes((U_N^{\lambda})^{\star})^{\otimes i}\rightarrow(U_N^{\lambda})^{\otimes i-1} \otimes ((U_N^{\lambda})^{\star})^{\otimes i-1}. $$
$$\tcoev^i_{\lambda}:(U_N^{\lambda})^{\otimes i-1}\otimes((U_N^{\lambda})^{\star})^{\otimes i-1}\rightarrow(U_N^{\lambda})^{\otimes i} \otimes ((U_N^{\lambda})^{\star})^{\otimes i}. $$
 the evaluation ( and coevaluation ) corresponding to the first and last component:
$$\ev^i_{\lambda}:= (\ev_{\lambda}\otimes Id^{\otimes 2i-2})\circ \tau_{2,2i}.$$

$$\tcoev^i_{\lambda}:=\tau_{2,2i} \circ (\tcoev_{\lambda}\otimes Id^{\otimes 2i-2}).$$

We notice that $\ev^1_{\lambda}=\ev_{\lambda}$ and $\tcoev^1_{\lambda}=\tcoev_{\lambda}$.
\end{notation}
\begin{proposition}(The ADO invariant from a braid presentation)\\{\label{defADO}}
Let $L$ be an oriented link. Consider $\beta_n \in B_n $ such that $L=\hat{\beta_n}$.
Then, the ADO invariant of $L$ can be expressed as follows:
$$\Phi_{N}(L,\lambda)=\xi^{(N-1)\lambda w(\beta_n)}\cdot$$
$$ <\left(Id_{U^{\lambda}_N}\otimes \left( \ev^{n-1}_{\lambda} \circ....\circ \ev^{1}_{\lambda}  \right) \right) \circ \left( \varphi_n^{\xi,\lambda}(\beta_n\otimes \mathbb I_{n-1} \right) \circ \left( Id_{U^{\lambda}_N} \otimes \left( \tcoev^{n-1}_{\lambda} \circ....\circ \tcoev^{1}_{\lambda} \right) \right)>.$$ 
\end{proposition} 
Here, the term $\xi^{(N-1)\lambda w(\beta_n)}$ corresponds to a framing correction (see \cite{Ito2}). 
 
 This definition provides a sequence of invariants indexed by the natural numbers. It is known that the first term of this sequence is the original Alexander polynomial, while certain specialisations at natural parameters lead to the sequence of Kashaev invariants ${K_N(K)}_{N \in \N}$. 
 \begin{theorem}
 For any knot $K$ one has the following relations:
 $$1) \Phi_N(L,\lambda)=\Delta(K,t)|_{t=\xi^{-2\lambda}}.$$
where $\Delta(K,t)$ is the Alexander polynomial of the knot.
$$2)  \Phi_N(L,N-1)=K_N(K).$$
 \end{theorem} 
\section{Homological braid groups representations}\label{3}

In this part we will present a sequence of braid group representations constructed by Lawrence. We will see later that this homology groups form homological counterparts which correspond to the the highest weight spaces, together with the quantum representation of the braid group from section \ref{2}.

\subsection{Covering space of the configuration space in the punctured disc }

\

Let $n,m \in \N$ be two natural numbers. 
Let us denote by $$D_n:=\D^2 \setminus \{p_1,...,p_n \} $$ the $n$-punctured disc, where $\D^2\subseteq \C$ is the unit closed disk (with boundary) and $\{ p_1,...,p_n\} $ are $n$ distinct points in its interior, which are also on the real axis.  

Consider the configuration space of $m$ unordered points in $D_n$: $$C_{n,m}=Conf_m(\D_n)= (\D^m_n \setminus \Delta) / Sym_m$$
where $\Delta=\{x=(x_1,...,x_n)\in \D^m_n | \ \exists \ i,j \ \text { such that } x_i=x_j\}.$

In the sequel we will use the homology of a certain covering space associated to $C_{n,m}$. We will define it using a certain local system as follows.
\begin{remark}
Let $$\rho: \pi_1(C_{n,m}) \rightarrow H_1(C_{n,m})$$ be the abelianisation map. 
Then, for any $n \in \N$ and $m\in N$, $m \geqslant 2$ one has:
$$H_1(C_{n,m})\simeq \ \Z^{n} \ \oplus \ \Z$$ 
$$ \ \ \ \ \ \ \ \ \ \ \ \ \ \ \ \ \ \ \ \ \ \ \ \ \ \ \ \ \ \ \ \ \ \ \ \ \ \ <\rho(\sigma_i)> \ <\rho(\delta)> \ \ \ \ \ {i\in \{1,...,n\}}. $$
Here, $\sigma_i\in \pi_1(C_{n,m})$ is the loop in the configuration space with $n-1$ fixed components and the first one going on a loop in $D_n$ around the puncture $p_i$. 

The last component $\delta \in \pi_1(C_{n,m})$ is given by a loop in the configuration space given by $(n-2)$ constant points and the first two components making a circle, which swaps the two initial points, as in the picture.
\end{remark}

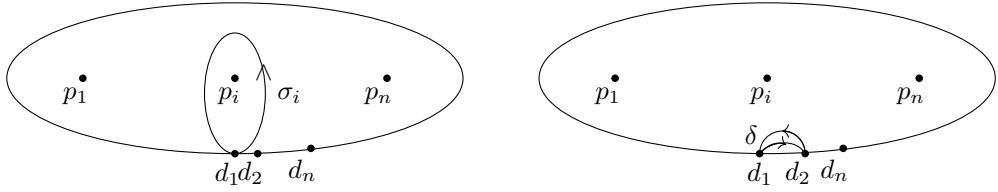
\begin{tikzpicture}
\foreach \x/\y in {0/2,2/2,4/2,2/1,2.3/1,3/1.07} {\node at (\x,\y) [circle,fill,inner sep=1pt] {};}
\node at (0.2,2) [anchor=north east] {$p_1$};
\node at (2.2,2) [anchor=north east] {$p_i$};
\node at (4.2,2) [anchor=north east] {$p_n$};
\node at (3,2) [anchor=north east] {$\sigma_i$};
\node at (2.2,1) [anchor=north east] {$d_1$};
\node at (2.5,1.02) [anchor=north east] {$d_2$};
\node at (3.2,1.05) [anchor=north east] {$d_n$};
\node at (2.62,2.3) [anchor=north east] {$\wedge$};
\draw (2,1.8) ellipse (0.4cm and 0.8cm);
\draw (2,2) ellipse (3cm and 1cm);
\foreach \x/\y in {7/2,9/2,11/2,8.9/1,9.5/1,10/1.07} {\node at (\x,\y) [circle,fill,inner sep=1pt] {};}
\node at (7.2,2) [anchor=north east] {$p_1$};
\node at (9.2,2) [anchor=north east] {$p_i$};
\node at (11.2,2) [anchor=north east] {$p_n$};
\node at (9.2,1) [anchor=north east] {$d_1$};
\node at (9.7,1.02) [anchor=north east] {$d_2$};
\node at (10.2,1.05) [anchor=north east] {$d_n$};
\node at (9,1.5) [anchor=north east] {$\delta$};
\draw (9,2) ellipse (3cm and 1cm);
\draw (9.5,1)  arc[radius = 3mm, start angle= 0, end angle= 180];
\draw [->](9.5,1)  arc[radius = 3mm, start angle= 0, end angle= 90];
\draw (8.9,1) to[out=50,in=120] (9.5,1);
\draw [->](8.9,1) to[out=50,in=160] (9.25,1.12);
\end{tikzpicture}
\begin{definition}(Covering space)\\
Consider the function 
$\epsilon: \Z^{n} \ \oplus \ \Z \rightarrow \ \Z \ \oplus \ \Z \ $

$ \ \ \ \ \ \ \ \ \ \ \ \ \ \ \ \ \ \ \ \ \ \ \ \ \ \ \ \ \ \ \ \ \ \ \ \ \ \ \ \ \  <x> <d>$\\
given by the augmentation of the first $n$ components  
$$\epsilon((x_1,...,x_n),y)=(x_1+...+x_n,y).$$
We consider the following local system 
$\phi: \pi_1(C_{n,m}) \rightarrow \Z \ \oplus \ \Z$ given by:

$ \ \ \ \ \ \ \ \ \ \ \ \ \ \ \ \ \ \ \ \ \ \ \ \ \ \ \ \ \ \ \ \ \ \ \ \ \ \ \ \ \ \ \ \ \ \ \ \ \ \ \ \ \ \ \ \ \ \ \ \ \  <x> <d>$
$$\phi= \epsilon \circ \rho.$$ \label{R:ll}
Let $\tilde{C}_{n,m}$ be the covering space of $C_{n,m}$ which corresponds to $Ker(\phi)$, and denote by $\pi: \tilde{C}_{n,m} \rightarrow C_{n,m}$ the projection map associated to it.
\end{definition}

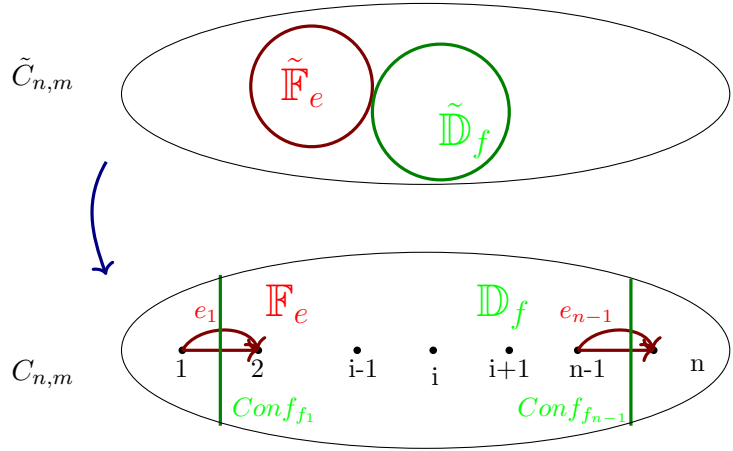
\begin{center}
\begin{tikzpicture}
\foreach \x/\y in {-1.2/2, -0.2/2 , 1.1/2 , 4/2 , 2.1/2 , 3.1/2 , 5/2 } {\node at (\x,\y) [circle,fill,inner sep=1pt] {};}
\node at (-1,2) [anchor=north east] {$1$};
\node at (0,2) [anchor=north east] {$2$};
\node at (1.5,2) [anchor=north east] {i-1};
\node at (2.3,1.9) [anchor=north east] {i};
\node at (3.5,2) [anchor=north east] {i+1};
\node at (4.5,2) [anchor=north east] {n-1};
\node at (5.8,2) [anchor=north east] {n};
\node at (-0.6,2.7) [anchor=north east] {$\color{red}e_1$};
\node at (4.6,2.7) [anchor=north east] {$\color{red}e_{n-1}$};
\node at (0.7,1.5) [anchor=north east] {\color{green}$Conf_{f_1}$};
\node at (4.8,1.5) [anchor=north east] {$\color{green}Conf_{f_{n-1}}$};
\node at (0.6,3) [anchor=north east] {\huge{\color{red}$\F_e$}};
\node at (3.5,3) [anchor=north east] {\huge{\color{green}$\D_f$}};
\node at (0.8,6) [anchor=north east] {\huge{\color{red}$\tilde{\F}_e$}};
\node at (3,5.4) [anchor=north east] {\huge{\color{green}$\tilde{\D}_f$}};
\node at (-2.5,2) [anchor=north east] {\large{$C_{n,m}$}};
\node at (-2.5,6) [anchor=north east] {\large{$\tilde{C}_{n,m}$}};

\draw [very thick,black!50!red,->](-1.2,2) to[out=60,in=110] (-0.2,2);
\draw [very thick,black!50!red,->](-1.2,2) to (-0.2,2);
\draw [very thick,black!50!red,->](4,2) to[out=60,in=110] (5,2);
\draw [very thick,black!50!red,->](4,2) to (5,2);
\draw [very thick,black!50!green,-](-0.7,1) to (-0.7,3);
\draw [very thick,black!50!blue,->](-2.2,4.5) to[out=-120,in=110] (-2.2,3);
 \draw[very thick,black!50!red] (0.5, 5.5) circle (0.8);
\draw [very thick,black!50!green,-](4.7,1) to (4.7,2.95);
 \draw[very thick,black!50!green] (2.2, 5.16) circle (0.9);

\
\draw (2,2) ellipse (4cm and 1.3cm);
\draw (2,5.4) ellipse (4cm and 1.2cm);

\end{tikzpicture}
\end{center}

\subsection{Braid group action}  
In this part, we see that this homology of the covering space of the configuration space in the puncture disc has the feature of carrying a braid group action. We remind that the braid group is the mapping class group of the punctured disc relative to its boundary:
$$B_n= MCG(D_n)=Homeo^+(D_n, \partial)/ \text{isotopy} \ \ \ $$
Then $B_n$ will act onto the configuration space $C_{n,m}$ by homeomorphisms and it will induce an action on its fundamental group $$B_n \curvearrowright \pi_1(C_{n,m}).$$ 
\begin{proposition}
This braid group action behaves well with respect to the local system $\phi$, and it can be lifted to an action onto the homology of the convering. Moreover, this action is compatible with the action of the deck transformations and one has that:
$$B_n \curvearrowright H^{\text{lf}}_m(\tilde{C}_{n,m}, \Z) \ (\text{ as a module over }\Z[x^{\pm}, d^{\pm}]).$$
\end{proposition}

\subsection{Lawrence representation}
The Lawrence representation will be defined as a certain subspace of the middle dimensional Borel-Moore homology of the covering $\tilde{C}_{n,m}$ described above.
We discuss this in the next part.
We start by reminding the following notations:
$$E_{n,m}=\{e=(e_1,...,e_{n-1})\in \N^n \mid e_1+...+e_{n-1}=m \}.$$
\begin{definition}({\bf Multiforks} \cite{Ito2}, \cite{Law1})\\
Let $d_1,...,d_m \in  \partial D_n$ which gives a base point $d=\{d_1,...,d_n\}$ in $C_{n,m}$.

{\bf{1) Submanifolds}}

For each partition $e \in E_{n,m}$, we will associate a submanifold in $\tilde{C}_{n,m}$ of dimension $m$, which will lead to a homology class in $H^{\text{lf}}_m(\tilde{C}_{n,m}, \Z)$. 
Let $i \in \{ 1,...,n-1 \} $, and consider $e_i$-horizontal segments in $D_n$, between the punctures $p_i$ and $p_{i+1}$, which are disjoint except meet at their boundaries. Denote this set of segments $$\{ I^{e}_1,...,I^{e}_{e_1},..., I^{e}_{m}\}.$$ 
$$I^{e}_{k}: (0,1)\rightarrow D_n, \forall k \in \{1,...,m\}. $$
Then, for $k \in \{ 1,...,m \} $, let us consider ${\gamma}^{e}_k$ to be a vertical path  between $d_k$ and a point from the segment $I^{e}_k$.
Since these segments are disjoint, then one can consider their product seen in the product of the disc $D_n$ minus the diagonal :
$$I^{e}_1\times...\times I^{e}_m: (0,1)^m\rightarrow D^m_n\setminus \Delta$$
Composing with the projection map 
$$\pi_m : D^m_n \setminus \{x=(x_1,...,x_n)| x_i=x_j\}) \rightarrow C_{n,m}$$
one gets a $m$-dimensional open ball embedded into the configuration space:
$$\F_{e}: \D^m(=(0,1)^m)\rightarrow C_{n,m}$$
$$\F_e=\pi_m \circ (I^{e}_1 \times ... \times I^{e}_{m}).$$

{\bf 2) Lifts to the covering space}\\
In the sequel we will see how to lift these open balls towards submanifolds in $\tilde{C}_{n,m}$. In order to do this, we use the paths $ \{ \gamma^e_k \}_{k\in \overline{1,m}}$ in order to construct certain base points.\\ 
Let us fix $\tilde{ \bf d} \in \pi^{-1} ({ \bf d}) $ a point in the covering space.
The union of the paths $\gamma^e_k$, from the segments to the base point $d$, give path in the configuration space.
Let $$\gamma^e := \pi_m \circ (\gamma^{e}_1, ..., \gamma^{e}_m ) : [0,1] \rightarrow C_{n,m}$$
Let us define $$\tilde{\gamma}^e:  [0,1]^m \rightarrow \tilde{C}_{n,m} $$  
to be the unique lift of $\gamma^e$ such that $\tilde{\gamma}^e(0)=\tilde{\bf d}$.

{\bf 3) Standard multiforks}\\ \label{f}
Let us consider $$\tilde{\F}_{e}: \D^m(=(0,1)^m)\rightarrow \tilde{C}_{n,m}$$ 
to be the uniqe lift of the ball $\F_e$ such that $\tilde{\gamma}^e(1) \in \tilde{\F}_{e}.$
This submanifold leads to a class in the Borel-Moore homology
$$[\tilde{\F}_{e}] \in H^{\text{lf}}_m(\tilde{C}_{n,m}, \Z)$$ called the multifork associated to the partition $e \in E_{n,m}$.
\end{definition}

\begin{definition}
Let us denote the following set:
$$\mathscr B_{\HH_{n,m}}:= \{  [\tilde{\F}_e] \ | \  e\in E_{n,m} .\}$$
Consider the subspace in the Borel-Moore homology generated by multiforks:
$$\HH_{n,m}:= <[\tilde{\F}_e] \ | \  e\in E_{n,m}>_{\Z[x^{\pm},d^{\pm}]} \subseteq H^{\text{lf}}_m(\tilde{C}_{n,m}, \Z)$$
It is known that the set $\mathscr B_{\HH_{n,m}}\subseteq H^{\text{lf}}_m(\tilde{C}_{n,m}, \Z)$ is linearly independent, and it gives a basis for $\HH_{n,m}$.
\end{definition}
\begin{proposition}(Lawrence representation)

The braid group action $B_n \curvearrowright H^{\text{lf}}_m(\tilde{C}_{n,m}, \Z) $ preserves the subspace $\HH_{n,m}$ and the basis $\mathscr B_{\HH_{n,m}}$ leads to a braid group representation:
$$l_{n,m}: B_n\rightarrow Aut(\HH_{n,m}, \Z[x^{\pm 1}, d^{\pm 1}])$$
called Lawrence representation. 
\end{proposition}
\subsection{Truncated Lawrence representation} \label{trunc}
On the quantum side, we have seen that the ADO invariant can be described through the representation of the braid group at roots of unity $\varphi_n^{\lambda}$, as in the equation \ref{defADO}. We will see in this section, that the Lawrence representation $\HH_{n,m}$ is a topological counterpart for the quantum representation $W^{q,\lambda}_{n,m}$, for generic $q,\lambda$. We will be interested in a specialisation of the quantum representation at $q$ root of unity $W^{\xi,\lambda}_{n,m}$. For this, we will use a result that has been shown in \cite{Ito2}, namely that a certain quotient of the Lawrence representation is a topological counterpart for the quantum representation of the braid group, at roots of unity . We will discuss this in details in the next part.

\begin{definition}
In the sequel, we use the following notation:
$$E^{N}_{n,m}=\{e \in E_{n,m} \mid e_i\leq N-1, \ \forall i \in \overline{1,n-1} \ \}.$$
$$E^{\geq N}_{n,m}=\{e \in E_{n,m} \mid \ \exists \ i \in \overline{1,n-1}, e_i> N-1\ \}.$$
 
Consider the subspace in the Borel-Moore homology of the covering space generated by the following elements:
$$\HH^{\geq N}_{n,m}:=<[\tilde{\F}_e] \ | \  e\in E^{\geq N}_{n,m}]>_{\Z[x^{\pm}, d^{\pm}]}\subseteq \HH_{n,m}$$
\end{definition}
We are interested in certain specialisations of the Lawrence representation.
\begin{definition} \label{coeffmorph}
For $\lambda, q \in \C$, let us consider the following ring morphism: 
$$\psi_{q,\lambda}: \Z[x^{\pm},d^{\pm}]\rightarrow \C$$ 
$$\psi_{q,\lambda}(x)= \qs^{-2 \lambda}; \ \ \psi_{\lambda}(d)=-\qs^{2}$$
\end{definition} 

\begin{proposition}(\cite{Ito2})
Let $N \in \N$ and $\xi=e^{\frac{2\pi i}{2N}}$. Let $\lambda \in \C$.
One has that:
$$\HH^{\geq N}_{n,m}|_{\psi_{\xi,\lambda}}\subseteq\HH_{n,m}|_{\psi_{\xi,\lambda}}$$
Then, the braid group action $l_{n,m}|_{\psi_{\xi,\lambda}}:B_n\rightarrow Aut(\HH_{n,m}|_{\psi_{\xi,\lambda}},\C)$ preserves the subspace $\HH^{\geq N}_{n,m}|_{\psi_{\xi,\lambda}}$.
\end{proposition}
\begin{definition}(Truncated Lawrence representation \cite{Ito2})
Let us consider the following quotient space of $\Z[x^{\pm1},d^{\pm1}]$-modules:
$$\HH^{N}_{n,m}:=\HH_{n,m}/ \HH^{\geq N}_{n,m}.$$
Then, the braid group action $l_{n,m}$, leads to a well defined braid group action on the quotient, when specialised at roots of unity using $\psi_{\xi, \lambda}$:
$$l^{N}_{n,m}=l_{n,m}|_{\psi_{\xi,\lambda}}: B_n \rightarrow Aut\left( \HH^{N}_{n,m}|_{\psi_{\xi,\lambda}},\C \right)$$
called truncated Lawrence representation.
\end{definition}
\section{Version of Lawrence representation}\label{3'}

So far, we have seen a good topological counterpart, which corresponds to highest weight space representations at roots of unity. We remind our purpose which is to describe the ADO invariant in a topological manner. The next step is to relate this topological counterpart with a dual homology of the covering space and pair them in a topological intersection form. 

Here there is a subtle point that we would like to emphasise. In \cite{Cr}, it was shown a topological model for the coloured Jones polynomials, using topological intersection pairings. More precisely, we used the Lawrence representation from section \ref{3} on one side and a "dual" Lawrence representation on the other, defined using certain homology classes called bar codes defined in the middle dimensional homology of the covering space $\tilde{C}_{n,m}$ relative to its boundary. 
Then, these two spaces were related by a duality which we showed to be non-degenerate over the polynomial ring $\Z[x^{\pm},d^{d^{\pm}}]$. The matrix of this form is diagonal with non-trivial polynomials in the variable $d$ on its diagonal. 

The specialisation $\psi_{q,\lambda}$ with $q$ variable and $\lambda\in \N$ would preserve the non-degeneracy of this intersection form, and this lead to the topological model for the coloured Jones polynomial $J_N(L,q)$.
Regarding this specialisation phenomenon, for the ADO case we have a small issue. We want to specialise using $\psi_{\xi, \lambda}$ at a root of unity $q=\xi=e^{\frac{2 \pi i}{2N}}$. This will corespond to the specialisation where $d=-\xi^{2}$. The non-degeneracy of the specialisation of the topological intersection pairing would be a question in this case. 

Then, our strategy will be to use a slightly different version of the Lawrence representation, that we introduce in this section. The idea is to split the infinity of the Borel-Moore part into two sets- the first part, which contains the points at infinity which approach a puncture, whereas the second part contains points at infinity that approach the diagonal $\Delta$ in the configuration space, in other words which are close to collide. Then, we use the homology relative to the first part of the infinity, responsible for the punctures and define the Lawrence representation as a certain subspace of this homology space. The second step is to define a dual version of it, using a certain subspace in the homology relative to the boundary and to the second part of infinity, which is responsible for collisions. This subspace will be generated by classes given by submanifolds called scan codes. 

The advantage will be that the topological pairing over the ring in two variables $\Z[x^{\pm},d^{\pm}]$ will be diagonal and its matrix will be the identity. In other words, we will have no problem concerning the degeneracy of this intersection form specialised at any parameters.
\subsection{Special Lawrence representation}
We will begin with the definition of the version of the Lawrence representation that we will use.
Consider as above the covering space $\tilde{C}_{n,m}$ of the configuration space in the punctured disc. In \cite{CM}, it was discussed in details this version of splitting the infinity part from the Borel-Moore homology into two separate parts. We will keep the same notation and consider 
$$H^{lf,\infty}_{m}(\tilde{C}_{n,m},\Z)$$ to be the homology relative to the infinity that encodes the boundary with punctures from the configuration space.  
\begin{definition}(Special multiforks)

Let $e \in E_{n,m}$ and consider the geometric submanifold as in \ref{f}:
$$\tilde{\F}_e\subseteq \tilde{C}_{n,m} $$
Then, it will lead to a well defined homology class:
$$[[\tilde{\F}_e]] \in H^{lf,\infty}_{m}(\tilde{C}_{n,m},\Z).$$ 
\end{definition}
\begin{proposition}(\cite{CM})\label{P:20}
Consider the following set:
$$\mathscr B_{\HH^{\infty}_{n,m}}:= \{  [\tilde{\F}_e] \ | \  e\in E_{n,m} .\}$$
Let the following subspace in the Borel-Moore homology related to the punctures, generated by special multiforks:
$$\HH^{\infty}_{n,m}:= <[\tilde{\F}_e] \ | \  e\in E_{n,m}>_{\Z[x^{\pm},d^{\pm}]} \subseteq H^{\text{lf},\infty}_m(\tilde{C}_{n,m}, \Z).$$
Then, the set $\mathscr B_{\HH^{\infty}_{n,m}}\subseteq H^{\text{lf},\infty}_m(\tilde{C}_{n,m}, \Z)$ is a basis for the homology space $\HH^{\infty}_{n,m}$.
\end{proposition}
Concerning the braid group action, this behaves well with respect to this version of homology. 
\begin{proposition}(Special Lawrence representation)\\
The braid group will induce a well defined action on the version of Borel-Moore homology relative to infinity, which leads to a representation using the basis $\mathscr B_{\HH^{\infty}_{n,m}}$:
$$l^{\infty}_{n,m}:B_n\rightarrow Aut(\HH^{\infty}_{n,m},\Z[x^{\pm 1}, d^{\pm 1}]).$$
We call this action special Lawrence representation.
\end{proposition}
In \cite{CM} it has been discussed the relationship between the two previous versions of the Lawrence representation, and it has been shown the following connection.
\begin{lemma}
Let us consider the map induced by the inclusion at the level of Borel-Moore homology:
$$\iota: H^{lf,\infty}_{m}(\tilde{C}_{n,m}, \Z)\rightarrow H^{lf}_{m}(\tilde{C}_{n,m}, \Z).$$
Then one has that $\iota$ is an injective function which sends a special multifork in its correspondent multifork:
$$\iota([[\tilde{\F}_e]])=[\tilde{\F}_e].$$
Moreover, the inclusion $\iota$ is equivariant with respect to the braid group actions
$l^{\infty}_{n,m}$ and $l_{n,m}$.
\end{lemma}  
\subsection{Truncated special Lawrence representation}
Regarding the discussion from \ref{trunc},  we have in mind quantum representations at roots of unity, which, as we have seen, correspond to the truncated versions of the Lawrence representation. For this purpose, we define truncated versions of the special Lawrence representation as follows.
\begin{definition}(Truncated special Lawrence representation)\\ \label{D:20}
Let us consider the following subspace in the special Lawrence representation:
$$\HH^{\infty, \geq N}_{n,m}:=<[[\tilde{\F}_e]] \ | \  e\in E^{\geq N}_{n,m}]>_{\Z[x^{\pm}, d^{\pm}]}\subseteq \HH^{\infty}_{n,m}.$$
\end{definition}
For the next part, we discuss the specialisation of this representation using the function $\psi_{q,\lambda}$. 
\begin{remark}
Let $q,\lambda \in \C$ and $\psi_{q,\lambda}:\Z[x^{\pm 1},d^{\pm 1}]\rightarrow \C$ the morphism of rings of coefficients from the definition \ref{coeffmorph}.
Consider the map induced at the level of specialised Lawrence prepresentations:
$$\iota|_{\psi_{q,\lambda}}:\HH^{\infty}_{n,m}|_{\psi_{q,\lambda}}\rightarrow \HH_{n,m}|_{\psi_{q,\lambda}}.$$
Then $\iota|_{\psi_{q,\lambda}}$ is injective.
\end{remark}
\begin{lemma}
Let $N \in \N$ and $\xi=e^{\frac{2\pi i}{2N}}$. Let $\lambda \in \C$.
One has that:
$$\HH^{\infty, \geq N}_{n,m}|_{\psi_{\xi,\lambda}}\subseteq\HH^{\infty}_{n,m}|_{\psi_{\xi,\lambda}}$$
Then, the special Lawrence representation $$l_{n,m}|_{\psi_{\xi,\lambda}}:B_n\rightarrow Aut(\HH^{\infty}_{n,m}|_{\psi_{\xi,\lambda}}\C)$$ preserves the subspace $\HH^{\infty,\geq N}_{n,m}|_{\psi_{\xi,\lambda}}$.
\end{lemma}
\begin{proof}
This follows from the injectivity of $\iota|_{\psi_{\xi,\lambda}}$ and the invariance of the subspace $\HH^{\geq N}_{n,m}|_{\psi_{\xi,\lambda}}$ through the braid group action.
\begin{center}
\begin{tikzpicture}
[x=1.2mm,y=1.4mm]

\node (b1)  [color=blue]             at (0,0)    {$\HH^{\infty,\geq N}_{n,m}|_{\psi_{\xi,\lambda}}$};
\node (b2) [color=black] at (0,10)   {$\HH^{\infty}_{n,m}|_{\psi_{\xi,\lambda}}$};
\node (b3)  [color=orange]             at (-4,16.5)   {\large $B_n$};
\node (t1) [color=blue] at (30,0)  {$\HH^{\geq N}_{n,m}|_{\psi_{\xi,\lambda}}$};
\node (t2)  [color=black]             at (30,10)    {$\HH_{n,m}|_{\psi_{\xi,\lambda}}$};
\node (t3) [color=orange] at (28,16.5)   {\large $B_n$};
\node (c1) [color=black] at (-2,13)   {$\curvearrowright$};
\node (c2) [color=black] at (30,13)   {$\curvearrowright$};
\node (d1) [color=orange] at (5,15)   {$l^{\infty}_{n,m}|_{\psi_{N-1}}$};
\node (d2) [color=orange] at (37,15)   {$l_{n,m}|_{\psi_{N-1}}$};

\draw[->,color=black]             (b2)      to node[left,xshift=6mm,yshift=3mm, font=\small]{$ \iota|_{\psi_{\xi,\lambda}} $}                           (t2);
\draw[->,color=blue] (b1)      to node[left,xshift=6mm,yshift=3mm, font=\small]{$ \iota|_{\psi_{\xi,\lambda}} $}                           (t1);
\draw[->,color=blue]   (b1)      to node[right,font=\small]{$$}                           (b2);
\draw[->,color=blue]  (t1)      to node[right,font=\small]{$$}                           (t2);

\

\end{tikzpicture}
\end{center}
\end{proof}
This shows that once we specialise at roots of unity, the braid group action onto the quotient through this subspace is well defined. More precisely, we obtain the following braid group action at roots of unity.
\begin{definition} (Truncated specialised Lawrence representation)\\
Let us define the following quotient over the ring in two variables $\Z[x^{\pm 1}, d^{\pm 1}]$:
$$\HH^{\infty,N}_{n,m}:=\HH^{\infty}_{n,m}/\HH^{\infty,\geq N}_{n,m}$$
Then, action of the braid group specialised at roots of unity $l_{n,m}|_{\psi_{\xi, \lambda}}$ leads to a well defined braid group action onto the quotient:
$$l^{\infty,N}_{n,m}=l_{n,m}|_{\psi_{\xi,\lambda}}: B_n \rightarrow Aut\left( \HH^{\infty,N}_{n,m}|_{\psi_{\xi,\lambda}},\C \right).$$
We call this action the truncated special Lawrence representation.
\end{definition}
In the next part we will see that actually at the level of the braid group action, the two versions of the Lawrence representation are isomorphic. The advantage of this special version of the Lawrence representation will be discussed regarding the yopological intersection pairing from section \ref{5}.
\begin{theorem}(\cite{CM})
The special Lawrence representation is isomorphic to the Lawrence representation over $\Z[x^{\pm},d^{\pm}]$, by an isomorphism which makes the correspondence between special multiforks and the usual multiforks:\\

$ \ \ \ \ \ \ \ \ \ \ \ \ \ \ \ \ \ \ \ \ \ \ \ \ \ \ \ \ \ \ \ \ \ \ \ l^{\infty}_{n,m} \ \ \ \ \ \ \ \ \ \ \ \ \ \ \ \ l_{n,m}$
$$B_n \curvearrowright \HH^{\infty}_{n,m}\rightarrow_{f} \HH_{n,m} \curvearrowleft B_n$$
$$f([[\tilde{\F}_e]])=[\tilde{\F}_e].$$
\end{theorem}
For $x \in \HH^{\infty}_{n,m}|_{\psi_{\xi,\lambda}}$, we denote its class in the quotient space $\HH^{\infty,N}_{n,m}|_{\psi_{\xi,\lambda}}$ by $\{\{  x \} \}$.
Similarly, for $y \in \HH_{n,m}|_{\psi_{\xi,\lambda}}$, we denote its class in the space $\HH^{N}_{n,m}|_{\psi_{\xi,\lambda}}$ by $\{  x \}$.
\begin{corollary} \label{truncident}
For any $N \in \N$, $\xi=e^{\frac{2 \pi i}{2N}}$ and $\lambda \in \C$, one has the following isomorphism of truncated braid group representations, specialised at roots of unity:\\

$ \ \ \ \ \ \ \ \ \ \ \ \ \ \ \ \ \ \ \ \ \ \ \ \ \ \ \ \ l^{\infty,N}_{n,m}|_{\psi_{\xi,\lambda}} \ \ \ \ \ \ \ \ \ \ \ \ \ \ \ \ \ l^N_{n,m}|_{\psi_{\xi,\lambda}}$
$$B_n \curvearrowright \HH^{\infty,N}_{n,m}|_{\psi_{\xi,\lambda}}\rightarrow_{f_N} \HH^N_{n,m}|_{\psi_{\xi,\lambda}} \curvearrowleft B_n$$
Where $f_N$ is defined on generators:
$$f_N\big( \{ \{ \  [[\tilde{\F}_e]] \ \} \} \big)=\{  [\tilde{\F}_e] \}.$$
\end{corollary}
\begin{proof}
Let us define $$f_N( \{ \{ x \} \} )=\{ f|_{\psi_{\xi,\lambda}}(x) \}, \forall x \in \HH^{\infty}_{n,m}|_{\psi_{\xi,\lambda}}.$$
Since $f$ is an isomorphism, it follows that $f|_{\psi_{\xi,\lambda}}$ is an isomorphism as well:
$$f|_{\psi_{\xi,\lambda}}:\HH^{\infty}_{n,m}|_{\psi_{\xi,\lambda}}\rightarrow \HH_{n,m}|_{\psi_{\xi,\lambda}}.$$ 
$$f|_{\psi_{\xi,\lambda}}\big( [[\tilde{\F}_e]] \big)=  [\tilde{\F}_e] .$$
We notice that it remains an isomorphism once we restrict it to the subspaces generated by multiforks with multiplicities more than $N$:
$$f|_{\psi_{\xi,\lambda}}:\HH^{\infty, \geq N}_{n,m}|_{\psi_{\xi,\lambda}}\rightarrow \HH^{\geq N}_{n,m}|_{\psi_{\xi,\lambda}}.$$ 
It follows that $f_N$ is a well defined isomorphism between the quotient spaces.
\end{proof}

\section{Quantum and Homological braid group representations}\label{4}

\begin{theorem}(Kohno's Theorem)\\
If the parameters $(q,\lambda)\in \C \times \C$ are generic, then the following braid group actions are isomorphic:

$ \ \ \ \ \ \ \ \ \ \ \ \ \ \ \ \ \ \ \ \ \ \ \ \ \ \ \ \ {^{W}\varphi}^{q,\lambda}_{n,m} \ \ \ \ \ \ \ \ \ \ \ \ \ \ \ \ \ l_{n,m}|_{\psi_{\xi,\lambda}}$
$$B_n \curvearrowright W^{q,\lambda}_{n,m}\leftarrow_{\Theta_{q,\lambda}} \HH_{n,m} |_{\psi_{q,\lambda}} \curvearrowleft B_n$$
$$\mathscr B_{W^{q,\lambda}_{n,m}} \ \ \ \ \ \mathscr B_{\HH_{n,m}}|_{\psi_{q,\lambda}}$$
where the isomorphism is defined on the multifork basis by:
$$\Theta_{q,\lambda}([\tilde{F}_e])=\phi_{q,\lambda}\big(\bar{v}_{\iota(e)}^{q,\lambda}\big).$$

\end{theorem}
We are interested in the case where $q$ is a root of unity, which is not a generic situation. In the sequel, we present an identification between quantum representation and the truncated Lawrence representation, which corresponds to this non-generic case.
\begin{theorem}(Ito \cite{Ito2})
Let $N \in \N$ and $\xi:=e^{\frac{2 \pi i}{2N}}$ a $2N^{th}$ root of unity.
Then, for any $\lambda \in \C$ generic with respect to $\xi$ one has the following isomorphic braid group representations:

$ \ \ \ \ \ \ \ \ \ \ \ \ \ \ \ \ \ \ \ \ \ \ \ \ \ \ \ \ {^{W}\varphi}^{\xi,\lambda}_{n,m} \ \ \ \ \ \ \ \ \ \ \ \ \ \ \ \ \ l^{N}_{n,m}|_{\psi_{\xi,\lambda}}$
$$B_n \curvearrowright W^{\xi,\lambda}_{n,m}\leftarrow_{\tilde{\Theta}_{\xi,\lambda}} \HH^{N}_{n,m} |_{\psi_{\xi,\lambda}} \curvearrowleft B_n$$
$$\mathscr B_{W^{\xi,\lambda}_{n,m}} \ \ \ \ \ \mathscr B_{\HH^{N}_{n,m}}|_{\psi_{\xi,\lambda}}$$
where the isomorphism is established by the formulas:
$$\tilde{\Theta}_{q,\lambda}(\{[\tilde{F}_e]\})=\phi_{\xi,\lambda}\big(\bar{v}_{\iota(e)}^{\xi,\lambda}\big).$$

\end{theorem}
Since we have seen in corollary \ref{truncident} that the special truncated Lawrence representation and the truncated Lawrence representation are actually isomorphic, we conclude following the identification between quantum and special Lawrence representations.
\begin{corollary}(Quantum and special homological braid group representations)\label{C:Ident}

For any $N\in N$, let $\xi=e^{\frac{2 \pi i}{2N}}$ and consider $\lambda \in \C$ generic with respect to $\xi$. Then, one has the following isomorphism of braid group actions:

$ \ \ \ \ \ \ \ \ \ \ \ \ \ \ \ \ \ \ \ \ \ \ \ \ \ \ \ \ {^{W}\varphi}^{\xi,\lambda}_{n,m} \ \ \ \ \ \ \ \ \ \ \ \ \ \ \ \ \ l^{\infty,N}_{n,m}|_{\psi_{\xi,\lambda}}$
$$B_n \curvearrowright W^{\xi,\lambda}_{n,m}\leftarrow_{\Theta^N_{\xi,\lambda}} \HH^{\infty,N}_{n,m} |_{\psi_{\xi,\lambda}} \curvearrowleft B_n.$$
$$\mathscr B_{W^{\xi,\lambda}_{n,m}} \ \ \ \ \ \mathscr B_{\HH^{\infty,N}_{n,m}}|_{\psi_{\xi,\lambda}}$$
where the correspondence is established by the function $$\Theta^N_{\xi,\lambda}=\tilde{\Theta}_{\xi,\lambda}\circ f^{-1}_N.$$
\end{corollary}
For the further purpose concerning the coloured Alexander invariants, we will be interested in highest weight spaces of the form $W_{2n-1,(n-1)(N-1)}$ for $n,N\in\N$.
\begin{notation}\label{N:2} 
Let $n,N \in \N$. We denote the identification function  between homological and quantum representations, which comes from corollary \ref{C:Ident} as follows:
$$B_{2n-1}\curvearrowright W^{\xi,\lambda}_{2n-1,(n-1)(N-1)}\leftarrow_{\Theta^{N,n}_{\xi,\lambda}} \HH^{\infty,N}_{2n-1,(n-1)(N-1)} |_{\psi_{\xi,\lambda}} \curvearrowleft B_{2n-1}$$
$$\mathscr B_{W^{\xi,\lambda}_{2n-1,(n-1)(N-1)}} \ \ \ \ \ \mathscr B_{\HH^{\infty,N}_{2n-1,(n-1)(N-1)}}|_{\psi_{\xi,\lambda}}$$

\end{notation}
\section{Dual Lawrence representation-Scans}\label{4'}

In this section we will define a dual space, using the homology of the covering related to the boundary and to a part of the infinity, which will correspond to the special Lawrence representation, through a topological pairing. 

Let us fix  $n,m \in \N$. As in \cite{CM}, we will use the homolgy of the convering space $C_{n,m}$ relative to its boundary and to the Borel-Moore part which corresponds to collisions of points in the configuration space (more precisely, the diagonal $\Delta$):
$$H^{lf, \Delta}_{m}(\tilde{C}_{n,m}, \partial; \Z).$$ 
For each element $f \in E_{n,m}$, we will associate a homology class in this homology relative to the boundary:
$$[[ \tilde{\D}_f]] \in H^{lf, \Delta}_{m}(\tilde{C}_{n,m}, \partial; \Z).$$
Let us fix $n-1$ vertical segments, as in the picture: 
$$J_1,...,J_{n-1}:[0,1]\rightarrow D_n$$
$ \ \ \ \ \ \ \ \ \ \ \ \ \ \ \ \ \ \ \ \ \ \ \ \ \ \ \ \ \  \ \ \ \ \ \ \ \ \ J_i \text{ vertical line between } p_i \text{ and } p_{i+1}, \ \ \ $

$ \ \ \ \ \ \ \ \ \ \ \ \ \ \ \ \ \ \ \ \ \ \ \ \ \ \ \ \ \ \ \ \ \ \  J_i(0),J_i(1) \in \partial D_n, \ \ \ \ \forall i \in \{1,..,n-1\}. $\\
For a natural  number $p \in \N$ and a segment $J: [0,1] \subseteq D_n$, we denote the unordered configuration space of $p$ points on the segment $J$ by:
$$Conf_{p}(J):=J\times...\times J/ Sym_{p} \subseteq Conf_{p}(D_n), $$
where $Sym_p$ is the action of the symmetric group of order $p$.
\begin{definition}(Scans)\label{D:21}

Let us fix a partition $f=(f_1,...,f_{n-1}) \in E_{n,m}$.\\ 
{ \bf 1) Submanifolds}
We consider the following submanifold in the configuration space:
$$\D_{f}:=Conf_{f_1}(J_1)\times...\times Conf_{f_{n-1}}(J_{n-1})/ Sym_m.$$ 
{ \bf 2) Base points}\\ 
For each $k \in \{1,...,m\}$, let us choose a path, as in the picture below 
$${\delta}^{f}_k:[0,1]\rightarrow D_n$$ 
$${\delta}^{f}_k(0)=d_k; \ \ \ {\delta}^{f}_k(1)\in J_{l} $$
for the unique $l \in \{1,...,n-1 \}$ with the property:
$$e_1+...+e_{l-1}<k\leq e_1+...+e_{l}.$$
Using these $m$ paths in the disc, we construct the following path in the configuration space:
$$\delta^f := \pi_m \circ (\delta^{f}_1, ..., \gamma^{f}_m ) : [0,1] \rightarrow C_{n,m}$$
Let us consider $\tilde{\delta}^f:  [0,1] \rightarrow \tilde{C}_{n,m} $ the unique lift of the path $\delta^f$ to the covering with the condition that: $$\tilde{\delta}^f(0)=\tilde{\bf d}.$$
{ \bf 3) Lifts to the covering-Scans}\\ 
We will use $\tilde{\delta}^f(1)$ as a base point for lifting submanifolds from the base space to the covering. Consider $\tilde{\D}_{f}$ to be the unique lift of $\D_{f}$ to the covering $\tilde{C}_{n,m}$ such that:
$$\tilde{\D}_{f}\cap \tilde{\delta}^f(1) \neq \emptyset.$$
Then, this submanifold will lead to a well defined homology class in the homology relative to the boundary and Borel-Moore with respect to the collisions:
$$[[\tilde{\D}_{f}]] \in H^{lf, \Delta}_{m}(\tilde{C}_{n,m}, \partial; \Z).$$
We call $[[\tilde{\D}_{f}]]$ the scan associated to the partition $f \in E_{n,m}.$
\end{definition}

\begin{definition}(Special dual Lawrence representation)\\
Consider the set of all scans, indexed by all possible partitions:
$$\mathscr B_{\HH^{\Delta,\partial }_{n,m}}:= \{  [[\tilde{\D}_f]] \ | \  f\in E_{n,m}\}.$$
Let us define the subspace in the homology relative to the boundary and Borel-Moore with respect to the collisions, generated by all scans:
$$\HH^{\Delta,\partial }_{n,m}:= <[[\tilde{\D}_f]] \ | \  f\in E_{n,m}>_{\Z[x^{\pm},d^{\pm}]} \subseteq H^{lf, \Delta}_{m}(\tilde{C}_{n,m}, \partial; \Z).$$
We call the $\Z[x^{\pm 1}, d^{\pm 1}]$-module $\HH^{\Delta,\partial }_{n,m}$  special dual Lawrence representation.
\end{definition}
\subsection{Non-degenerate pairing}
This part is devoted to the study of intersection pairings between homologies in complementary dimensions of covering spaces. This pairings were used by Bigelow\cite{Big} and \cite{Big2} mainly between the middle dimensional Borel Moore homology of the covering and the homology of the covering relative to its boundary.  Here, we will use a similar idea, but for a different splitting of the "boundary" and "boundary at infinity" of the covering $C_{n,m}$, namely:
$$H^{lf,\infty}_{m}(\tilde{C}_{n,m},\Z) \text{ and }  H^{lf, \Delta}_{m}(\tilde{C}_{n,m}, \partial; \Z)$$
 
Following \cite{CM} from a Lefschetz duality type techniques one can deduce the following result.
\begin{theorem}
There exist a non-degenerate intersection pairing between the two versions of the middle dimensional Borel-Moore homology of the covering space:
$$< , >: H^{lf,\infty}_{m}(\tilde{C}_{n,m},\Z) \otimes  H^{lf, \Delta}_{m}(\tilde{C}_{n,m}, \partial; \Z)\rightarrow \Z[x^{\pm}, d^{\pm}].$$
\end{theorem}
From \cite{Big} and \cite{Cr} (Section 4), one can compute explicitly this pairing, given two homology classes which are represented by lifts of certain embedded submanifolds.  
\begin{remark}\label{R:bases}
Consider a ring $R$.
Let $A, B$ be two $R$-modules such that they are generated by two sets:
$$A=<a_i \mid i \in \overline{1,N}>_{R} \ \ \ \ \ B=<b_j \mid j \in \overline{1,N}>_{R}$$ and a sesquilinear form
$$< , >A \times B \rightarrow R.$$ 
Suppose that using these basis, the form has the following matrix:
$$<a_i,b_j>=\delta_{i,j}.$$
Then $\{a_i\}_{i \in \overline{1,N}}$ is a basis for $A$, $\{b_j\}_{j \in \overline{1,N}}$ is a basis for $B$ and $< , >$ is non-degenerate.
\end{remark}
\begin{proof}
Let $\sum_{i=1}^N \alpha_i a_i=0 \in A$ for $\alpha_i \in R$. Then, pairing with a fixed element we get:
$$<\sum_{i=1}^N \alpha_i a_i,b_j>=0 \Rightarrow \sum_{i=1}^N \alpha_i <a_i,b_j>=0\Rightarrow a_j=0, \  \forall j \in \overline{1,N}.$$
So $\{a_i\}_{i \in \overline{1,N}}$ is a basis for $A$. Similarly for $B$, and so the form is non-degenerate.
\end{proof}
\begin{proposition}
This pairing restricted to the special Lawrence representation and its dual, leads to a intersection form which can be calculated using the graded intersection pairing in $C_{n,m}$. Then, this intersection form has the following values:
$$<[[\tilde{\F}_e]],[[\tilde{\D}_f]]>=\delta_{e,f}.$$
 \end{proposition}
Combining this computation with the remark \ref{R:bases}, we conclude the following result. 
\begin{corollary}(Intersection form)\\
The special multiforks and scans form two bases for the special and dual Lawrence representations:
$$\mathscr B_{H^{\infty}_{n,m}} \subseteq H^{\Delta,\infty}_{n,m} \text { and } \mathscr B_{H^{\Delta,\partial}_{n,m}} \subseteq H^{\Delta,\partial}_{n,m}.$$
There is a non-degenerate sesquilinear topological intersection pairing:
$$< , >: H^{\infty}_{n,m}\otimes  H^{\Delta,\partial}_{n,m}\rightarrow \Z[x^{\pm}, d^{\pm}]$$
whose matrix with respect to the bases $\mathscr B_{H^{\Delta,\partial}_{n,m}}$ and $\mathscr B_{H^{\Delta,\partial}_{n,m}}$ is the identity.
\end{corollary}
\section{Topological intersection pairing at roots of unity}\label{5}
\subsection{Truncated dual special Lawrence representation}
We are interested in a dual space which correspond to the truncated special Lawrence representation. We will defined a truncated version in the homology relative to the boundary as follows.

\begin{definition}(Truncated special dual Lawrence representation)\\
We define the following subspace inside the dual special Lawrence representation, generated by scans with one multiplicity at most $N$:
$$\HH^{\Delta,\partial, N}_{n,m}:=<[[\tilde{\D}_f]] \ | \  f\in E^{\geq N}_{n,m}]>_{\Z[x^{\pm}, d^{\pm}]}\subseteq \HH^{\Delta,\partial}_{n,m}.$$
We call this homology the truncated special dual Lawrence representation.
\end{definition}
\begin{remark}The topological pairing vanishes on the subspaces corresponding to the partitions from $E^{\geq N}_{n,m}$ and $E^{N}_{n,m}$:
$$< , >: H^{\infty,\geq N}_{n,m} \otimes  H^{\Delta,\partial,N}_{n,m}\rightarrow \Z[x^{\pm}, d^{\pm}]$$
\end{remark}
\begin{proposition}\label{P:1} 
There is a well defined topological pairing induced from the pairing $<, >$, on the truncated special Lawrence representations:
$$\ll , \gg: H^{\infty, N}_{n,m} \otimes  H^{\Delta,\partial,N}_{n,m}\rightarrow \Z[x^{\pm}, d^{\pm}]$$
$$\ll \{ \{ \ [[\tilde{\F}_e]] \} \} \ ,[[\tilde{\D}_f]]\gg=\delta_{e,f}, \ \forall e,f \in E^{N}_{n,m}.$$
\end{proposition}
\begin{corollary}(Non-degenerate pairing at roots of unity)\label{C:pairing}

The topological intersection form $\ll , \gg$ induces a non-degenerate intersection pairing, for any specialisation $\psi_{\xi,\lambda}$ corresponding to the variable $\lambda \in \C$:
$$\ll , \gg |_{\psi_{\xi,\lambda}}: H^{\infty, N}_{n,m}|_{\psi_{\xi,\lambda}} \otimes  H^{\Delta,\partial,N}_{n,m} |_{\psi_{\xi,\lambda}}\rightarrow \C.$$

\end{corollary}

\section{Topological model for the coloured Alexander invariants}\label{6}
In this part we aim to give a topological model for the ADO invariant. The strategy is to start with a link and present it as a closure of a braid. The coloured Alexander invariant of the link is then obtained using a Reshetikhin-Turaev type construction applied on the link diagram with the first strand that is cut. In order to pass towards the topological part, we split the diagram into three main levels and we investigate the functor applied to each of these parts. Let us make this precise.

\begin{theorem}({\bf Topological model for coloured Alexander invariants})\label{C1:pr2}
Let $N \in \mathbb N$ be the colour of the invariant. Let us consider $n$ to be a natural number. 

Then there exist two families of homology classes $$\mathscr F^{N,\lambda}_n \in H^{\infty,N}_{2n-1, (n-1)(N-1)}|_{\psi_{\xi,\lambda}}  \ \ \ \ \ \  \mathscr G^{N,\lambda}_n \in H^{\Delta,\partial,N}_{2n-1, (n-1)(N-1)}|_{\psi_{\xi,\lambda}}$$ such that if $L$ is a link and $\beta \in B_{n}$ with $L=\hat{\beta} $ (normal closure), the $N^{th}$ coloured Alexander invariant has the formula:
$$\Phi_N(L,\lambda)= \xi^{(N-1)\lambda w(\beta_n)} \ll (\beta_n \cup \mathbb I_{n-1})) \mathscr{F}^N_n , \mathscr{G}^N_n\gg|_{\psi_{\xi,\lambda}}, \ \ \ \forall \lambda \in \C \setminus \Z.$$ 
\end{theorem}
\begin{proof}
Let $\lambda \in \C \setminus \Z$. This means that $\lambda$ is generic with respect to $\xi$.
Consider $L$ be a link as in the statement and let us denote by $\hat{L}$ the $(1,1)$-tangle obtained from $L$ by cutting its first strand. In other words, $\hat{L}$ is the partial closure of the braid $\beta_n$, corresponding to the last $n-1$ strands.
We remind the definition of the ADO invariant, as presented in proposition \ref{defADO}:
\begin{equation}
\begin{aligned}
\Phi_{N}(L,\lambda)=\xi^{(N-1)\lambda w(\beta_n)}\cdot \ \ \ \ \ \ \ \ \ \ \ \ \ \ \ \  \ \ \ \ \ \ \ \ \ \ \ \ \ \ \ \ \ \ \ \ \ \ \ \ \ \ \ \ \ \ \ \ \ \ \\
<\left(Id_{U^{\lambda}_N}\otimes \left(\ev^{n-1}_{\lambda} \circ...\circ \ev^{1}_{\lambda}\right)\right) \circ \left( \varphi_{2n-1}^{\xi,\lambda}(\beta_n\otimes \mathbb I_{n-1} \right) \circ \\
\ \ \ \ \ \ \ \ \ \ \ \ \ \ \ \circ \left( Id_{U^{\lambda}_N} \otimes \left(\tcoev^{n-1}_{\lambda} \circ....\circ \tcoev^{1}_{\lambda}\right) \right)>.
\end{aligned}
\end{equation}

{{\bf Strategy}} \ \ We will interpret homologically each of the three morphisms that occur in this formula.
Our strategy is to show that this construction can be seen through a particular highest  weight space inside $U_N(\lambda)^{\otimes 2n-1}$. After that, we use the truncated special Lawrence representation as a topological counterpart for the braid. Let us make this precise.

In order to compute the scalar that comes from the normalised Reshetikhin-Turaev functor, we will evaluate the morphism corresponding to
$$\mathscr F_{\lambda}(\hat{L}) \in End_{U_{\xi}(sl(2))}(U_N(\lambda)).$$
For a reason that it will be motivated in the next step, it will be important to compute this scalar starting from the highest weight vector $v_0$.
\begin{notation}
Consider the projection of $U_N(\lambda)$ onto the subspace generated by $v_0$ defined as follows:
$$\pi_{v_0}: U_N(\lambda)\rightarrow\C$$
$$\pi_{v_0}(v_i)=\begin{cases}
1, \ \ \ \ \ \ \  i=0\\
0, \ \ \ \ \ \ \ \ otherwise.\\
\end{cases}$$
\end{notation}
With this notation, the ADO invariant has the following formula:
\begin{equation}\label{DEFINITION}
\begin{aligned}
\Phi_{N}(L,\lambda)=\xi^{(N-1)\lambda w(\beta_n)}\cdot
 \pi_{v_0} \circ \left(Id_{U^{\lambda}_N}\otimes \left(\ev^{n-1}_{\lambda} \circ...\circ \ev^{1}_{\lambda}\right)\right) \circ \ \ \ \ \ \ \ \ \ \ \ \ \ \\
\left( \varphi_{2n-1}^{\xi,\lambda}(\beta_n\otimes \mathbb I_{n-1} )\right) \circ \left( Id_{U^{\lambda}_N} \otimes \left(\tcoev^{n-1}_{\lambda} \circ....\circ \tcoev^{1}_{\lambda}\right) \right)(v_0). \ \ \ \ \ \ \ \ 
\end{aligned}
\end{equation}
We will start with the bottom part of the $(1,1)-$tangle, which corresponds to the first morphism from the previous formula:
\begin{equation}\label{eq:-1}
\begin{aligned}
\mathscr F_{\lambda} \left(\uparrow \tikz[x=1mm,y=1mm,baseline=0.5ex]{\draw[<-] (0,3) .. controls (0,0) and (3,0) .. (3,3); \draw[<-] (-3,3) .. controls (-3,-3) and (6,-3) .. (6,3); \draw[draw=none, use as bounding box](-0.5,0) rectangle (3,3);}\right)= Id_{U^{\lambda}_N} \otimes \left(\tcoev^{n-1}_{\lambda} \circ....\circ \tcoev^{1}_{\lambda}\right) \in \ \ \ \ \ \ \ \ \ \ \ \ \ \ \ \ \ \ \ \ \ \ \  \\
\in \textit{ \large Hom}_{U_{\xi}(sl(2))}\Big( \ U_N(\lambda) \ \text{\Large ,} \  U_N(\lambda)^{\otimes n}\otimes (U_N(\lambda)^{*})^{\otimes n-1} \Big).
\end{aligned}
\end{equation}
\subsection{{\bf Step 1-Normalising the coevaluation}}
\

 In this first technical step, we deal with the part of the diagram corresponding to the $n-1$ strands that are connected through caps, namely:
\begin{equation*}
\begin{aligned}
\mathscr F_{\lambda} \left( \tikz[x=1mm,y=1mm,baseline=0.5ex]{\draw[<-] (0,3) .. controls (0,0) and (3,0) .. (3,3); \draw[<-] (-3,3) .. controls (-3,-3) and (6,-3) .. (6,3); \draw[draw=none, use as bounding box](-0.5,0) rectangle (3,3);}\right)= \left(\tcoev^{n-1}_{\lambda} \circ....\circ \tcoev^{1}_{\lambda}\right) \in \ \ \ \ \ \ \ \ \ \ \ \ \ \ \ \ \ \ \ \ \ \ \ \ \ \ \ \ \ \ \ \ \ \ \ \ \ \ \ \ \\
\in \textit{ \large Hom}_{U_{\xi}(sl(2))} \Big( \ \C \ \text{\Large{,}} \ U_N(\lambda)^{\otimes n-1}\otimes (U_N(\lambda)^{*})^{\otimes n-1}\Big).
\end{aligned}
\end{equation*}
The interesting part is that this is a morphism of $U_{\xi}(sl(2))$-representations.
This means, that it commutes with the actions of the generators $E,K$ of the quantum group.
We denote the following vector by $$v^{n-1}_{Coev}:=\left(\tcoev^{n-1}_{\lambda} \circ....\circ \tcoev^{1}_{\lambda}\right)(1). $$
Using the formulas for the dualities of the quantum group given in \ref{E:DualityForCat}, we have:
\begin{equation}\label{eq:0}
v^{n-1}_{Coev}=\sum_{i_1,...,i_{n-1}=0}^{N-1} v_{i_1}\otimes ... \otimes v_{i_{N-1}} \otimes (v_{i_{N-1}})^* \otimes ... \otimes (v_{i_1})^*.
\end{equation}
Coming back to the properties of the Reshetikhin-Turaev construction, since $\C$ is the module with the trivial action of the quantum group, using the functoriality property we obtain:
$$K(v^{n-1}_{Coev})=v^{n-1}_{Coev}$$
$$E(v^{n-1}_{Coev})=0$$
We notice that $v^{n-1}_{Coev}$ looks like a vector in a particular highest weight space, but the issue come from the fact that this is not inside the tensor power of a fixed representation (it is in a tensor power of a certain representation and its dual ). This part deals with the subtlety that the representation $U_N^{\lambda}$ is not self-dual over the quantum group $U_{\xi}(sl(2))$. However, its dual has the same dimension and they are isomorphic as vector spaces over $\C$.

The idea in the sequel is to compose the co-evaluation with an additional morphism $$\psi^N_n:(U_N(\lambda)^{*})^{\otimes n-1}\rightarrow U_N(\lambda)^{\otimes n-1}$$ in such way so that:
$$Id_{U_N(\lambda)}^{\otimes n-1}\otimes \psi^N_n $$
sends the vector $v^{n-1}_{Coev}$ into a highest weight space inside $U_N(\lambda)^{\otimes 2n-1}$.
In other words, one requires the following conditions:
\begin{equation}\label{eq:1}
\begin{aligned}
K\curvearrowright \left( Id_{U_N(\lambda)}^{\otimes n-1}  \ \otimes \ \varphi^{N}_{n}\right)(v^{n-1}_{Coev}) \simeq \left( Id_{U_N(\lambda)}^{\otimes n-1}  \ \otimes \ \varphi^{N}_{n}\right)(v^{n-1}_{Coev})\\
E \curvearrowright \left( Id_{U_N(\lambda)}^{\otimes n-1}  \ \otimes \ \varphi^{N}_{n}\right)(v^{n-1}_{Coev})=0. \ \ \ \ \ \ \ \ \ \ \ \ \ \ \ \ \ \ \ \ \ \ \ \ \ \ \ \ \ \ \ \ 
\end{aligned}
\end{equation}
Here, we used the following convention.
\begin{notation}
If two vectors $v$ and $w$ are proportional, we denote it by: 
$$v \simeq w.$$
\end{notation}

We will define $\psi^N_n$ as a tensor power of morphisms corresponding to each coordinate.
To begin with, we construct a sequence of isomorphisms of $\C$-vector spaces:
$$\{ f_i: U_N(\lambda)^{\star}\rightarrow U_N(\lambda)\mid i\in \overline{1,n-1} \}$$ and then use their tensor product:
$$\psi^N_n=f_1\otimes...\otimes f_{n-1}.$$ 
\begin{remark*}
There is the following isomorphism of $U_{\xi}(sl(2))$-representations:
$$f_{\lambda}:U_N(\lambda)^{*}\rightarrow U_N(2N-2-\lambda)$$
$$f_{\lambda}(v_k^*) \simeq v_{N-1-k}.$$
\end{remark*}
We search $f_k: U_N(\lambda)^{\star}\rightarrow U_N(\lambda)$ of the following form:
\begin{equation}\label{eq:2}
f_k(v_i^*)= c^{i}_k \cdot  v_{N-1-i}.
\end{equation}
\begin{proposition} For any sequence of functions $\{f_k | k \in \overline{1,n-1}\}$ as in the equation \ref{eq:2}, one has that $\psi^N_n(v^{n-1}_{Coev})$ is a weight vector in $U_N(\lambda)^{\otimes n-1} $ for the operator $K$.
\end{proposition}
\begin{proof}
Let us consider such a sequence that leads to the function:
$$\psi^N_n=f_1\otimes...\otimes f_{n-1}.$$
From the equation \ref{eq:0}, we obtain that:
\begin{equation*}
\begin{aligned}
\left( Id_{U_N(\lambda)}^{\otimes n-1}  \ \otimes \ \psi^N_n\right)(v^{n-1}_{Coev})=\sum_{i_1,...,i_{n-1}=0}^{N-1} \left( \prod_{k=1}^{n-1}c_{k}^{i_k} \right) \ \ \ \ \ \ \ \ \ \ \ \ \ \ \ \ \ \ \ \ \ \ \ \ \ \ \ \ \ \ \ \ \ \ \ \ \ \ \ \\
 \left( v_{i_1}\otimes ... \otimes v_{i_{n-1}} \otimes v_{N-1-i_{n-1}} \otimes ... \otimes v_{N-1-i_1} \right).
\end{aligned}
\end{equation*}
Now we will compute the action of $K$ on this vector.
$$K\left( Id_{U_N(\lambda)}^{\otimes n-1}  \ \otimes \ \psi^N_n \right)(v^{n-1}_{Coev})=\sum_{i_1,...,i_{n-1}=0}^{N-1} \left( \prod_{k=1}^{n-1}c_{k}^{i_k} \right) q^{(n-1)(2\lambda-2(N-1))}  \ \ \ \ \ \ \ \ \ \ $$
$$ \ \ \ \ \ \ \ \ \ \ \ \ \ \ \ \ \ \ \ \ \ \ \ \ \ \ \ \ \ \ \ \ \ \ \ \ \ \cdot \left( v_{i_1}\otimes ... \otimes v_{i_{n-1}} \otimes v_{N-1-i_{n-1}} \otimes ... \otimes v_{N-1-i_1} \right)=$$
$$ \ \ \ \ \ \ \ \ \ \ \ \ \ \ \ \ \ \ \ \ \ \ \ \ \ \ \  \ \ \ \ \ \ \ \ =q^{2(n-1)\lambda-2(n-1)(N-1)} \left( Id_{U_N(\lambda)}^{\otimes n-1}  \ \otimes \ \psi^N_n \right)(v^{n-1}_{Coev}).$$
\end{proof}
This shows that the composition between the normalising function $ Id_{U_N(\lambda)}^{\otimes n-1}  \ \otimes \ \psi^N_n $ and the coevaluation leads indeed in a particular weight space:
$$\left( Id_{U_N(\lambda)}^{\otimes n-1}  \ \otimes \ \psi^N_n \right)(v^{n-1}_{Coev}) \in V^{\xi,\lambda}_{2(n-1),(n-1)(N-1)}\subseteq U_N(\lambda)^{\otimes 2(n-1)}.$$ 

\begin{lemma}
There exist a sequence of coefficients for the functions $f_k$ such that the vector $$\left( Id_{U_N(\lambda)}^{\otimes n-1}  \ \otimes \ \psi^N_n \right)(v^{n-1}_{Coev})$$ belongs to the corresponding highest weight space $$W^{\xi,\lambda}_{2(n-1),(n-1)(N-1)}\subseteq U_N(\lambda)^{\otimes 2(n-1)}.$$ 
\end{lemma}
\begin{proof}
We construct this sequence by induction on $n$. 

For $m \in \N$, we consider the induction hypothesis $P(m)$: there exist functions $f_1,...,f_m$ as in equation \ref{eq:2} such that:
$$E \curvearrowright \left( Id_{U_N(\lambda)}^{\otimes m}  \ \otimes \ ( f_1 \otimes ... \otimes f_{m}) \right)(v^{m}_{Coev})=0.$$
The operator $E$ acts on a tensor product using the iterated comultiplication:
$$\Delta^{2m}(E)=\sum_{j=1}^{2m}Id^{j-1}\otimes E \otimes K^{2m-j}.$$
This means that one requires:
\begin{equation}\label{eq:3}
\begin{aligned}
\left( \sum_{j=1}^{2m}Id^{j-1}\otimes E \otimes K^{m-j-1} \right) \curvearrowright \ \ \ \ \ \ \ \ \ \ \ \ \ \ \ \ \ \ \ \ \ \ \ \ \ \ \ \ \ \ \ \ \ \ \ \ \ \ \ \ \ \ \ \ \ \ \ \ \ \ \ \ \\
\left( \sum_{i_1,...,i_{m}=0}^{N-1} \left( \prod_{k=1}^{m}c_{k}^{i_k} \right) \left( v_{i_1}\otimes ... \otimes v_{i_{m}} \otimes v_{N-1-i_{m}} \otimes ... \otimes v_{N-1-i_1} \right) \right) =0.
\end{aligned}
\end{equation}

{ \bf Case m=1}
We search for $f_1: U_N(\lambda)^{\star}\rightarrow U_N(\lambda)$ of the form:
$$f_1(v_i^*)= c^{i}_1 \cdot  v_{N-1-i}$$ with the requirement 
\begin{equation}
\left( E \otimes K + Id \otimes E \right) \left( \sum_{i_1=0}^{N-1} c_1^{i_1} v_{i_1} \otimes v_{N-1-i_1} \right)  =0.
\end{equation}
This is equivalent to:
\begin{equation}
\begin{aligned}
\sum_{i_1=0}^{N-2} c_1^{i_1+1} [\lambda-i_1]_{\xi} q^{\lambda-2(N-2-i_1)} \ v_{i_1} \otimes v_{N-2-i_1}+\\
+ \sum_{i_1=0}^{N-1} c_1^{i_1} [\lambda+1-(N-1-i_1)]_{\xi} \ v_{i_1} \otimes v_{N-2-i_1}  =0.
\end{aligned}
\end{equation}
Which impose the equation:
\begin{equation}
\frac{c_1^{i_1+1}}{c_1^{i_1}}=-q^{-\lambda+2(N-2-i_1)}\frac{[\lambda+2-N+i_1]_{\xi}}{[\lambda-i_1]_{\xi}}.
\end{equation}
We fix $c_1^0=1$ and the two conditions together determine an unique sequence of coefficients $\{c_1^k\}_{k \in \overline{1,N-1}}$. This completes the first step.

{\bf General case: $P(m)\rightarrow P(m+1)$}
Suppose we have constructed a sequence of functions such that $P(m)$ is satisfied. We search for the function $f_{m+1}$ such that the following relation is satisfied (as in \ref{eq:3}):
\begin{equation}
\begin{aligned}
\left( \sum_{j=1}^{2m+2}Id^{j-1}\otimes E \otimes K^{m-j-1} \right) \curvearrowright \ \ \ \ \ \ \ \ \ \ \ \ \ \ \ \ \ \ \ \ \ \ \ \ \ \ \ \ \ \ \ \ \ \ \ \ \ \ \ \ \ \ \ \ \ \ \ \ \ \ \ \ \\
\left( \sum_{i_1,...,i_{m+1}=0}^{N-1} \left( \prod_{k=1}^{m+1}c_{k}^{i_k} \right) \left( v_{i_{m+1}}\otimes ... \otimes v_{i_1} \otimes v_{N-1-i_{1}} \otimes ... \otimes v_{N-1-i_{m+1}} \right) \right) =0.
\end{aligned}
\end{equation}
We will separate the above sum in the following manner: we will separate the part that corresponds to the action of $E$ into two terms, one which correspond to the first and the $2m+2$ strand, and the other term which contains the rest. For each of the two actions, we will split the sum corresponding to the vector coming from the coevaluation into two sums. Let us make this precise. 

Separating the two actions, the equation that we need becomes the following:
\begin{equation}\label{eq:4''}
\begin{aligned}
\left( \left( E \otimes K^{\otimes 2m+1}+Id^{\otimes 2m+1} \otimes E \right) +  Id \otimes \left( \sum_{j=1}^{m}Id^{j-1}\otimes E \otimes K^{m-j-1} \right)\otimes K \right) \curvearrowright \\
\left( \sum_{i_1,...,i_{m+1}=0}^{N-1} \left( \prod_{k=1}^{m+1}c_{k}^{i_k} \right) \left( v_{i_{m+1}}\otimes ... \otimes v_{i_1} \otimes v_{N-1-i_{1}} \otimes ... \otimes v_{N-1-i_{m+1}} \right) \right) =0.
\end{aligned}
\end{equation}
We consider the following vectors:
\begin{align*}
C_m^{i_1,...,i_m}:=\prod_{k=1}^{m}c_{k}^{i_k} \ \ \ \ \ \ \ \ \ \ \ \ \ \ \ \ \ \ \ \ \ \ \ \\
v_{2m}^{i_1,...,i_m}:=v_{i_{m}}\otimes ... \otimes v_{i_1} \otimes v_{N-1-i_{1}} \otimes ... \otimes v_{N-1-i_{m}}.
\end{align*}
With this notation and the formula for the comultiplication of $E$, the equation \ref{eq:4''} becomes:
 \begin{equation}\label{eq:4}
\begin{aligned}
\left( \left( E \otimes K^{\otimes 2m+1}+Id^{\otimes 2m+1} \otimes E \right) + \left(  Id \otimes \Delta^m(E) \otimes K \right) \right) \curvearrowright  \ \ \ \ \ \ \ \ \ \ \ \ \\
\left( \sum_{i_1,...,i_{m+1}=0}^{N-1} c_{m+1}^{i_{m+1}} C_{m}^{i_1,...,i_{m}} \left( v_{i_{m+1}}\otimes v_{2m}^{i_1,...,i_{m}}\otimes v_{N-1-i_{m+1}} \right) \right) =0.
\end{aligned}
\end{equation}
\begin{remark}\label{R:10}
We notice that the generator  $K$ acts in the following manner:
$$K^{\otimes 2m}\curvearrowright(v_{2m}^{i_1,...,i_m})=q^{2m \lambda -2(i_m+...+i_{1}+(N-1-i_1)+...+(N-1-i_m))} v_{2m}^{i_1,...,i_m}=$$
$$=q^{2m\lambda-2m(N-1)}v_{2m}^{i_1,...,i_m}.$$
\end{remark}
With these notations, $P(m)$ is equivalent to:
$$\Delta^m(E)(Id^{\otimes m}\otimes f_1 \otimes ...\otimes f_m )(v^{m}_{Coev})=0.$$
This can be written as follows:
\begin{equation}\label{eq:5}
\Delta^m(E) \left( \sum_{i_1,...,i_{m}=0}^{N-1}  C_{m}^{i_1,...,i_{m}} \cdot  v_{2m}^{i_1,...,i_{m}}  \right)=0 
\end{equation}
Then, equation \ref{eq:4} becomes:
 \begin{equation}\label{eq:6}
\begin{aligned}
\left( E \otimes K^{\otimes 2m+1}+Id^{\otimes 2m+1} \otimes E \right) \ \ \ \ \ \ \ \ \ \ \ \ \ \ \  \ \ \ \ \ \ \ \ \ \ \ \ \ \ \ \ \ \ \ \ \ \ \ \ \ \ \ \ \ \ \ \ \ \ \ \ \ \ \ \ \ \ \ \ \ \ \ \ \ \ \ \  \\
\left( \sum_{i_1,...,i_{m+1}=0}^{N-1} c_{m+1}^{i_{m+1}} C_{m}^{i_1,...,i_{m}} v_{i_{m+1}}\otimes v_{2m}^{i_1,...,i_{m}}\otimes v_{N-1-i_{m+1}} \right) \ \ \ \ \ \ \ \ \ \ \ \ \ \ \ \ \ \ \ \ \ \ \ \ \ \ \\ 
+  \mathlarger{\sum_{i_{m+1}=0}^{N-1}} q^{\lambda-2(N-1-i_{m+1})} \cdot c_{m+1}^{i_{m+1}} \cdot \ \ \ \ \ \ \ \ \ \ \ \ \ \ \ \ \ \ \ \ \ \ \ \ \ \ \ \ \ \ \ \ \ \ \ \ \ \ \ \ \ \ \ \ \ \ \ \ \ \ \ \ \ \ \ \ \ \ \ \ \ \ \ \ \ \ \ \  \ \\
\cdot \left( v_{i_{m+1}}^{m+1} \otimes   {\color{red}\Delta^m(E) \left( \sum_{i_1,...,i_{m}=0}^{N-1}  C_{m}^{i_1,...,i_{m}} v_{2m}^{i_1,...,i_{m}} \right)} \otimes v_{N-1-i_{m+1}} \right) =0 \ \ \ \ \ \ \ \ \ \ \ \ \ \ \ \ 
\end{aligned}
\end{equation}
Using the induction step $P(m)$ written in formula \ref{eq:5}, we conclude that the second sum from the previous formula vanishes anyway, since it contains the sum written in red. So, we obtain that the requirement that we need is equivalent to the following expression:
\begin{equation}
\begin{aligned}
\left( E \otimes K^{\otimes 2m+1}+Id^{\otimes 2m+1} \otimes E \right) \ \ \ \ \ \ \ \ \ \ \ \ \ \ \ \ \ \ \ \ \ \ \ \ \ \ \ \ \ \ \ \ \ \ \ \ \ \ \ \ \ \\ 
\left( \sum_{i_1,...,i_{m+1}=0}^{N-1} c_{m+1}^{i_{m+1}} \cdot C_{m}^{i_1,...,i_{m}}\cdot  v_{i_{m+1}}\otimes v_{2m}^{i_1,...,i_{m}}\otimes v_{N-1-i_{m+1}}\right)=0.
\end{aligned}
\end{equation}
\begin{notation}\label{N:1}
We denote by:
$$w^m_{Coev}:=(Id^{\otimes m}\otimes f_1 \otimes ...\otimes f_m) \ v^{m}_{Coev}=\sum_{i_1,...,i_{m}=0}^{N-1} C_{m}^{i_1,...,i_{m}}  \cdot v_{2m}^{i_1,...,i_{m}}.$$
\end{notation}
In this way, the condition that we are interested in can be written as:
\begin{equation}
\left( E \otimes K^{\otimes 2m+1}+Id^{\otimes 2m+1} \otimes E \right) \left( \sum_{i_{m+1}=0}^{N-1} c_{m+1}^{i_{m+1}} \cdot  v_{i_{m+1}}\otimes w^m_{Coev}\otimes v_{N-1-i_{m+1}}\right)=0.
\end{equation}
Since we have the operator $K^{\otimes 2m}$ in the middle of the first term that acts above, we compute its action on the corresponding vector.
\begin{remark} \label{R:11}Using the action from remark \ref{R:10}, we get:
$$K^{\otimes 2m}\curvearrowright w^m_{Coev}=q^{2m \lambda-2m(N-1)}w^m_{Coev}.$$
\end{remark}
We arrive at the following condition:
\begin{equation}
\left( E \otimes K^{\otimes 2m+1}+Id^{\otimes 2m+1} \otimes E \right) \left( \sum_{i_{m+1}=0}^{N-1} c_{m+1}^{i_{m+1}} \cdot  v_{i_{m+1}}\otimes w^m_{Coev}\otimes v_{N-1-i_{m+1}}\right)=0.
\end{equation}
Splitting it up into two sums, and using the $K$-action from \ref{R:11} one gets:
\begin{equation}
\begin{aligned}
\sum_{i_{m+1}=0}^{N-1} c_{m+1}^{i_{m+1}} [\lambda+1-i_{m+1}]_{\xi}\cdot  q^{2m\lambda-2m(N-1)+\lambda-2(N-1-i_{m+1})}\cdot \ \ \ \ \ \ \ \ \ \ \ \ \ \ \ \ \ \ \ \ \\
\cdot v_{i_{m+1}-1}\otimes w^m_{Coev}\otimes v_{N-1-i_{m+1}} + \ \ \ \ \ \ \ \ \ \ \ \ \ \ \ \\
+ \sum_{i_{m+1}=0}^{N-2} c_{m+1}^{i_{m+1}} \cdot  [\lambda+1-(N-1-i_{m+1})]_{\xi} \cdot v_{i_{m+1}}\otimes w^m_{Coev}\otimes v_{N-1-i_{m+1}}=0. \ \ \ \ \ \ \ \ \ \ 
\end{aligned}
\end{equation}
The change of variables $k=i_{m+1}-1$ in the first sum and $k=i_{m+1}$ leads to the following equation:
\begin{equation}\label{eq:7}
\begin{aligned}
\mathlarger{\sum_{k=0}^{N-2}} \left( c_{m+1}^{k+1} [\lambda-k]_{\xi}\cdot  q^{(2m+1)\lambda-2(m+1)(N-1)+2(k+1)}+c_{m+1}^{k} \cdot  [\lambda+2-N+k]_{\xi} \right) \cdot \ \ \ \ \ \ \ \ \ \ \ \ \ \ \ \ \ \ \ \ \\
\cdot v_{k}\otimes w^m_{Coev}\otimes v_{N-2-k}=0 \ \ \ \ \ \ \ \ \ \ \ \ \ \ \ \ \ \ \ \ \ \ \ \ \ \ \ \ \ \ \ \ \ \ \ \ \ \ \ \ \ \ \ \ \ \ \ \ \ \ \ \ 
\end{aligned}
\end{equation}
Let us define the family of coefficients $\{ c_{m+1}^k\}_{k\in \overline{1,N-1}}$ prescribed by the following system:
\begin{equation}\label{eq:13}
\begin{aligned}
 \frac{c_{m+1}^{k+1}}{c_{m+1}^k}= - \frac{[\lambda+2-N+k]_{\xi}}{[\lambda-k]_{\xi}} \cdot q^{-(2m+1)\lambda+2(m+1)(N-1)-2(k+1)} \\
    c_{m+1}^0=1. \ \ \ \ \ \ \ \ \ \ \ \ \ \ \ \ \ \ \ \ \ \ \ \ \ \ \ \ \ \ \ \ \ \ \ \ \ \ \ \ \ \ \ \ \ \ \ \ \ \ \ \ \ \ \ \ \ \ \ \ \ \ 
\end{aligned}
\end{equation}
Then, we obtain that the equation \ref{eq:7} is true for this sequence. This concludes the induction step. 
\end{proof}
\subsection{{ \bf Step 2-Normalising the partial coevaluation}}\label{Step2}
\

So far, we have seen that we can modify the last $n$ components of the tensor product by the isomorphism $\varphi_n$, such that we arrive in a particular highest weight space. 
$$(Id^{\otimes n-1}_{U_N(\lambda)}\otimes \psi^N_n) \cdot \mathscr F_{\lambda} \left( \tikz[x=1mm,y=1mm,baseline=0.5ex]{\draw[<-] (0,3) .. controls (0,0) and (3,0) .. (3,3); \draw[<-] (-3,3) .. controls (-3,-3) and (6,-3) .. (6,3); \draw[draw=none, use as bounding box](-0.5,0) rectangle (3,3);}\right)(1)\in
W^{\xi,\lambda}_{2(n-1),(n-1)(N-1)}\subseteq U_N(\lambda)^{\otimes 2(n-1)}.$$
In fact, in the definition \ref{DEFINITION}, we have the partial coevaluation, corresponding to the first strand being open. In this part, we show that if we start with the highest weight vector $v_0$, we will arrive actually in a highest weight space.
We remind the notation from \ref{N:1}:
\begin{equation}\label{eq:8}
w^{n-1}_{Coev}:=(Id^{\otimes n-1}_{U_N(\lambda)}\otimes \psi^N_n) \cdot \mathscr F_{\lambda} \left( \tikz[x=1mm,y=1mm,baseline=0.5ex]{\draw[<-] (0,3) .. controls (0,0) and (3,0) .. (3,3); \draw[<-] (-3,3) .. controls (-3,-3) and (6,-3) .. (6,3); \draw[draw=none, use as bounding box](-0.5,0) rectangle (3,3);}\right)(1).
\end{equation}
For the bottom part of the diagram, we have the following morphism corresponding to the partial coevaluation as in \ref{DEFINITION}. We will modify this using the isomorpfism of vector spaces from the first step. So, we are interested in the morphism:
\begin{equation*}
\begin{aligned}
\big(Id^{\otimes n}_{U_N(\lambda)}\otimes \psi^N_n\big) \circ \mathscr F_{\lambda} \left(\uparrow \tikz[x=1mm,y=1mm,baseline=0.5ex]{\draw[<-] (0,3) .. controls (0,0) and (3,0) .. (3,3); \draw[<-] (-3,3) .. controls (-3,-3) and (6,-3) .. (6,3); \draw[draw=none, use as bounding box](-0.5,0) rectangle (3,3);}\right)= Id_{U^{\lambda}_N} \otimes \left(\tcoev^{n-1}_{\lambda} \circ....\circ \tcoev^{1}_{\lambda}\right) \in \ \ \ \ \ \ \ \ \ \ \ \ \ \ \ \ \\
\in \textit{\large Hom}_{\C}\Big( \ U_N(\lambda)  \ \text{\Large ,} \ U_N(\lambda)^{\otimes 2n-1} \Big). \ \ \ \ \ \ \ \ \ \ \ \ \ \ \ 
\end{aligned}
\end{equation*}
We will evaluate it onto the highest weight vector $v_0\in U_N(\lambda)$. Our aim is to show that it arrives in a particular highest weight space.

From the functoriality of $\mathscr F_{\lambda}$ and equation \ref{eq:8}, we obtain:
\begin{equation}
(Id^{\otimes n}_{U_N(\lambda)}\otimes \psi^N_n) \circ \mathscr F_{\lambda} \left(\uparrow \tikz[x=1mm,y=1mm,baseline=0.5ex]{\draw[<-] (0,3) .. controls (0,0) and (3,0) .. (3,3); \draw[<-] (-3,3) .. controls (-3,-3) and (6,-3) .. (6,3); \draw[draw=none, use as bounding box](-0.5,0) rectangle (3,3);}\right)(v_0)=v_0\otimes w^{n-1}_{Coev}. 
\end{equation}
We compute the $E$ action on this vector, which is done using the comultiplication:
\begin{equation}\label{eq:8'}
\begin{aligned}
E(v_0\otimes w^{n-1}_{Coev})=(E\otimes K+Id\otimes E)(v_0\otimes w^{n-1}_{Coev})=\\
=E(v_0)\otimes K(w^{n-1}_{Coev})+v_0\otimes E(w^{n-1}_{Coev})=0.
\end{aligned}
\end{equation}
The first term is zero, because we started with a highest weight vector. The second term vanishes from the construction ( property \ref{eq:8}). 
Concerning the $K$-action, using the weight of $w^{n-1}_{Coev}$ from \ref{eq:8}, we have:
\begin{equation}\label{eq:9}
\begin{aligned}
K(v_0\otimes w^{n-1}_{Coev})=q^{\lambda}q^{2(n-1)\lambda-2(N-1)}v_0\otimes w^{n-1}_{Coev}=\\
=q^2(n-1)\lambda-2(N-1)v_0\otimes w^{n-1}_{Coev}.
\end{aligned}
\end{equation}
The equations \ref{eq:8'} and \ref{eq:9} tell us that the modified partial coevaluation leads indeed in a highest weight space:
\begin{equation}
v_0\otimes w^{n-1}_{Coev}\in W^{\xi,\lambda}_{2n-1,(n-1)(N-1)}\subseteq U_N(\lambda)^{\otimes 2n-1}.
\end{equation}
We would like to emphasise at this point that here we see the importance of starting with a highest weight vector for computing the invariant. This property had a role in order to obtain equation \ref{eq:8}.

\subsection{\bf Step 3 Highest weight space}
\

So far, we have seen that starting from the partial coevaluation, we can arrange to arrive in a highest weight space, if we pay the price of composing with the morphism 
$$Id^{\otimes n}_{U_N(\lambda)} \otimes \psi^N_n.$$
In the sequel, we will see that actually if we insert this morphism and its inverse to the definition of the ADO polynomial, this will not change the value of this invariant.
\begin{lemma}
The ADO invariant can be obtained as follows:
\begin{equation}\label{DEFINITION'}
\begin{aligned}
\Phi_{N}(L,\lambda)=\xi^{(N-1)\lambda w(\beta_n)}\cdot
 \pi_{v_0} \circ \left(Id_{U^{\lambda}_N}\otimes \left(\ev^{n-1}_{\lambda} \circ...\circ \ev^{1}_{\lambda}\right)\right) \circ  \ \ \ \ \  \ \ \ \ \ \ \ \ \ \\
 {\color{blue} \left( Id^{\otimes n}_{U_N(\lambda)} \otimes \left(\psi^N_n \right)^{-1} \right)}\circ \left( \varphi_{2n-1}^{\xi,\lambda}(\beta_n\otimes \mathbb I_{n-1} \right) \circ {\color{blue} \left( Id^{\otimes n}_{U_N(\lambda)} \otimes \psi^N_n \right)}\circ\\
 \left( Id_{U^{\lambda}_N} \otimes \left(\tcoev^{n-1}_{\lambda} \circ....\circ \tcoev^{1}_{\lambda}\right) \right)(v_0) .
\end{aligned}
\end{equation}
\end{lemma}
\begin{proof}
The formula looks the same as the definition \ref{DEFINITION}, with the exception of the two additional morphisms written in blue. The key point is the fact that all the changes that occur from the function from \ref{Step2}, act nontrivially on the part of the diagram corresponding to the straight strands $\unit_{n-1}$, while being trivial on the part of the diagram which contains the non-trivial braid $\beta_n$. This shows that we have the following property:
\begin{equation}
\begin{aligned}
{\color{blue} \left( Id^{\otimes n}_{U_N(\lambda)} \otimes \left(\psi^N_n \right)^{-1} \right)}\circ
\left( \varphi_{2n-1}^{\lambda}(\beta_n\otimes \mathbb I_{n-1} \right) \circ {\color{blue} \left( Id^{\otimes n}_{U_N(\lambda)} \otimes \psi^N_n \right)}= \ \ \ \ \ \ \ \ \ \ \ \\
={\color{blue} \left( Id^{\otimes n}_{U_N(\lambda)} \otimes \left(\psi^N_n \right)^{-1} \right)}\circ \left( \varphi_{n}^{\lambda}(\beta_n)\otimes Id^{\otimes n}_{U_N(\lambda)} \right) \circ {\color{blue} \left( Id^{\otimes n}_{U_N(\lambda)} \otimes \psi^N_n \right)}= \ \ \ \ \ \ \ \ \ \\
=\left( \varphi_{n}^{\xi,\lambda}(\beta_n)\otimes Id^{\otimes n}_{U_N(\lambda)} \right). \ \ \ \ \ \ \ \ \ \ \ \ \ \ \ \ \ \ \ \ \  \ \ \ \ \ \ \ \ \ \ \ \ \ \ 
\end{aligned}
\end{equation}

Using this property, we conclude the new expression of the invariant. 
\end{proof}
This means that the effect obtained by composing with this normalising function and its inverse, do not affect the final value of the invariant. However, this function has an important role- it helps us to define the whole construction, which a priori passes through the representation $U_N(\lambda)^{\otimes n}\otimes (U_N(\lambda)^{*})^{\otimes n-1}$ through the particular highest weight space $W^{\xi,\lambda}_{2n-1,(n-1)(N-1)}$.
\begin{remark} Using the notation from equation \ref{eq:8} and the quantum representation from section \ref{Quantumrep}, we conclude the following description for the ADO invariant:
\begin{equation}\label{DEFINITION''}
\begin{aligned}
\Phi_{N}(L,\lambda)=\xi^{(N-1)\lambda w(\beta_n)}\cdot
 \pi_{v_0} \circ  \ \ \ \ \ \ \ \ \ \ \ \ \ \ \ \ \ \ \ \ \ \ \ \ \ \ \ \ \ \ \ \ \ \ \ \ \ \ \ \ \ \ \ \ \ \ \ \ \ \ \ \ \ \\
  \circ \left(Id_{U^{\lambda}_N}\otimes \left(\ev^{n-1}_{\lambda} \circ...\circ \ev^{1}_{\lambda}\right)\right) \circ {\color{blue} \left( Id^{\otimes n}_{(U_N(\lambda)} \otimes \left(\psi^N_n \right)^{-1} \right)}\circ \ \ \ \ \ \ \ \ \ \\
\left( ^{W}\varphi_{2n-1,(n-1)(N-1)}^{\xi, \lambda} \left(\beta_n\otimes \mathbb I_{n-1}\right) \right) (v_0 \otimes w^{n-1}_{Coev}).
\end{aligned}
\end{equation}
\end{remark}

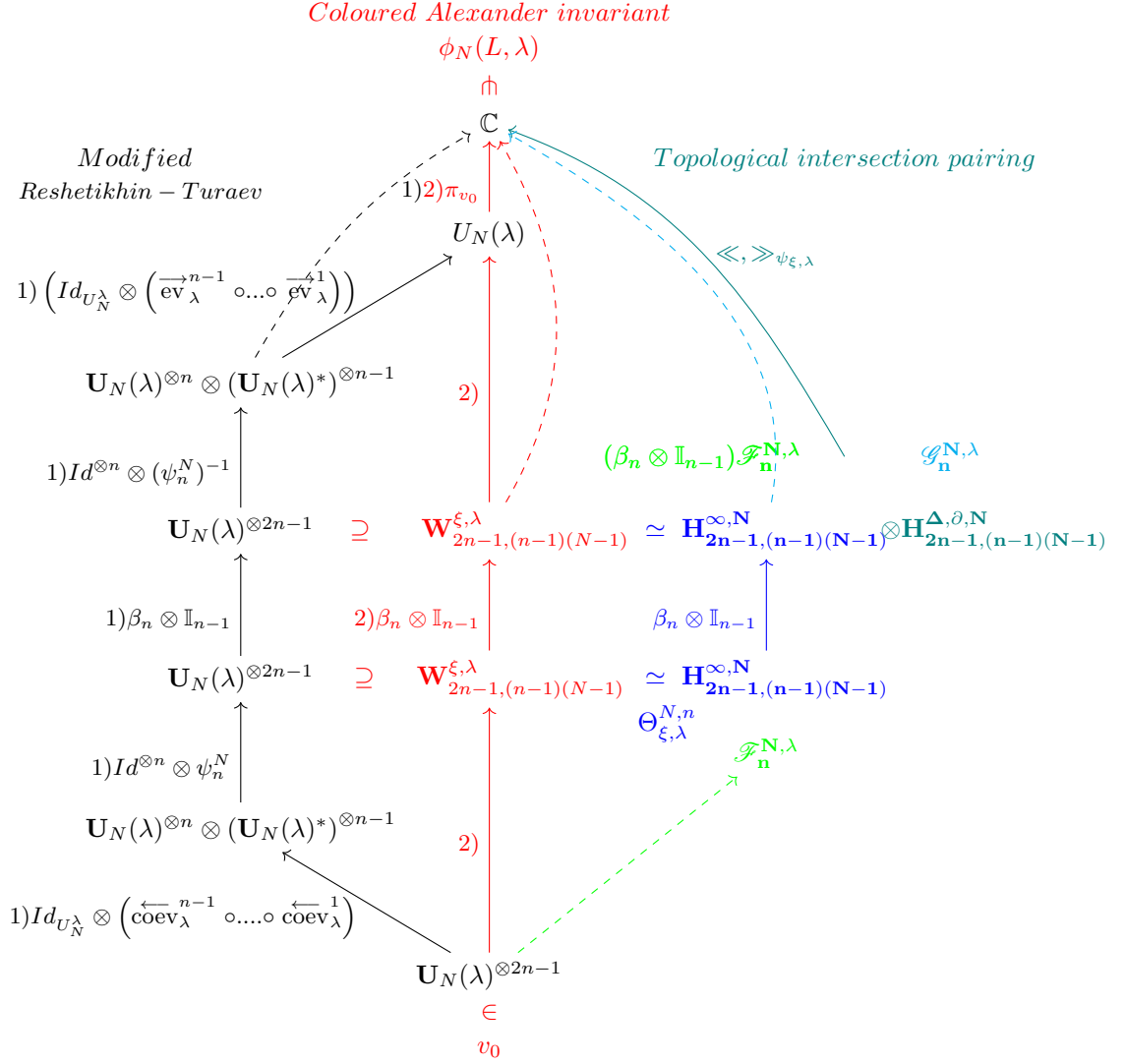
\newcommand{\tEv}{\stackrel{\longleftarrow}{\operatorname{Ev}}}
\begin{center}\label{pict}
\begin{tikzpicture}
[x=0.95mm,y=1mm]
\node (Alex)  [color=red]             at (-5,130)    {$Coloured \ Alexander \ invariant$};
\node (c0')  [color=red]             at (-5,-10)    {$v_0$};
\node (c0)               at (-5,0)    {$\mathbf{U}_N(\lambda)^{\otimes 2n-1}$};
\node (c1) [color=red] at (-5,40)   {$\supseteq \ \ \ \ \mathbf{W}^{\xi,\lambda}_{2n-1,(n-1)(N-1)}$};
\node (c2) [color=red]   at (-5,60)   {$\supseteq \ \ \ \ \ \mathbf{W}^{\xi,\lambda}_{2n-1,(n-1)(N-1)}$};
\node (c3) []  at (-5,100)   {$U_N(\lambda)$};
\node (c4) []  at (-5,115)   {$\C$};
\node (c4') [color=red]  at (-5,125)   {$\phi_N(L,\lambda)$};
\node (c5) [color=red]  at (-5,120)   {$\pitchfork$};
\node (c5') [color=red]  at (-5,-5)   {$\in$};
\node (e) [color=blue]  at (20,34)   {$\Theta_{\xi,\lambda}^{N,n}$};
\node (f) [color=green]  at (34,30)   {$\mathbf{\mathscr{F}_{n}^{N,\lambda}}$};
\node (g) [color=cyan]  at (60,70)   {$\mathbf{\mathscr{G}_{n}^{N,\lambda}}$};
\node (bf) [color=green]  at (25,70)   {$(\beta_n\otimes \mathbb I_{n-1})\mathbf{\mathscr{F}_{n}^{N,\lambda}}$};
\node (bf) [color=green]  at (25,70)   {$(\beta_n\otimes \mathbb I_{n-1})\mathbf{\mathscr{F}_{n}^{N,\lambda}}$};
\node (q) [color=teal]  at (45,110)   {$Topological \ intersection \ pairing$};
\node (qq) [color=black]  at (-55,110)   {$Modified$};

\node (b1)               at (-40,20)    {$\mathbf{U}_N(\lambda)^{\otimes n}\otimes {(\mathbf{U}_N(\lambda)^{*})}^{\otimes n-1}$};
\node (b2) [] at (-40,40)   {$\mathbf{U}_N(\lambda)^{\otimes 2n-1}$};
\node (b3) []   at (-40,60)   {$\mathbf{U}_N(\lambda)^{\otimes 2n-1}$};
\node (b4) []  at (-40,80)   {$\mathbf{U}_N(\lambda)^{\otimes n}\otimes {(\mathbf{U}_N(\lambda)^{*})}^{\otimes n-1}$};
\node (d1) [color=blue]  at (34,40)  {$\simeq \ \mathbf{H^{\infty,N}_{2n-1,(n-1)(N-1)}}$};
\node (d2) [color=blue]  at (34,60)  {$\simeq \ \mathbf{H^{\infty,N}_{2n-1,(n-1)(N-1)}}$};
\node (d2') [color=teal]  at (66,60)  {$\otimes\mathbf{H^{\Delta,\partial,N}_{2n-1,(n-1)(N-1)}}$};

\draw[->]             (c0) to node[left,yshift=-2mm,font=\small]{$1)Id_{U^{\lambda}_N} \otimes \left(\tcoev^{n-1}_{\lambda} \circ....\circ \tcoev^{1}_{\lambda}\right)$} (b1);
\draw[->]             (b1) to node[left,yshift=-2mm,font=\small]{$1) Id^{\otimes n}\otimes \psi_n^N$} (b2);
\draw[->]             (b2) to node[left,yshift=-2mm,font=\small]{$1) \beta_n\otimes \mathbb I_{n-1}$} (b3);
\draw[->]             (b3) to node[left,yshift=-2mm,font=\small]{$1) Id^{\otimes n}\otimes (\psi_n^N)^{-1}$} (b4);
\draw[->]             (b4) to node[left,yshift=2mm,font=\small]{$1)\left(Id_{U^{\lambda}_N}\otimes \left(\ev^{n-1}_{\lambda} \circ...\circ \ev^{1}_{\lambda}\right)\right) $} (c3);
\draw[->]    [color=red]         (c0) to node[left,yshift=-2mm,font=\small]{$2)$} (c1);
\draw[->]   [color=red]          (c1) to node[left,yshift=-2mm,font=\small]{$2) \beta_n\otimes \mathbb I_{n-1}$} (c2);
\draw[->]   [color=red]          (c2) to node[left,yshift=-2mm,font=\small]{$2) $} (c3);
\draw[->]  [color=red]           (c3) to node[left,yshift=-2mm,font=\small]{${\color{black} 1)}{\color{red} 2) \pi_{v_0}}$} (c4);
\draw[->,dashed]  [in=210,out=60]           (b4) to node[left,yshift=5mm,xshift=-10mm,font=\small]{$Reshetikhin-Turaev$} (c4);
\draw[->,color=teal] [in=-20,out=120]  (45,70)  to node[right,font=\small]{$ \ll , \gg_{\psi_{\xi,\lambda}} $}      (c4);
\draw[->]   [color=blue]          (d1) to node[left,yshift=-2mm,font=\small]{$ \beta_n\otimes \mathbb I_{n-1}$} (d2);
\draw[->]   [color=green,dashed]          (c0) to node[left,yshift=-2mm,font=\small]{} (f);
\draw[->]   [color=red,dashed]  [in=-60,out=60]        (c2) to node[left,yshift=-2mm,font=\small]{} (c4);
\draw[->]   [color=cyan,dashed]  [in=-30,out=80]        (d2) to node[left,yshift=-2mm,font=
\small]{} (c4);

\end{tikzpicture}
\end{center}

\subsection{{ \bf Step 4-Construction of the homology class $\mathscr F_n^{N,\lambda}$}}\label{Step4}

\

As we have seen in definition \ref{DEFINITION''}, the advantage of the technical part from the previous steps, helps us to see the coloured Alexander invariant through the highest weight space $W^{\xi,\lambda}_{2n-1,(n-1)(N-1)}$. Now we will pass towards topology. 

A special property of highest weight spaces is the powerfull tool that they have topological counterparts. We refer to the discussion from section \ref{4}. Using the identification discussed in section \ref{4} and the notation \ref{N:2}, we obtain:\\

$ \ \ \ \ \ \ \ \ \ \ \ \ \ \ \ \ \ {^{W}\varphi}^{\xi,\lambda}_{2n-1,(n-1)(N-1)} \ \ \ \ \ \ \ \ \ \ \ \ \ \ \ \ \ l^{\infty,N}_{2n-1,(n-1)(N-1)}|_{\psi_{\xi,\lambda}}$
$$B_{2n-1} \curvearrowright W^{\xi,\lambda}_{2n-1,(n-1)(N-1)}\simeq_{\Theta_{\xi,\lambda}^{N,n}} \HH^{\infty,N}_{2n-1,(n-1)(N-1)} |_{\psi_{\xi,\lambda}} \curvearrowleft B_{2n-1}.$$
\begin{definition}(Homology class corresponding to the cups $\mathscr F_n^{N,\lambda}$)\\
Let us define the class:
\begin{equation}\label{eq:12}
\mathscr F_n^{N,\lambda}:={\left( \Theta^{N,n}_{\xi,\lambda}\right) } ^{-1}\big(v_0 \otimes w^{n-1}_{Coev}\big).
\end{equation}
\end{definition}
  
\subsection{{ \bf Step 5-Construction of the dual homology class $\mathscr G_n^{N,\lambda}$}}\label{Step5}
\

In this step, we aim construct the dual manifold $\mathscr G_n^N$, which corresponds to the caps from the algebraic side of the of the invariant. 
Consider the function:
$$\pi_{v_0} \circ \left(Id_{U^{\lambda}_N}\otimes \left(\ev^{n-1}_{\lambda} \circ...\circ \ev^{1}_{\lambda}\right)\right) \circ  {\color{blue} \left( Id^{\otimes n}_{(U_N(\lambda)} \otimes \left(\psi^N_n \right)^{-1} \right)}: U_N(\lambda)^{\otimes 2n-1} \rightarrow \C.$$
If we restrict this map to the highest weight space, we obtain the following function:
\begin{equation}\label{eq:10} 
Ev^{\lambda}_{N,n}: W^{\xi,\lambda}_{2n-1,(n-1)(N-1)}\rightarrow \C.
\end{equation}
In other words, we have that $Ev^{\lambda}_{N,n} \in Hom_{\C}(W^{\xi,\lambda}_{2n-1,(n-1)(N-1)},\C)$.
Composing with the isomorphism which related the highest weight space with special Lawrence representation, we obtain:
$$Ev^{\lambda}_{N,n} \circ { \Theta^N_{\xi,\lambda} } \in Hom_{\C}\left( \HH^{\infty,N}_{2n-1,(n-1)(N-1)} |_{\psi_{\xi,\lambda}},\C \right).$$
At this moment, it occurs the point where we use the topological pairing. On the other hand, from corollary \ref{C:pairing}, we have the non-degenerate topological form:
$$\ll , \gg |_{\psi_{\xi,\lambda}}: \HH^{\infty, N}_{2n-1,(n-1)(N-1)}|_{\psi_{\xi,\lambda}} \otimes  \HH^{\Delta,\partial,N}_{2n-1,(n-1)(N-1)} |_{\psi_{\xi,\lambda}}\rightarrow \C.$$
Since, this is a non-degenerate form over a field, each element in the dual of the first space, can be identified with the pairing with respect to a fix element in the second homology space. Using this, we are able to introduce the second homology class.
\begin{definition}(Homology class corresponding to the caps $\mathscr G^{N,\lambda}_{n}$)\label{G}

Let us consider the unique homology class $\mathscr G^{N,\lambda}_{n} \in \HH^{\Delta,\partial,N}_{2n-1,(n-1)(N-1)} |_{\psi_{\xi,\lambda}}$ with the property that:
$$Ev^{\lambda}_{N,n} \circ { \Theta^{N,n}_{\xi,\lambda}}  (\cdot )=\ll \cdot, \mathscr G^{N,\lambda}_{n} \gg|_{\psi_{\xi,\lambda}}$$
as elements in the dual space $ \left( \HH^{\infty,N}_{2n-1,(n-1)(N-1)} |_{\psi_{\xi,\lambda}} \right) ^*.$
\end{definition}

\subsection{{ \bf Step 6-Proof of the topological model}}\label{Step6}

So far, we defined the two homology classes that we work with, which are prescribed by cups and caps. Now, we will put everything together and show that their pairing leads to the coloured Alexander polynomials.
\begin{equation}
\begin{aligned}
\Phi_{N}(L,\lambda)=^{\ref{DEFINITION''}} \ \ \xi^{(N-1)\lambda w(\beta_n)}\cdot
 \pi_{v_0} \circ  \ \ \ \ \ \ \ \ \ \ \ \ \ \ \ \ \ \ \ \ \ \ \ \ \ \ \ \ \ \ \ \ \ \ \ \ \ \ \ \ \ \ \ \ \ \ \ \ \ \ \ \ \ \\
  \circ \left(Id_{U^{\lambda}_N}\otimes \left(\ev^{n-1}_{\lambda} \circ...\circ \ev^{1}_{\lambda}\right)\right) \circ {\color{blue} \left( Id^{\otimes n}_{(U_N(\lambda)} \otimes \left(\psi^N_n \right)^{-1} \right)}\circ \ \ \ \ \ \ \ \ \ \\
\left( ^{W}\varphi_{2n-1,(n-1)(N-1)}^{\xi, \lambda} \left(\beta_n\otimes \mathbb I_{n-1}\right) \right) { \color{ green} \left( {\left( \Theta^{N,n}_{\xi,\lambda}\right) }  {\left( \Theta^{N,n}_{\xi,\lambda}\right) } ^{-1} \right) }(v_0 \otimes w^{n-1}_{Coev})=\\
=^{\ref{Step4}} \ \ \ \xi^{(N-1)\lambda w(\beta_n)}\cdot
 \pi_{v_0} \circ  \ \ \ \ \ \ \ \ \ \ \ \ \ \ \ \ \ \ \ \ \ \ \ \ \ \ \ \ \ \ \ \ \ \ \ \ \ \ \ \ \ \ \ \ \ \ \ \ \ \ \ \ \ \\
  \circ \left(Id_{U^{\lambda}_N}\otimes \left(\ev^{n-1}_{\lambda} \circ...\circ \ev^{1}_{\lambda}\right)\right) \circ {\color{blue} \left( Id^{\otimes n}_{U_N(\lambda)} \otimes \left(\psi^N_n \right)^{-1} \right)}\circ \ \ \ \ \ \ \ \ \ \\
{ \color{ orange}  {\left( \Theta^{N,n}_{\xi,\lambda}\right) }   {\left( \Theta^{N,n}_{\xi,\lambda}\right)^{-1} } }\left( ^{W}\varphi_{2n-1,(n-1)(N-1)}^{\xi, \lambda} \left(\beta_n\otimes \mathbb I_{n-1}\right) \right) { \color{ green} { \Theta^{N,n}_{\xi,\lambda} }  }(\mathscr F_n^{N,\lambda})=\\
=^{\ref{C:Ident}} \ \ \ \xi^{(N-1)\lambda w(\beta_n)}\cdot
 \pi_{v_0} \circ  \ \ \ \ \ \ \ \ \ \ \ \ \ \ \ \ \ \ \ \ \ \ \ \ \ \ \ \ \ \ \ \ \ \ \ \ \ \ \ \ \ \ \ \ \ \ \ \ \ \ \ \ \ \\
  \circ \left(Id_{U^{\lambda}_N}\otimes \left(\ev^{n-1}_{\lambda} \circ...\circ \ev^{1}_{\lambda}\right)\right) \circ {\color{blue} \left( Id^{\otimes n}_{U_N(\lambda)} \otimes \left(\psi^N_n \right)^{-1} \right)}\circ { \color{ orange}  {\left( \Theta^{N,n}_{\xi,\lambda}\right)  }} \ \ \ \ \ \ \ \ \ \\
 { \color{ green}  l^{\infty,N}_{n,m}|_{\psi_{\xi,\lambda}} (\beta_n \otimes \mathbb I_{n-1}) }(\mathscr F_n^{N,\lambda})= \ \ \ \ \ \ \ \ \ \ \ \ \ \ \ \ \ \ \ \ \ \ \ \ \ \ \ \ \ \ \ \ \ \ \ \ \ \ \ \ \ \ \ \\
=^{\ref{eq:10}} \ \ \ \xi^{(N-1)\lambda w(\beta_n)}\cdot \Big( Ev^{\lambda}_{N,n} \circ { \color{ orange}  {\left( \Theta^{N,n}_{\xi,\lambda}\right) } } \Big) \Big(
 { \color{ green}  l^{\infty,N}_{n,m}|_{\psi_{\xi,\lambda}} (\beta_n \otimes \mathbb I_{n-1}) }(\mathscr F_n^{N,\lambda}) \Big) =\\
=^{\ref{G}} \ \ \ \xi^{(N-1)\lambda w(\beta_n)}\cdot \ll 
 { \color{ green}  l^{\infty,N}_{n,m}|_{\psi_{\xi,\lambda}} (\beta_n \otimes \mathbb I_{n-1}) }\mathscr F_n^{N,\lambda}, \mathscr G_n^{N,\lambda} \gg|_{\psi_{\xi,\lambda}}.\\
\end{aligned}
\end{equation}
\end{proof}
\section{Globalisation of the homology classes}\label{10}

So far, for each $\lambda \in \C \setminus \Z$, we have constructed two homology classes 
$$\mathscr F^{N,\lambda}_n \in \HH^{\infty,N}_{2n-1, (n-1)(N-1)}|_{\psi_{\xi,\lambda}}  \ \ \ \ \ \  \mathscr G^{N,\lambda}_n \in \HH^{\Delta,\partial,N}_{2n-1, (n-1)(N-1)}|_{\psi_{\xi,\lambda}}$$

that lead to the coloured Alexander invariant $\varphi_N(L,\lambda)$ evaluated at $\lambda$, through the topological intersection form. The aim for this section is to show that these classes can be globalised. We show that they come from two classes that live in the corresponding homologies, but which are not specialised by a specialisation that depends on the parameter $\lambda$. In other words, the two new homology classes are in some sense intrinsic with respect to this parameter. 
The main idea is the fact that the function $\phi_{q,\lambda}$ can be defined over the ring $\Z[q^{\pm1},s^{\pm 1}]$. Then, we try to lift all the spaces over this ring, in order to globalise the parameter $\lambda \in \C$ into the abstract variable $s$, through the specialisation $\eta_{\xi,\lambda}$.

\subsection{Correspondence highest weight spaces - Lawrence representation}
\begin{definition}(Specialisation of coefficients)\label{D:1} 

{\bf I. Quantum side} Consider $\eta_{\xi,\lambda}: \Z[q^{\pm},s^{\pm}]\rightarrow \C$ defined as 
\begin{equation}
\begin{aligned}
\eta_{\xi,\lambda}(q)=\xi; \ \ \ \ \ 
\eta_{\xi,\lambda}(s)=\xi^{\lambda}.
\end{aligned}
\end{equation}
{\bf II. Homological side} Let $\gamma: \Z[x^{\pm},d^{\pm}]\rightarrow \Z[q^{\pm 1};s^{\pm}]$ defined as:
\begin{equation}
\begin{aligned}
\gamma(x)= s^{2}; \ \ \ \ \ 
\gamma(d)=-q^{-2}.
\end{aligned}
\end{equation}

\end{definition}
\begin{definition}(New "Highest weight space" )

We use the fact that all the actions of the quantum group onto $U_N(\lambda)$ actually can be written using the ring $\Z[\xi^{\pm1},\xi^{\pm \lambda}]$.

Consider the vector space $V^N_{n,m},W^N_{n,m}$ having the same generators as $V^{\xi,\lambda}_{n,m},W^{\xi,\lambda}_{n,m}$, over the ring $\Z[q^{\pm1},s^{\pm 1}]$. 
\end{definition}
We notice that the actions of the quantum group that send weight spaces onto highest weight spaces from \ref{P:bqu1}, can be written in these new versions over the ring $\Z[q^{\pm1},s^{\pm 1}]$. In the following, we will not discuss about the braid group action, since the homological classes are predicted by the evaluations and coevaluations, we do not need to consider the braid part.

\begin{definition} (New version of basis){\label{D:newbasis}}

Consider the function $\phi^{N}_{n,m}: E^N_{n,m}\rightarrow W^{N}_{n,m}$ by the formula:

\begin{equation}
\begin{aligned}
\phi^{N}_{n,m}(e_1,...,e_{n-1}):=\sum_{k=0}^{N-1} (-1)^k {s}^{-k(n-1)} {q}^{2mk-k(k+1)}  \ \ \ \ \ \ \ \ \ \ \ \ \ \ \ \ \ \ \ \ \\ 
v_k \otimes E^k\big(s^{\sum_{i=1}^{n-1} (i+1) e_i}v_{e_1}\otimes...\otimes v_{e_{n-1}}\big).
\end{aligned}
\end{equation}
\end{definition}
Now, we define a truncated model of the function $\Theta$, at the generic level, such that it globalise the function $\Theta_{\xi,\lambda}$, over the ring of Laurent polynomials in two variables.
\begin{definition} (Truncated function at generic parameters)

Consider the map defined as follows on generators:
$$\Theta^{N,n}:\HH^{\infty,N}_{2n-1, (n-1)(N-1)}|_{\gamma}\rightarrow W^N_{2n-1, (n-1)(N-1)}$$ 
$$\Theta^{N,n}([\tilde{\F_e}]):=\phi^{N}_{2n-1, (n-1)(N-1)}(e).$$
\end{definition}
\begin{remark}
a) Following the specialisation of coefficients:
\begin{equation}\label{eq:10}
\Theta^{N,n}|_{\tilde{\eta}_{\xi,\lambda}}=\Theta^{N,n}_{\xi,\lambda}.
\end{equation}
b) On the other hand, know that the highest weight spaces 
$$(W^N_{2n-1, (n-1)(N-1)}, \Z[q^{\pm1},s^{\pm 1}]) \ \ \ \ \ \ \ \ \ (W^{\xi,\lambda}_{2n-1, (n-1)(N-1)}, \C)$$ have the same dimension, and $\Theta^N_{\xi,\lambda}$ evaluated on the indexing set $E_{2n-1, (n-1)(N-1)}$ gives a basis of $W^{\xi,\lambda}_{2n-1, (n-1)(N-1)}$. Combined with relation \ref{eq:10},
we obtain that 
\begin{equation}\label{eq:11}
\Theta^{N,n}\Big(\mathscr B_{H^{\infty,N}_{2n-1, (n-1)(N-1)}}\Big)
\end{equation}
 gives a linearly independent set in the highest weight space $W^N_{2n-1, (n-1)(N-1)}$. 
\end{remark}

The main part of this section will be devoted to the understanding of the function $\Theta^{N,n}$ in relation to the ring of coefficients. More precisely, we will show that we can enlarge the ring of coefficients $\Z[q^{\pm1},s^{\pm}]$ such that the function $\Theta^{N,n}$ becomes invertible. 

We aim to look at this function with respect to the multifork basis and the set give in equation \ref{eq:11}. Now, we enlarge the set of coefficients so that the set 
$$\Theta^{N,n}\Big(\mathscr B_{H^{\infty,N}_{2n-1, (n-1)(N-1)}}\Big)
$$
becomes a basis for $W^N_{2n-1, (n-1)(N-1)}$.
Looking at this map in the multifork basis and the basis for the weight space defined in \ref{R:basisweight}, we notice that $$\Theta^{N,n} \in \mathscr M_{d^N_{n,m}}\big(\Z[q^{\pm 1}, s^{\pm 1}] \big).$$ 

Using the actions of the quantum group onto the module $U_N(\lambda)$, given by formulas from Proposition \ref{P:actions}, we notice that all coefficients that occur in the formulas for $\Theta^{N,n}$, which come from the action of $E$, belong to the set:
$$\mathscr L:=<< det(\Theta^{N,n}),\frac{sq^{-k}-s^{-1}q^{k}}{q-q^{-1}},s^k,q^k| k\in \N \}>>_{\Z[q^{\pm 1},s^{\pm 1}]}\subseteq \Z[q^{\pm 1},s^{\pm 1}].$$
(here,  $<< T >>_{\Z[q^{\pm 1},s^{\pm 1}]}$ means the multiplicative system generated by the set $T$).
\begin{definition}(Coefficients)
Consider the ring of coefficients where one inverts the mutiplicative system $\mathscr L$:  
$$\mathscr I:= \Z[q^{\pm 1}, s^{\pm 1}](\mathscr L)\subseteq \Q(q,s).$$
\end{definition}
\begin{notation}
Denote the inclusion of rings by: $$\iota: \Z[q^{\pm 1}, s^{\pm 1}]\hookrightarrow \mathscr I.$$
Also, we extend the specialisations from \ref{D:1} to this bigger ring:\\
{\bf I. Quantum side} $\tilde{\eta}_{\xi,\lambda}: \mathscr I \rightarrow \C$ defined as 
\begin{equation}
\begin{aligned}
\tilde{\eta}_{\xi,\lambda}(q)=\xi; \ \ \ \ \ 
\tilde{\eta}_{\xi,\lambda}(s)=\xi^{\lambda}.
\end{aligned}
\end{equation}
We notice that $\eta_{\xi,\lambda}=\iota \circ\tilde{\eta}_{\xi,\lambda}.$\\
{\bf II. Homological side} $\tilde\gamma: \Z[x^{\pm},d^{\pm}]\rightarrow \mathscr I$ defined as:
\begin{equation}
\tilde{\gamma}:=\iota \circ \gamma; 
\end{equation}

\end{notation}
\begin{remark}
An important observation is the fact that the specialisation $$\tilde{\eta}_{\xi,\lambda}: \mathscr I \rightarrow \C$$ is well defined for any $\lambda\in \C \setminus \N$. This happens due to the fact that we started with the specialisation $$\eta_{\xi,\lambda}: \Z[q^{\pm},s^{\pm}]\rightarrow \C$$ and we considered the multiplicative system $\mathscr L \subseteq \Z[q^{\pm},s^{\pm}]$ such that $$\mathscr L \cap Ker (\eta_{\xi,\lambda})=\emptyset.$$
This ensures that one can extend the map $\eta_{\xi,\lambda}$ to $\tilde{\eta}_{\xi,\lambda}$, which is defined on the ring of coefficients where we invert the multiplicative system $\mathscr L$.
\end{remark}
\definecolor{bittersweet}{rgb}{1.0, 0.44, 0.37}
\definecolor{brickred}{rgb}{0.8, 0.25, 0.33}
\begin{center}
\begin{tikzpicture}
[x=1.1mm,y=1.1mm]

\node (b1)  [color=blue] at (-30,30)    {$\boldsymbol{\mathbb Z[x^{\pm1},d^{\pm1}]}$};
\node (b2) [color=orange] at (0,30)   {$\boldsymbol{\mathbb Z[s^{\pm1},q^{\pm1}]}$};
\node (b3) [color=red]   at (30,30)   {$\boldsymbol{\mathscr I=\mathbb Z[s^{\pm},q^{\pm1}](\mathscr L)}$};
\node (b1')  [color=blue] at (-30,50)    {$\boldsymbol{H^{\infty,N}_{n,m}}$};
\node (b1'')  [color=blue] at (0,50)    {$\boldsymbol{H^{\infty,N}_{n,m}|_{\gamma}}$};
\node (b1''')  [color=blue] at (30,50)    {$\boldsymbol{H^{\infty,N}_{n,m}|_{\tilde{\gamma}}}$};

\node (b1'''')  [color=blue] at (-30,-5)    {$\boldsymbol{H^{\infty,N}_{n,m}|_{\psi_{\xi,\lambda}}}$};
\node (b3'') [color=orange]   at (0,-5)   {$\boldsymbol{W^{N}_{n,m}|_{\eta_{\xi,\lambda}}}$};

\node (b2') [color=orange] at (0,40)   {$\boldsymbol{W^{N}_{n,m}}$};
\node (b3') [color=red]   at (30,40)   {$\boldsymbol{W^{N}_{n,m}|_{\iota}}$};
\node (b4') [color=brickred]  at (48,30)   {$\hookrightarrow$};
\node (b4) [color=brickred]  at (55,30)   {$\Q(q,s)$};
\node (b5) [color=green]  at (0,2)   {$\boldsymbol{\mathbb C}$};
\node (b6) [color=blue]  at (-22,27)   {$\boldsymbol {x}$};
\node (b6') [color=blue]  at (-22,24)   {$\boldsymbol {d}$};
\node (b7) [color=orange]  at (-7,27)   {$\boldsymbol {s^2}$};
\node (b7') [color=orange]  at (-7,24)   {$\boldsymbol {-q^{-2}}$};
\node (b8) [color=orange]  at (-6,20)   {$\boldsymbol {q}$};
\node (b8') [color=orange]  at (-2,20)   {$\boldsymbol {s}$};
\node (b9) [color=green]  at (-6,10)   {$\boldsymbol {\xi}$};
\node (b9') [color=green]  at (-2,10)   {$\boldsymbol {\xi^{\lambda}}$};

\draw[<->,color=green, dashed]   (b1'''')      to node[right,font=\small, yshift=3mm]{}                           (b3'');
\draw[<->,color=red, dashed]   (b1''')      to node[right,font=\small, yshift=3mm]{}                           (b3');
\draw[<->,color=orange, dashed]   (b1'')      to node[right,font=\small, yshift=3mm]{}                           (b2');

\draw[->,color=blue, dashed]   (b6)      to node[right,font=\small, yshift=3mm]{}                           (b7);
\draw[->,color=blue,dashed]   (b6')      to node[right,font=\small, yshift=3mm]{}                           (b7');
\draw[->,color=blue]   (b1)      to node[right,font=\small, yshift=3mm]{$\boldsymbol{\gamma}$}                           (b2);

\draw[->,color=red]             (b2)      to node[left,font=\small, yshift=3 mm]{$\boldsymbol{\iota}$}   (b3);
\draw[->, color=cyan, thick]             (b1)     to[in=160,out=20] node[right,yshift=-3mm,xshift=-5mm,font=\small]{$\boldsymbol {\tilde{\gamma}}$}   (b3);
\draw[->,color=red, thick]             (b3)      to node[right,font=\small, xshift=1mm]{$\boldsymbol{\tilde{\eta}_{\xi,\lambda}}$}   (b5);
\draw[->,color=orange,thick]   (b2)      to node[right,font=\small]{$\boldsymbol{\eta_{\xi,\lambda}}$}                        (b5);
\draw[->,color=green]   (b1)      to node[left,font=\small]{$\boldsymbol{\psi_{\xi,\lambda}}$}                        (b5);
\draw[->,color=orange, dashed]   (b8)      to node[right,font=\small, yshift=3mm]{}                           (b9);
\draw[->,color=orange,dashed]   (b8')      to node[right,font=\small, yshift=3mm]{}                           (b9');


\end{tikzpicture}
\end{center}

\begin{remark}
Putting the previous results together, we conclude the following: 
\begin{enumerate}
\item The set $\Theta^{N,n}\Big(\mathscr B_{H^{\infty,N}_{2n-1, (n-1)(N-1)}}\Big)$ is a linearly independent system of the generic highest weight space $W_{2n-1, (n-1)(N-1)}^{N}$ over the ring $\Z[q^{\pm}, s^{\pm}]$.
\item The determinant of the function $\Theta^{N,n}$ becomes invertible over $\mathscr I$.
\end{enumerate}
\end{remark}
They lead to the following:
\begin{proposition}(Globalisation of isomorphism relating the quantum and homological representations)\label{P:2}\\ 
a) The set $\Theta^{N,n}\Big(\mathscr B_{\HH^{\infty,N}_{2n-1, (n-1)(N-1)}}\Big)$ gives a basis of $W_{2n-1, (n-1)(N-1)}^{N}$ over the extended ring $ \mathscr I$.\\
 b) The function $${\Theta^{N,n}: \HH^{\infty,N}_{2n-1, (n-1)(N-1)}|}_{\tilde{\gamma}}\rightarrow {W^N_{2n-1, (n-1)(N-1)}|}_{\iota} $$
 becomes an invertible morphism of modules over $\mathscr I$.\\
c) The globalisation property: $$\Theta^{N,n}|{\tilde{\eta}_{\xi,\lambda}}=\Theta^{N,n}_{\xi,\lambda}.$$
\end{proposition}

\subsection{Globalisation of the first homology class- $\mathscr F_{n}^{N}$}
\

We remind that the first homology class for the topological model belongs to the specialisation of the Lawrence representation:
$$\mathscr F^{N,\lambda}_n \in \HH^{\infty,N}_{2n-1, (n-1)(N-1)}|_{\psi_{\xi,\lambda}}.$$
This is given by the formula from \ref{eq:12}
$$\mathscr F_n^{N,\lambda}:={\left( \Theta^{N,n}_{\xi,\lambda}\right) } ^{-1}(v_0 \otimes w^{n-1}_{Coev}).$$
where the vector $w^{n-1}_{Coev}$ is computed in equation \ref{eq:8}:
\begin{equation}\label{eq:14}
w^{n-1}_{Coev}:=(Id^{\otimes n-1}_{U_N(\lambda)}\otimes \psi^N_n) \cdot \mathscr F_{\lambda} \left( \tikz[x=1mm,y=1mm,baseline=0.5ex]{\draw[<-] (0,3) .. controls (0,0) and (3,0) .. (3,3); \draw[<-] (-3,3) .. controls (-3,-3) and (6,-3) .. (6,3); \draw[draw=none, use as bounding box](-0.5,0) rectangle (3,3);}\right)(1) \in W^{\xi,\lambda}_{2(n-1),(n-1)(N-1)}.
\end{equation}
\begin{remark}(Globalisation of the partial coevaluation over 2 variables)

Using the formulas for the coevaluations, we notice that they can be written using the variables $q$ and $s$, where the parameter $s$ should be thought as $q^{\lambda}$. One can globalise the Reshetikhin-Turaev morphism over the ring in two variables $\mathscr I$. Also, we remember the formulas for the additional function
$$\psi^N_n=f_1\otimes...\otimes f_{n-1}.$$
Here, all the coefficients for the functions $f_1,...,f_n$ are given in \ref{eq:13}, and we notice that they can be written in the variables $q$ and $s$.
\end{remark}
We obtain the following lift of the coevaluation.
\begin{lemma}\label{L:1} 
There exists an element $$w^{n-1,N}_{Coev} \in {W^{N}_{2(n-1),(n-1)(N-1)}|}_{\iota}$$ such that
$${w^{n-1,N}_{Coev}|}_{\tilde{\eta}_{\xi,\lambda}}=w^{n-1}_{Coev}.$$
\end{lemma}
On the other hand, from the previous discussion, we have the isomorphism:
\begin{equation}\label{eq:15}
\big(\Theta^{N,n}|_{\iota}\big)^{-1}: {W^N_{2n-1, (n-1)(N-1)}|}_{\iota} \rightarrow {H^{\infty,N}_{2n-1, (n-1)(N-1)}|}_{\tilde{\gamma}}.
\end{equation}
\begin{proposition} \label{P:3} Putting together the formulas for the homology class $\mathscr F_{n}^{N,\lambda}$ from equation \ref{eq:14}, Lemma \ref{L:1} and the isomorphism \ref{eq:15}, we obtain that there exist a homology class 
$$\mathscr F_{n}^{N} \in \HH^{\infty,N}_{2n-1, (n-1)(N-1)}|_{\tilde{\gamma}}$$ over the ring in two variables $\mathscr I:= \Z[q^{\pm 1}, s^{\pm 1}](\mathscr L)\subseteq \Q(q,s)$ such that it globalises the homology classes parametrised by $\lambda$:
$${\mathscr F_{n}^{N}}|_{\tilde{\eta}_{\xi,\lambda}}=\mathscr F_{n}^{N,\lambda}.$$
\end{proposition}
\subsection{Globalisation of the second homology class- $\mathscr G_{n}^{N}$}
\

The construction of the second homology class $\mathscr G^{N,\lambda}_n$ from the step \ref{Step5}, was done using the evaluation map on the highest weight space:
$$Ev^{\lambda}_{N,n}: W^{\xi,\lambda}_{2n-1,(n-1)(N-1)}\rightarrow \C.$$

In order to define it, one uses the map below, which is described using the evaluations from the quantum group and the normalisation $\psi_n^N$, and then restrict it to the particular highest weight space $ W^{\xi,\lambda}_{2n-1,(n-1)(N-1)}$ (the one that is useful for the topological construction):
\begin{center}
\begin{tikzpicture}
[x=1mm,y=1mm]

\node (b1)  [color=black] at (-40,0)    {$\pi_{v_0} \circ \left(Id_{U^{\lambda}_N}\otimes \left(\ev^{n-1}_{\lambda} \circ...\circ \ev^{1}_{\lambda}\right)\right) \circ  {\color{blue} \left( Id^{\otimes n-1}_{(U_N(\lambda)} \otimes \left(\psi^N_n \right)^{-1} \right)}: U_N(\lambda)^{\otimes 2n-1}$};
\node (b2) [color=black] at (20,0)   {$\C$};
\node (b3) [color=black]   at (2,-15)   {$W^{\xi,\lambda}_{2n-1,(n-1)(N-1)}$};
\node (b4) [color=black]   at (-18,-15)   {$Ev^{\lambda}_{N,n}:$};
\node (b5) [color=red]   at (-18,-30)   {$\tilde{Ev}_{N,n}:$};
\node (b6) [color=red]   at (2,-30)   {$W^N_{2n-1,(n-1)(N-1)}|_{\iota}$};
\node (b7) [color=red]   at (20,-15)   {$\mathscr I$};

\draw[->,color=red]             (b6) to node[yshift=-3mm,font=\small]{} (b7);
\draw[->,color=green,dashed]             (b7) to node[yshift=-3mm,font=\small]{} (b2);
\draw[->,color=green,dashed]             (b6) to node[left,font=\small]{$\tilde{\eta}_{\xi,\lambda}$} (b3);

\draw[->]             (b1) to node[yshift=-3mm,font=\small]{} (b2);

\draw[->]             (-18,-3) to node[yshift=-3mm,font=\small]{} (b4);
\draw[->, dashed]             (b3) to [] node[yshift=-3mm,font=\small]{} (2,-3);
\draw[->]             (b3) to node[yshift=-3mm,font=\small]{} (b2);

\end{tikzpicture}
\end{center}
We notice that the formulas for the evaluations, come using the action of the generator $K$ of the quantum group. For the sequel, we notice that we can define this action onto $U_N^{\lambda}$ by the formulas:
$$K v_i=sq^{-2i}.$$
In other words, the part coming from the evaluations \ref{E:DualityForCat}, can be defined over the ring $\mathscr I:= \Z[q^{\pm 1}, s^{\pm 1}](\mathscr L)$. Also, as we have already discussed, the function $\psi_n^N$ can be globalised over this ring. This leads to the fact that one can extend the evaluation $Ev^{\lambda}_{N,n}$ towards an evaluation 
$$\tilde{Ev}_{N,n}: W^{N}_{2n-1,(n-1)(N-1)}|_{\iota}\rightarrow \mathscr I$$ such that
\begin{equation}\label{eq:16}
\tilde{Ev}_{N,n}|_{\tilde{\eta}_{\xi,\lambda}}=Ev^{\lambda}_{N,n}.
\end{equation}
On the other hand, we remind that the specialised homology class $\mathscr G_n^{N,\lambda}$ was prescribed by the formula given in the definition \ref{G}:
\begin{equation}\label{eq:17}
Ev^{\lambda}_{N,n} \circ { \Theta^{N,n}_{\xi,\lambda}}  (\cdot )=\ll \cdot, \mathscr G^{N,\lambda}_{n} \gg|_{\psi_{\xi,\lambda}} \ \ \in \  \left( \HH^{\infty,N}_{2n-1,(n-1)(N-1)} |_{\psi_{\xi,\lambda}} \right) ^*.
\end{equation}
We remark that the specialisations of coefficients satisfy the following relation:
\begin{equation}
\psi_{\xi,\lambda}=\tilde{\eta}_{\xi,\lambda} \circ \tilde{\gamma}.
\end{equation}
Regarding the topological pairing, using this remark, we notice that it commutes with the specialisation maps:
\begin{lemma}(The intersection pairing commutes with the specialisation map)\label{L:2} 
For any homology classes $$\mathscr F \in {\HH^{\infty,N}_{\big(2n-1,(n-1)(N-1)\big)}|}_{\tilde{\gamma}} \ \ \ \text{ and } \ \ \ \mathscr G \in {\HH^{\Delta,\partial,N}_{\big(2n-1,(n-1)(N-1)\big)}|}_{\tilde{\gamma}}$$
one has the following commuting property:
$$\tilde{\eta}_{\xi,\lambda} \big( \ll \mathscr F, \mathscr G \gg|_{\tilde{\gamma}} \big)= \ll \mathscr F|_{ \tilde{\eta}_{\xi,\lambda} }, \mathscr G|_{ \tilde{\eta}_{\xi,\lambda} } \gg|_{ \psi_{\xi,\lambda} }.$$
\end{lemma}
We summarise this in the following diagram.
\definecolor{cadmiumgreen}{rgb}{0.0, 0.42, 0.24}
\begin{center}
\begin{tikzpicture}
[x=1mm,y=1mm]

\node (b1)  [color=orange] at (-32,0)    {$W^{\xi,\lambda}_{\big(2n-1,(n-1)(N-1)\big)}$};
\node (b2) [color=blue] at (11,0)   {${\HH^{\infty,N}_{\big(2n-1,(n-1)(N-1)\big)}|}_{\psi_{\xi,\lambda}}$};
\node (b3) [color=teal]   at (50,0)   {${\HH^{\Delta,\partial,N}_{\big(2n-1,(n-1)(N-1)\big)}|}_{\psi_{\xi,\lambda}}$};
\node (b4)  [color=green] at (0,26)    {${\boldsymbol{\mathbb C}}$};
\node (b22)  [color=blue] at (29,0)    {$\otimes$};
\node (b5)  [color=cyan] at (47,8)    {$\boldsymbol{\ll \cdot, {\color{teal}\mathscr G_{n}^{N,\lambda}}{\color{cyan} \gg_{\psi_{\xi,\lambda}}}}$};

\draw[->]             (b2) to node[yshift=-3mm,font=\small]{$\Theta_{\xi,\lambda}^N$} (b1);
\draw[->,color=green]   (b1)      to node[left,font=\small, xshift=-5mm]{$Ev_{n,N}^{\lambda}$}                           (b4);
\draw[->,color=cyan, dashed]   (28,10)      to node[right,font=\small, yshift=3mm]{$\ll ,\gg_{\psi_{\xi,\lambda}}$}                           (b4);
\draw[->,color=cadmiumgreen, thick]   (10,7)      to node[left,font=\small, yshift=-8mm,xshift=2mm]{$\boldsymbol{Ev_{n,N}^{\lambda}\circ \Theta_{\xi,\lambda}^N}$}                           (b4);
\draw[<->, color=cyan, thick]             (10,6)     to[in=160,out=20] node[right,yshift=-3mm,xshift=-5mm,font=\small]{$$}   (35,6);
\end{tikzpicture}
\end{center}
\begin{center}
\begin{tikzpicture}
[x=1mm,y=1mm]

\node (b1)  [color=red] at (-32,0)    {${W^{N}_{\big(2n-1,(n-1)(N-1)\big)}|_{\iota}}$};
\node (b2) [color=blue] at (11,0)   {${\HH^{\infty,N}_{\big(2n-1,(n-1)(N-1)\big)}|}_{\tilde{\gamma}}$};
\node (b3) [color=teal]   at (50,0)   {${\HH^{\Delta,\partial,N}_{\big(2n-1,(n-1)(N-1)\big)}|}_{\tilde{\gamma}}$};
\node (b4)  [color=red] at (0,26)    {$\mathscr I=\mathbb Z[s^{\pm},q^{\pm1}](\mathscr L)$};
\node (b22)  [color=blue] at (30,0)    {$\otimes$};
\node (b5)  [color=cyan] at (47,8)    {$\boldsymbol{{\ll \cdot, {\color{teal}\mathscr G_{n}^{N}}{\color{cyan} \gg}\color{cyan}|}_{\color{cyan}\tilde{\gamma}}}$};

\draw[->,color=green,very thick,dashed]             (-35,10) to node[left,yshift=-1mm,font=\small]{$\boldsymbol{\tilde{\psi}_{\xi,\lambda}} $} (-35,28);
\draw[->]             (b2) to node[yshift=-3mm,font=\small]{$\Theta^{N,n}$} (b1);
\draw[->,color=red]   (b1)      to node[left,font=\small, xshift=-5mm]{$\tilde{Ev}_{n,N}$}                           (b4);
\draw[->,color=cyan, dashed]   (25,10)      to node[right,font=\small, yshift=3mm]{$\ll ,\gg_{\tilde{\gamma}}$}                           (b4);
\draw[->,color=cadmiumgreen, thick]   (10,7)      to node[left,font=\small, yshift=-8mm,xshift=2mm]{$\boldsymbol{\tilde{Ev}_{n,N}\circ \Theta^{N,n}}$}                           (b4);
\draw[<->, color=cyan, thick]             (10,6)     to[in=160,out=20] node[right,yshift=-3mm,xshift=-5mm,font=\small]{$ $}   (35,6);
\end{tikzpicture}
\end{center}
Since the topological pairing $\ll, \gg$ described in proposition \ref{P:1} has the identity matrix in the bases corresponding to multiforks and scans, it follows that any specialisation of it will be non-degenerate. We obtain the following:
\begin{remark}
The pairing $\ll, \gg|_{\tilde{\gamma}}$ is non-degenerate over the ring $\mathscr I$.
\end{remark}  
\begin{corollary}(Globalisation of the second homology class $\mathscr G_{n}^{N}$)\\
There exists a unique homology class $\mathscr G^{N}_{n} \in \HH^{\Delta,\partial,N}_{2n-1,(n-1)(N-1)} |_{\tilde{\gamma}}$ such that
\begin{equation}\label{eq:18}
\tilde{Ev}_{N,n} \circ { \Theta^{N,n}}  (\cdot )=\ll \cdot, \mathscr G^{N}_{n} \gg|_{\tilde{\gamma}}.
\end{equation}
\end{corollary}
\begin{proposition}\label{P:4}  
Combining the globalisation properties from above, we obtain that: 
$${\mathscr G_{n}^{N}}|_{\tilde{\eta}_{\xi,\lambda}}=\mathscr G_{n}^{N,\lambda}.$$ 
\end{proposition}
\begin{proof}
We will use the results concerning the evaluation functions given in equation \ref{eq:16} and   \ref{eq:17}, the definiton of the globalised class $\mathscr G_{n}^{N}$ from \ref{eq:18} and the commuting property of the pairing \ref{L:2}. 
Using these properties, we have:
$$\ll \cdot \ , \mathscr G^{N,\lambda}_n  \gg|_{ \psi_{\xi,\lambda} }=^{Eq\ \ref{eq:17}} \ Ev_{N,n}^{\lambda}\circ \Theta_{\xi,\lambda}^{N,n}(\cdot)= \ \ \ \ \ \ \ \ \ $$
$$=^{Eq\ \ref{eq:16} \ Prop \ref{P:2}} \ {\tilde{Ev}_{N,n}|}_{\tilde{\eta}_{\xi,\lambda}}\circ {\Theta^{N,n}|}_{\tilde{\eta}_{\xi,\lambda}}(\cdot)= \ \ \ \ \ \ \ \ \ $$
$$=\tilde{\eta}_{\xi,\lambda}\big(\tilde{Ev}_{N,n}\circ \Theta^{N,n}(\cdot)\big)=^{Eq\ \ref{eq:18}} \ \tilde{\eta}_{\xi,\lambda} \big( \ll \cdot \ , \mathscr G^{N}_n  \gg|_{ \tilde{\gamma} } \big) $$
$$ \ \ \ \ \ \ \ \ \ \ =^{Lemma \ \ref{L:2}} \ \  \ll \cdot \ , {\mathscr G^{N}_n|}_{\tilde{\eta}_{\xi,\lambda}}  \gg|_{\psi_{\xi,\lambda}} ,  \ \ \forall \cdot \in H^{\infty,N}_{2n-1,(n-1)(N-1)} |_{\psi_{\xi,\lambda}} .$$
Using the form of the pairing given in the corollary \ref{C:pairing}, one has that $\ll ,\gg|_{ \psi_{\xi,\lambda} }$ is non-degenerate. This shows that the two homology classes coincide and concludes the proof.
\end{proof}
\subsection{Braid group actions}
So far, we have constructed the globalisation of the homology classes that lead to the coloured Alexander invariant, over an extension of the ring of Laurent polynomials in two variables. This part concerns the braid part (more precisely encoded into the braid group action onto the special Lawrence representation) and the globalisation of the topological pairing. We summarise these two parts in the following diagram. The commutativity of the pairing with respect to the specialisation map was discussed in \ref{L:2} and refers to the blue part of the diagram.  

\begin{center}
\begin{tikzpicture}
[x=1mm,y=1mm]

\node (t1) [color=blue]   at (0,-8)   {$\mathscr F_{n}^N$};
\node (t2) [color=teal]   at (40,-8)   {$\mathscr G_{n}^N$};
\node (t1') [color=blue]   at (0,-22)   {$\mathscr F_{n,\lambda}^N$};
\node (t2') [color=teal]   at (40,-22)   {$\mathscr G_{n,\lambda}^N$};

\node (c0) [color=brickred]   at (-10,-12)   {$\boldsymbol{\equiv}$\ \ \ \ref{L:3}};
\node (c0') [color=cyan]   at (60,-12)   {$\boldsymbol{\equiv}$ \ \ \ \ref{L:2}};

\node (b0) [color=brickred]   at (-20,7)   {$B_{2n-1}\curvearrowright$};
\node (b0') [color=brickred]   at (-20,-23)   {$B_{2n-1}\curvearrowright$};

\node (b1) [color=cyan]   at (-30,0)   {$\ll ,\gg_{\tilde{\gamma}}:$};
\node (b2)  [color=black] at (0,0)    {${\HH^{\infty,N}_{\big(2n-1,(n-1)(N-1)\big)}|}_{\tilde{\gamma}}$};
\node (b3) [color=black] at (40,0)   {${\HH^{\Delta,\partial,N}_{\big(2n-1,(n-1)(N-1)\big)}|}_{\tilde{\gamma}}$};
\node (b3') [color=black] at (20,0)   {$\bigotimes$};

\node (b4) [color=red]   at (80,0)   {$\mathscr I=\mathbb Z[s^{\pm},q^{\pm1}](\mathscr L)$};
\node (b1') [color=cyan]   at (-30,-30)   {$\ll ,\gg_{\psi_{\xi,\lambda}}:$};
\node (b2')  [color=black] at (0,-30)    {${\HH^{\infty,N}_{\big(2n-1,(n-1)(N-1)\big)}|}_{\psi_{\xi,\lambda}}$};
\node (b3') [color=black] at (40,-30)   {${\HH^{\Delta,\partial,N}_{\big(2n-1,(n-1)(N-1)\big)}|}_{\psi_{\xi,\lambda}}$};
\node (b3'') [color=black] at (20,-30)   {$\bigotimes$};

\node (b4') [color=green]   at (80,-30)   {$\C$};

\draw[->,color=green,dashed]             (b4) to node[yshift=-3mm,font=\small]{} (b4');
\draw[->,color=green,dashed]             (b1) to node[left,font=\small]{$\tilde{\eta}_{\xi,\lambda}$} (b1');
\draw[->,color=black]             (b3) to node[yshift=-3mm,font=\small]{} (b4);
\draw[->,color=black]             (b3') to node[yshift=-3mm,font=\small]{} (b4');
\draw[->,color=green,dashed]             (t1) to node[yshift=-3mm,font=\small]{} (t1');
\draw[->,color=green,dashed]             (t2) to node[yshift=-3mm,font=\small]{} (t2');
\draw[->,color=green,dashed]             (b0) to node[yshift=-3mm,font=\small]{} (b0');



\end{tikzpicture}
\end{center}
The following lemma tells us that the braid group action commutes with the specialisation of the coefficients. This refers to the pink part of the diagram and follows directly from the definition of the specialisation of a module, as described in the notation \ref{N:spec}.

\begin{lemma}(Braid group action onto specialised Lawrence representation commutes with specialisation map)\label{L:3}\\
For any homology class $$\mathscr F \in {\HH^{\infty,N}_{\big(2n-1,(n-1)(N-1)\big)}|}_{\tilde{\gamma}}$$
and any braid $\beta \in B_{2n-1}$, one has the following commuting properties:
$$\big( \beta \curvearrowright \mathscr F  \big)|_{\tilde{\eta}_{\xi,\lambda}}= \beta \curvearrowright \big( \mathscr F |_{ \tilde{\eta}_{\xi,\lambda} }\big).$$
\end{lemma}
\begin{notation}\label{N:3}
We remind that the pairing $$\ll , \gg |_{\tilde{\gamma}}: \HH^{\infty, N}_{n,m}|_{\tilde{\gamma}} \otimes \HH^{\Delta,\partial,N}_{n,m} |_{\tilde{\gamma}}\rightarrow \mathscr I$$
$$\tilde{\eta}_{\xi,\lambda}:\mathscr I \rightarrow \C.$$
Then, we denote the evaluation of this specialisation applied to the pairing in two variables $\ll , \gg |_{\tilde{\gamma}}$ by:
$$\ll \cdot , \cdot \gg _{\tilde{\eta}_{\xi,\lambda}}=\tilde{\eta}_{\xi,\lambda}(\ll \cdot, \cdot \gg |_{\tilde{\gamma}}).$$
\end{notation}

Collecting together the results from this sections, one obtains that the coloured Alexander invariant can be obtained by the generic pairing between the globalised homology classes, which is then evaluated through the specialisation $\tilde{\eta}_{\xi,\lambda}$.

 More precisely, we obtain the topological model from Theorem \ref{THEOREM}:
\begin{theorem*}(Topological model for coloured Alexander invariant with globalised classes)

$$\Phi_N(L,\lambda)= \xi^{(N-1)\lambda w(\beta_n)} \ll (\beta_n \cup I^{n-1}) \mathscr{F}^N_n , \mathscr{G}^N_n\gg_{\tilde{\eta}_{\xi,\lambda}}, \ \ \ \forall \lambda \in \C \setminus \Z.$$ 
\end{theorem*}
\begin{proof}$$ \xi^{(N-1)\lambda w(\beta_n)} \ll (\beta_n \cup \mathbb I_{n-1})) \mathscr{F}^N_n , \mathscr{G}^N_n\gg_{\tilde{\eta}_{\xi,\lambda}}=$$
$$=^{Not \ \ref{N:3}} \xi^{(N-1)\lambda w(\beta_n)}  {\tilde{\eta}_{\xi,\lambda}} \big({\ll (\beta_n \cup \mathbb I_{n-1})) \mathscr{F}^N_n , \mathscr{G}^N_n\gg|}_{\tilde{\gamma}}\big)=$$
$$=^{Lemma \ \ref{L:2}} \xi^{(N-1)\lambda w(\beta_n)} {\ll {\big( (\beta_n \cup \mathbb I_{n-1})) \mathscr{F}^N_n \big)|}_{{\tilde{\eta}_{\xi,\lambda}}} , {\mathscr{G}^N_n|}_{{\tilde{\eta}_{\xi,\lambda}}}\gg|}_{\psi_{\xi,\lambda}}=$$
$$=^{Lemma \ \ref{L:3}} \xi^{(N-1)\lambda w(\beta_n)} {\ll { (\beta_n \cup \mathbb I_{n-1})) \big( \mathscr{F}^N_n |}_{{\tilde{\eta}_{\xi,\lambda}}} \big), {\mathscr{G}^N_n|}_{{\tilde{\eta}_{\xi,\lambda}}}\gg|}_{\psi_{\xi,\lambda}}=$$
$$=^{ Prop \ \ref{P:3} \ \ \ref{P:4}} \xi^{(N-1)\lambda w(\beta_n)} \ll { (\beta_n \cup \mathbb I_{n-1}))  \mathscr{F}^{N,\lambda}_n, {\mathscr{G}^{N,\lambda}_n}\gg|}_{\psi_{\xi,\lambda}}=$$
$$=^{Thm \ \ref{C1:pr2}} \Phi_N(L,\lambda).$$

\end{proof}
}

\
\
\url{https://www.maths.ox.ac.uk/people/cristina.palmer-anghel}


\begin{thebibliography}{99}
\bibitem{Cr} C. Anghel- {\em A Homological model for the coloured Jones polynomials}, \href{https://arxiv.org/abs/1712.04873}{arXiv:1712.04873},(2017)
\bibitem{CM} C. Anghel, M. Palmer-{ \em Poincar\'e duality, pairings and bases for different flavours of the Lawrence representations of braid groups}
\bibitem {ADO} Y. Akustu, T. Deguchi, T. Ohtsuki- {\em Invariants of colored links}, J. Knot Theory Ramifications 1 (1992) 161–184.
\bibitem{Big}  S. Bigelow - {\em A homological definition of the Jones polynomial},  Invariants of knots and 3-manifolds (Kyoto, 2001), volume 4 of Geom. Topol. Monogr., pages 29-41. Geom. Topol. Publ., Coventry, 2002.
\bibitem{Big2} S. Bigelow-{\em Homological representations of the Iwahori-Hecke algebra}, Geometry and Topology Monographs, Volume 7: Proceedings of the Casson Fest, Pages 493-507


\bibitem{Ito}  T. Ito - {\em Reading the dual Garside length of braids from homological and quantum representations.} 
Comm. Math. Phys., 335(1):345-367, 2015.
\bibitem{Ito2} T. Ito -{\em A homological representation formula of colored Alexander invariants} Adv. Math. 289, 142-160, 2016.  
\bibitem{Ito3} Tetsuya Ito-{Topological formula of the loop expansion of the colored Jones polynomials}, Trans. Amer. Math. Soc., 2019 

\bibitem{JK} C. Jackson, T. Kerler- {\em The Lawrence–Krammer–Bigelow representations of the braid groups via $U_q(sl_2)$}, Adv. Math. 228, 1689-1717, 2011.
\bibitem{Koh}  T. Kohno -{\em Homological representations of braid groups and KZ connections.} J. Singul., 594-108, 2012.
\bibitem{Koh2} T. Kohno -{\em Quantum and homological representations of braid groups.} Configuration Spaces - Geometry, Combinatorics and Topology, Edizioni della Normale (2012), 355-372. 
\bibitem{Law}  R. J. Lawrence - {\em A functorial approach to the one-variable Jones polynomial.} J. Differential Geom., 37(3):689-710, 1993.
\bibitem{Law1} R. J. Lawrence -{ \em Homological representations of the Hecke algebra}, Comm. Math. Phys. 135 (1990) 141–191.
\bibitem{Mu} J. Murakami- { \em Colored Alexander invariants and cone-manifolds}, Osaka J. Math. 45 (2008) 541–564.
\end{thebibliography}
\end{document}